\pdfoutput=1
\RequirePackage{ifpdf}
\ifpdf 
\documentclass[pdftex]{sigma}
\else
\documentclass{sigma}
\fi

\numberwithin{equation}{section}

\newtheorem{Theorem}{Theorem}[section]
\newtheorem{Corollary}[Theorem]{Corollary}
\newtheorem{Lemma}[Theorem]{Lemma}

 { \theoremstyle{definition}
\newtheorem{Definition}[Theorem]{Definition}

\newtheorem{Remark}[Theorem]{Remark} }

\usepackage{mathabx}

\makeatletter
\renewcommand{\@biblabel}[1]{#1.}
\makeatother

\makeatletter
\newcommand{\vast}{\bBigg@{3}} 
\newcommand{\Vast}{\bBigg@{4}}
\makeatother

\newenvironment{Remark*}{\begin{remark*}\normalfont}{\end{remark*}}

\newtheorem*{remarks*}{Remarks}
\newenvironment{Remarks*}{\begin{remarks*}\normalfont}{\end{remarks*}}

\newtheorem*{rems}{Remarks} 

\newtheorem*{sol}{Solution}

\newtheorem*{nt}{Notes}

\makeatletter
\renewcommand{\@biblabel}[1]{#1.}
\makeatother

\newcommand{\suml}{\sum\limits}

\newcommand{\sumN}{{\left| \boldsymbol{N} \right|}}

\newcommand{\sumvec}[2]{#1_1+\cdots + #1_#2} 

\newcommand\sumK{{\left| \boldsymbol{K} \right|}}

\newcommand\sumj{{\left| \boldsymbol {j} \right|}}

\newcommand\sumk{{\left| \boldsymbol{k} \right|}}

\newcommand{\B}{{ \mathbf B}}
\newcommand{\M}{{ \mathbf M}}

\newcommand{\N}{{ \boldsymbol N}}
\newcommand{\K}{{ \boldsymbol K}}
\renewcommand{\k}{{ \boldsymbol{k}}}
\renewcommand{\j}{{ \boldsymbol{j}}}

\newcommand{\kvec}{{k_1, \dots,k_n}}

\newcommand{\sumbinomial}[2]{\suml_{r=1}^{#2} \binom{#1_r}2}

\newcommand{\multsum}[3]{{\sum\limits_{\substack{{0\le #1_#3 \le #2_#3} \\
{#3 =1,2,\dots, n}}}}}

\newcommand{\qrfac}[2]{{\left({#1}; q\right)_{#2}}} 

\newcommand{\pqrfac}[3]{{\left({#1};#3\right)_{#2}}}

\newcommand{\triprod}[1]{\prod\limits_{1\le r < s \le #1}}

\newcommand{\sqprod}[1]{\prod\limits_{r, s =1}^{#1}} 

\newcommand{\smallprod}[1]{\prod\limits_{r =1}^{#1}} 

\newcommand{\xover}[1]{#1_{r}/#1_{s}}

\newcommand{\vandermonde}[3]{\triprod{#3} 
 \frac{1-q^{#2_r-#2_s} \xover {#1} }{1-\xover{#1}}
}

\newcommand{\ellipticqrfac}[2]{{\left({#1}; q, p\right)_{#2}}} 
\newcommand{\elliptictheta}[1]{\theta \left({#1} ; p\right) }
\newcommand{\ellipticvandermonde}[3]{\triprod{#3} 
\!\! \frac{\elliptictheta{q^{#2_r-#2_s} \xover {#1} }}{\elliptictheta{\xover{#1}}}
}

\begin{document}

\allowdisplaybreaks

\newcommand{\arXivNumber}{1704.00020}

\renewcommand{\thefootnote}{}

\renewcommand{\PaperNumber}{025}

\FirstPageHeading

\ShortArticleName{Elliptic Well-Poised Bailey Transforms and Lemmas on Root Systems}

\ArticleName{Elliptic Well-Poised Bailey Transforms\\ and Lemmas on Root Systems\footnote{This paper is a~contribution to the Special Issue on Elliptic Hypergeometric Functions and Their Applications. The full collection is available at \href{https://www.emis.de/journals/SIGMA/EHF2017.html}{https://www.emis.de/journals/SIGMA/EHF2017.html}}}

\Author{Gaurav BHATNAGAR and Michael J.~SCHLOSSER}

\AuthorNameForHeading{G.~Bhatnagar and M.J.~Schlosser}

\Address{Fakult\"at f\"ur Mathematik, Universit\"at Wien, Oskar-Morgenstern-Platz 1, 1090 Wien, Austria}
\Email{\href{mailto:bhatnagarg@gmail.com}{bhatnagarg@gmail.com}, \href{mailto:michael.schlosser@univie.ac.at}{michael.schlosser@univie.ac.at}}
\URLaddress{\url{http://www.gbhatnagar.com}, \url{http://www.mat.univie.ac.at/~schlosse/}}

\ArticleDates{Received September 01, 2017, in final form March 13, 2018; Published online March 22, 2018}

\Abstract{We list $A_n$, $C_n$ and $D_n$ extensions of the elliptic WP Bailey transform and lemma, given for $n=1$ by Andrews and Spiridonov. Our work requires multiple series extensions of Frenkel and Turaev's terminating, balanced and very-well-poised ${}_{10}V_9$ elliptic hypergeometric summation formula due to Rosengren, and Rosengren and Schlosser. In our study, we discover two new $A_n$ ${}_{12}V_{11}$ transformation formulas, that reduce to two new $A_n$ extensions of Bailey's $_{10}\phi_9$ transformation formulas when the nome $p$ is $0$, and two multiple series extensions of Frenkel and Turaev's sum. }


\Classification{33D67}

\renewcommand{\thefootnote}{\arabic{footnote}}
\setcounter{footnote}{0}

\section{Introduction}
The many different proofs of the famous Rogers--Ramanujan identities have led to a plethora of fruitful ideas in mathematics and physics. This paper contains some results ultimately following a path that began with Watson's 1929 proof of these identities. Watson~\cite{GNW1929} proved a very general transformation formula with many parameters. The Rogers--Ramanujan identities follow by taking the limit as (most of) these parameters go to infinity, and then invoking the Jacobi triple product identity. The proof of Watson's transformation was later simplified by Bailey during the course of his study of Rogers' work; and the ensuing ideas used by Slater to prove more than a hundred Rogers--Ramanujan type identities. Eventually, Bailey's approach was perfected by Andrews as a combination of three ideas. According to Andrews' formulation, the Bailey transform is a specific (invertible) lower-triangular matrix that transforms a sequence to another sequence. A Bailey pair is a pair of sequences which satisfies such a relationship, and the Bailey lemma is a method to generate a new Bailey pair (with additional parameters) from a given pair. Thus, the Bailey lemma can be used to generate new identities from known results. In particular, Watson's transformation follows by two steps of the Bailey lemma, applied to the unit Bailey pair.

In two important papers, Milne~\cite{Milne1997} and Milne and Lilly~\cite{ML1995} lifted the Bailey transform and lemma machinery to the context of multiple basic hypergeometric series associated with the root systems $A_n$ and $C_n$.

The primary purpose of this paper is to extend Milne's and Milne and Lilly's work to the setting of the well-poised (WP) Bailey transform and lemma. In the $n=1$ case this is again based on Bailey's~\cite{WNB1947} ideas; and again, these ideas have been made accessible by Andrews'~\cite{Andrews2001} exposition. In particular, the WP Bailey transform and lemma captures the generalizations of the matrix formulation of the Bailey transform due to Bressoud~\cite{DB1983} and Bailey's~\cite{WNB1929} famous $_{10}\phi_9$ transformation which further generalizes Watson's formula.

At the same time, we work in the setting of elliptic hypergeometric series associated with the root systems $A_n$, $C_n$ and $D_n$. Elliptic hypergeometric series appeared explicitly in the work of Frenkel and Turaev~\cite{FT1997} in 1997, and it was quickly realized that some of the classical methods used for studying basic hypergeometric series apply as well to this kind of series. In this context, we find the work of Warnaar~\cite{SOW2002} very useful and influential, because it introduced a notation for elliptic hypergeometric series very much like the one used for basic hypergeometric series. The elliptic series contain a parameter $p$, called the nome. When the nome $p=0$, the formulas reduce to formulas for basic hypergeometric series. A key ingredient in Warnaar's paper~\cite{SOW2002} is an elliptic matrix inverse which as a special case contains an elliptic extension of the WP Bailey transform. Spiridonov~\cite{VPS2002} found an elliptic analogue of the WP Bailey lemma, and Warnaar~\cite{SOW2003} applied and further extended Andrews' ideas in the elliptic setting. A~comprehensive survey of elliptic hypergeometric functions has been given by Spiridonov~\cite{VPS2008}. Many of the central summation and transformation formulas concerning multiple elliptic hypergeometric series were given by Warnaar~\cite{SOW2002} and Rosengren~\cite{HR2004}.

In this paper, we provide several extensions of the elliptic WP Bailey transform and lemma. A feature of the theory of series associated with root systems is that often there are many extensions on root systems of the same result. Indeed, in this paper we present six extensions of the elliptic WP Bailey transform (and associated Bressoud matrices), and eight elliptic WP Bailey lemmas which can be used in multiple ways to generate different identities. For $n=1$ all our six elliptic WP Bailey transforms specialize to the same result by Warnaar, and, also, all our eight elliptic WP Bailey lemmas specialize to the same result by Spiridonov.

When $n=1$ and $p=0$, the WP Bailey transform and lemma depend in an essential way on a~summation result of Jackson~\cite{Jackson1921}, which is contained in Watson's transformation. Multivariate extensions of Jackson's sum on the root systems~$A_n$, $C_n$ and $D_n$ have been given by Milne~\cite{Milne1988}, Milne and Lilly~\cite{ML1995}, Denis and Gustafson~\cite{DG1992}, the first author~\cite{GB1999a}, and the second author~\cite{MS1997,MS2008}. As the second author showed in~\cite{MS2007b}, Milne's~\cite{Milne1988}~$A_n$ Jackson sum can be used to form a~bridge between~$A_n$ basic hypergeometric series and Macdonald polynomials~\cite[Chapter~VI]{Mac1995}.

Our work in this paper depends on the elliptic $A_n$, $C_n$ and $D_n$ generalizations of this summation due to Rosengren~\cite{HR2004} and one such result due to Rosengren and the second author~\cite{RS2017}. (Recently, Rosengren~\cite{HR2017a} gave yet another such result, which we do not include in our study.) One of the goals of our study is to recover the extensions of Bailey's transformation formula listed by Rosengren~\cite{HR2004}, which were obtained by a straightforward extension of the approach followed in \cite{BS1998} for the basic hypergeometric case. Indeed, we recover all these results. In addition, we give two new elliptic generalizations of Bailey's transformation formula.

Our results are closely related with other work in this area. One of our elliptic generalizations of Bailey's $_{10}\phi_9$ sum was found independently by Rosengren~\cite{HR-PC-2016}. It is motivated by formulas previously given by the authors~\cite{GB1995, MS2008}. The $p=0$ case of one of our matrices was considered by Milne~\cite{Milne1997} and a related Bailey lemma (again when $p=0$) was previously obtained by Zhang and Liu~\cite{ZL2016}. One of our Bailey lemmas is equivalent to a result of Zhang and Huang~\cite{ZH-preprint}. Most of the matrix inversions that appear in our work can be obtained as special cases of very general matrix inversions due to Rosengren and the second author~\cite{RS2017}.
Warnaar~\cite{SOW-notes-2016} found four of the WP Bailey lemmas (including the one due to Zhang and Huang~\cite{ZH-preprint}) a few years ago but did not publish his work.

Like Milne's~\cite{Milne1997} important paper, we hope that our work too will open the door for further work in this area. Further development of Bailey's ideas in, for example, \cite{Andrews2001,AndBer2002, Jouhet2010, MSZ2009, VPS2002, VPS2004, SSY2017, SOW2001, SOW2003} certainly suggests that there is a lot that can be done. Notably, Andrews, Schilling and Warnaar~\cite{ASW1999} established $A_2$ Bailey lemmas and
corresponding identities of Rogers--Ramanujan type. However, the type of series these authors deal with in \cite{ASW1999} are quite different from the ($n=2$ instances of the) $A_n$ series considered in this paper. Spiridonov and Warnaar~\cite{SW2006}, building upon work by Spiridonov~\cite{VPS2004}, obtained integral Bailey transforms relating integrals on root systems of type $A_n$ and $C_n$. Br\"unner and Spiridonov~\cite{BS2016} recently gave applications of the integral analogue of the $A_n$ Bailey lemma to supersymmetric linear quiver theories. Coskun~\cite{HC2008} had previously found a $BC_n$ WP Bailey transform for a slightly different kind of series than the one we consider and found multiple Rogers--Ramanujan identities. In this context, we also mention the work of Griffin, Ono and Warnaar~\cite{GOW2016}.

This paper is organized as follows. In Section~\ref{sec:notation}, we record the notation and terminology of elliptic and basic hypergeometric series. In Section~\ref{sec:WPBailey} we give a short introduction to the WP Bailey transform and lemma, in a format that suits our work.
From Section~\ref{sec:dougall1} until Section~\ref{sec:cn-case} (except Section~\ref{sec:piszero}) we systematically extend the analysis of
Section~\ref{sec:WPBailey} by considering in each section a particular multivariable extension of the elliptic Jackson summation. In Section~\ref{sec:piszero}, we record some basic hypergeometric results which follow as special cases of our work. In Section~\ref{sec:summary} we summarize our work and motivate our final set of results in Section~\ref{sec:new-dougall}, which contain an extra parameter. The extra parameter disappears when $n=1$.

\section{The notation and terminology}\label{sec:notation}
In this section, we record the notation and terminology used in this paper, following, for the most part, Gasper and Rahman \cite{GR90}. We recommend Rosengren's lectures~\cite{HR2016-lectures} for a friendly introduction to elliptic hypergeometric series.

The series considered in this paper are all of the form
\begin{gather*}
\sum\limits_{\substack{{k_r\geq 0 } \\ {r =1,2,\dots, n}}} S_\k,
\end{gather*}
where the $\k=(\kvec )$ is an $n$-tuple of non-negative integers $k_1, k_2, \dots, k_n$. The positive integer $n$ is called the {\em dimension} of the sum. We use the notation $\sumk:= \sumvec{k}{n}$ for the sum of components of the $n$-tuple; $\k+\j$ for $(k_1+j_1, \dots, k_n+j_n)$, obtained by component-wise addition; and $\k-\j$ for $(k_1-j_1, \dots, k_n-j_n)$.

The summand $S_\k$ itself is a product of various $q, p$-shifted factorials, which we define shortly.

First, for arbitrary integers $k$ and $m$, we define products as follows.
\begin{gather}\label{proddef}
\prod_{j=k}^m A_j := \begin{cases}
A_kA_{k+1}\cdots A_m & \text{if } m\geq k,\\
1 & \text{if } m= k-1,\\
\big(A_{m+1}A_{m+2}\cdots A_{k-1}\big)^{-1} &\text{if } m\leq k-2.
\end{cases}
\end{gather}
The primary reason why this definition is useful is because the relation
\begin{gather*}
\prod_{j=k}^{m+1} A_j = \prod_{j=k}^m A_j \cdot A_{m+1},
\end{gather*}
applies for all integers $k$ and $m$.

Next, we define the {\em $q$-shifted factorials}, for $k$ any integer, as
\begin{gather*}
\qrfac{a}{k} :=\prod\limits_{j=0}^{k-1} \big(1-aq^j\big),
\end{gather*}
and for $|q|<1$,
\begin{gather*}
\qrfac{a}{\infty} := \prod\limits_{j=0}^{\infty} \big(1-aq^j\big).
\end{gather*}
The parameter $q$ is called the {\em base}. With this definition, we have the {\em modified Jacobi theta function} defined as
\begin{gather*} \elliptictheta{a} := \pqrfac{a}{\infty}{p} \pqrfac{p/a}{\infty}{p},\end{gather*}
where $a\neq 0$ and $|p|<1$. We define the {\em $q, p$-shifted factorials} (or {\em theta shifted factorials}), for $k$ an integer, as
\begin{gather*}
\ellipticqrfac{a}{k} := \prod\limits_{j=0}^{k-1} \elliptictheta{aq^j}.
\end{gather*}
The parameter $p$ is called the {\em nome}. When the nome $p=0$, the modified theta function $\elliptictheta{a}$ reduces to $(1-a)$; and $\ellipticqrfac{a}{k}$ reduces to $ \pqrfac{a}{k}{q}$.

We use the short-hand notations
\begin{gather*}
\elliptictheta{a_1, a_2, \dots, a_r} := \elliptictheta{a_1} \elliptictheta{a_2}\cdots \elliptictheta{a_r},\\
\ellipticqrfac{a_1, a_2,\dots, a_r}{k} := \ellipticqrfac{a_1}{k} \ellipticqrfac{a_2}{k}\cdots \ellipticqrfac{a_r}{k},\\
\qrfac{a_1, a_2,\dots, a_r}{k} := \qrfac{a_1}{k} \qrfac{a_2}{k}\cdots \qrfac{a_r}{k}.
\end{gather*}

Observe that in view of \eqref{proddef}, we have
\begin{gather*}\ellipticqrfac{a}{0} = 1,
\end{gather*}
and, for $k$ an arbitrary integer
\begin{gather*}
\ellipticqrfac{a}{-k} =\frac{1}{\ellipticqrfac{aq^{-k}}{k}}.
\end{gather*}
Two important properties of the modified theta function are
\cite[equation~(11.2.42)]{GR90}
\begin{subequations}
\begin{gather}\label{GR11.2.42}
\elliptictheta{a} =\elliptictheta{p/a} =-a\elliptictheta{1/a},
\end{gather}
and \cite[p.~451, Example~5]{WW1996}
\begin{gather}\label{addf}
\elliptictheta{xy,x/y,uv,u/v}-\elliptictheta{xv,x/v,uy,u/y}=\frac uy\,\elliptictheta{yv,y/v,xu,x/u}.
\end{gather}
\end{subequations}
Using \eqref{GR11.2.42}, we can ``reverse the products'', and obtain the identity \cite[equation~(11.2.53)]{GR90}
\begin{gather}\label{GR11.2.53}
\ellipticqrfac{a}{-k} = \frac{\left( -q/a \right)^k}{\ellipticqrfac{q/a}{k}} q^{\binom k2}.
\end{gather}
We can combine the two identities above to obtain
\begin{gather}\label{GR11.2.50}
\big({aq^{-k};q,p}\big)_{k} =\ellipticqrfac{q/a}{k} \left(-{a}/{q}\right)^kq^{-{\binom k2}}.
\end{gather}
More generally, we have \cite[equation~(11.2.49) rewritten]{GR90}
\begin{gather}\label{GR11.2.49r}
\ellipticqrfac{q^{1-n}/a}{k} = \frac{\ellipticqrfac{a}{n}}
{\ellipticqrfac{a}{n-k}} \left( - {q}/{a}\right)^kq^{{\binom k2}-nk}.
\end{gather}
Two other useful identities we use are
\cite[equation~(11.2.47)]{GR90}
\begin{gather}\label{GR11.2.47}
\ellipticqrfac{a}{n+k} =\ellipticqrfac{a}{n}\ellipticqrfac{aq^n}{k},
\end{gather}
and a special case of \cite[equation~(11.2.55)]{GR90},
\begin{gather}\label{GR11.2.55a}
\ellipticqrfac{pa}{k} = (-1)^ka^{-k}q^{-{\binom k2}}\ellipticqrfac{a}{k}.
\end{gather}

We will use the notation of {\em basic hypergeometric series} (or $_r\phi_s$ series). This series is of the form
\begin{gather*}
_{r}\phi_s \left[\begin{matrix}
a_1,a_2,\dots,a_r \\
b_1,b_2,\dots,b_s\end{matrix} ; q, z
\right] :=
\sum_{k=0}^{\infty} \frac{\qrfac{a_1,a_2,\dots, a_r}{k}}{\qrfac{q, b_1,b_2,\dots, b_s}{k}}
\big( (-1)^kq^{\binom k2}\big)^{1+s-r} z^k.
\end{gather*}
See Gasper and Rahman~\cite{GR90} for the convergence conditions for these series.

A series is terminating if it is not an infinite series. Usually, this happens due to a factor $\qrfac{q^{-N}}{k}$ in the numerator of the summand. When $k>N$, this factor is $0$, and so the series terminates. Observe that in view of~\eqref{proddef}
\begin{gather}\label{term-from-below}
\frac{1}{\qrfac{q}{k}} = 0, \qquad \text{when} \quad k<0.
\end{gather} This ensures that the series terminates naturally from below too. That is, the terms with negative index in the series are $0$.

An $_{r+1}\phi_r$ series is called {\em well-poised} if $a_1q=a_2b_1=a_3b_2 = \cdots = a_{r+1}b_r$. In addition, if $a_2=qa_1^{1/2}, a_3=-qa_1^{1/2}$, then the series is called very-well-poised. Note that in such a case, taking the special parameter $a_1=a$, the summand contains the term
\begin{gather*}
\frac{\qrfac{qa^{1/2}, -qa^{1/2} }{k}}{\qrfac{a^{1/2}, -a^{1/2}}{k}} = \frac{1-aq^{2k}}{1-a},
\end{gather*}
which we call the {\em very-well-poised} part. This suggests the compact notation
\begin{gather*}_{r+1}W_r(a; a_4, a_5, \dots, a_{r+1}; q, z)
\end{gather*} for very-well-poised series.

An $_{r+1}\phi_r$ series is called {\em balanced} if $b_1\cdots b_r=qa_1\cdots a_{r+1}$ and $z=q$.

Recall that for an ordinary (resp.\ basic) hypergeometric series $\sum c_k$ the quotient $g(k) = c_{k+1}/c_k$ is a rational function in $k$ (resp.~$q^k$). Now, a series $\sum c_k$ is called an {\it elliptic hypergeometric series} if $g(k) = c_{k+1}/c_k$ is an elliptic function of $k$ with $k$ considered as a complex variable, i.e., the function~$g(x)$ is a~doubly periodic meromorphic function of the complex variable~$x$. Without loss of generality, by the theory of theta functions, we may assume that
\begin{gather*}
g(x)=\frac{\elliptictheta{ a_1q^x,a_2q^x,\dots,a_{r+1}q^x}} {\elliptictheta{q^{1+x},b_1q^x,\dots,b_rq^x}} z,
\end{gather*}
where the {\em elliptic balancing condition},
\begin{gather*}
a_1a_2\cdots a_{r+1}=qb_1b_2\cdots b_r,
\end{gather*}
holds. If we write $q=e^{2\pi i\sigma}$, $p=e^{2\pi i\tau}$, with complex $\sigma$, $\tau$, then $g(x)$ is indeed doubly periodic in~$x$ with periods $\sigma^{-1}$ and $\tau\sigma^{-1}$.

One usually requires $a_{r+1}=q^{-N}$ ($N$ being a nonnegative integer), so that the sum of an elliptic hypergeometric series is terminating, and hence convergent.

{\em Very-well-poised elliptic hypergeometric series} are def\/ined as
\begin{gather*}
{}_{r+1}V_r(a_1;a_6,\dots,a_{r+1};q,p):=
\sum_{k=0}^{\infty}\frac{\elliptictheta{a_1q^{2k}}}{\elliptictheta{a_1}}
\frac{(a_1,a_6,\dots,a_{r+1};q,p)_k}
{(q,a_1q/a_6,\dots,a_1q/a_{r+1};q,p)_k}q^k,
\end{gather*}
where
\begin{gather*}
q^2a_6^2a_7^2\cdots a_{r+1}^2=(a_1q)^{r-5}.
\end{gather*}
Note that in the elliptic case the number of pairs of numerator and denominator parameters involved in the construction of the {\em very-well-poised term} $\theta(a_1q^{2k};p)/\theta(a_1;p)$ is {\em four} (whereas in the basic case this number is {\em two}, in the ordinary case only {\em one}).
See Gasper and Rahman~\cite[Chapter~11]{GR90} for details. The notions of balancing, well-poisedness and very-well-poisedness were explained from the point of view of elliptic functions for the first time in Spiridonov's paper~\cite{VPS2002b}. This justifies the notations ${}_{10}V_9$ and ${}_{12}V_{11}$ (corresponding to ${}_8\phi_7$ and ${}_{10}\phi_9$ in the $p=0$ case) for the series below.

Frenkel and Turaev~\cite{FT1997} discovered the following $_{10}V_9$ summation formula (as a result of a~more general ${}_{12}V_{11}$ transformation), see \cite[equation~(11.4.1)]{GR90}:
\begin{gather}\label{10V9}
{}_{10}V_9\big(a;b,c,d,e,q^{-N};q,p\big) =\frac {(aq,aq/bc,aq/bd,aq/cd;q,p)_N} {(aq/b,aq/c,aq/d,aq/bcd;q,p)_N},
\end{gather}
where $a^2q^{N+1}=bcde$.

It should be clear that for $p=0$ elliptic hypergeometric series reduce to basic hypergeometric series. When the reference point is a basic hypergeometric result, we refer to the result by adding the word ``elliptic'' to it. For example, we refer to Frenkel and Turaev's~\cite{FT1997} elliptic extension of Bailey's $_{10}\phi_9$ transformation formula \cite[equation~(11.5.1)]{GR90} as the elliptic Bailey $_{10}\phi_9$ transformation. This formula transforms a $_{12}V_{11}$ series which is terminating, balanced and very-well-poised to a multiple of another series of the same kind. Similarly, we refer to elliptic extensions of Bressoud's matrix.

Next we note some identities useful in our calculations with multiple series.

The summand of the series we consider typically has a factor of the form
\begin{gather*}\ellipticvandermonde{x}{k}{n}.
\end{gather*}
This is referred to as the {\em elliptic Vandermonde product} of type $A$. When $p=0$, it reduces to
\begin{gather*}\vandermonde{x}{k}{n}.
\end{gather*}

Next we note two useful lemmas which in the $p=0$ case were given by Milne~\cite{Milne1997}. The first lemma shows an alternate way of writing the elliptic Vandermonde product.
\begin{Lemma}\label{th:e-magiclemma1} We have
\begin{gather*}
\sqprod n \ellipticqrfac{q\xover x }{k_r-k_s} = \ellipticvandermonde{x}{k}{n}\\
\hphantom{\sqprod n \ellipticqrfac{q\xover x }{k_r-k_s} =}{} \times
(-1)^{(n-1)\sumk}
q^{\sum\limits_{r=1}^n (r-1)k_r+ (n-1)\sumbinomial{k}{n} -\sum\limits_{r<s}k_rk_s } \smallprod{n} x_r^{nk_r-\sumk} .
\end{gather*}
\end{Lemma}

The $p=0$ case of this identity appeared as \cite[Lemma~3.12]{Milne1997}. Its proof in the elliptic case proceeds along the same lines.

Next, we have a lemma that extends the elementary identity \cite[equation~(I.12)]{GR90}
\begin{gather*}\label{GRI.12}
\frac{1}{\qrfac{q}{N-k}} = \frac{\qrfac{q^{-N}}{k} }{\qrfac{q}{N}} (-1)^k q^{Nk-\binom k2}.
\end{gather*}
\begin{Lemma}\label{th:e-magiclemma2} We have
\begin{gather*}
\sqprod n \frac{1}{\ellipticqrfac{q^{1+k_r-k_s}\xover x }{N_r-k_r}} =
\ellipticvandermonde{x}{k}{n}
\sqprod n \frac{\ellipticqrfac{q^{-N_s}\xover{x}}{k_r} }{\ellipticqrfac{q\xover{x}}{N_r} } \\
\hphantom{\sqprod n \frac{1}{\ellipticqrfac{q^{1+k_r-k_s}\xover x }{N_r-k_r}} =}{} \times (-1)^{\sumk} q^{\sumN\sumk-\binom{\sumk}2+\sum\limits_{r=1}^n (r-1)k_r}.
\end{gather*}
\end{Lemma}
The special case $N_r=k_r$ appears in Rosengren~\cite[equation~(3.8)]{HR2004}. When $p=0$, this reduces to Milne~\cite[Lemma~4.3]{Milne1997}. The proof uses Lemma~\ref{th:e-magiclemma1} but otherwise just proceeds as in the $p=0$ case.

Next, as motivation, we outline the basic hypergeometric case of our work, for $n=1$.

\section[The WP Bailey transform and lemma: a very short introduction]{The WP Bailey transform and lemma:\\ a very short introduction}\label{sec:WPBailey}

Here we outline an approach to the theory of basic hypergeometric series beginning with Jackson's sum. From this result one can extract all the components of the well-poised Bailey transform and lemma. Our purpose here is to explain and motivate the steps of our approach to find analogous results on root systems. This should help the reader get a bird's eye view of our work in later sections.

The idea for the WP Bailey transform and lemma was given by Bailey~\cite{WNB1947} and explained (and then extended) by Andrews~\cite{Andrews2001}. It was further studied by various authors (see \cite{AndBer2002,Jouhet2010, MSZ2009, VPS2002, SSY2017, SOW2003, ZL2016} for a small selection of references). The matrix inverse due to Bressoud~\cite{DB1983} is an important ingredient. The ideas presented below are primarily based on Andrews~\cite{Andrews2001}, but they have been sequenced in a way to explain our work in later sections.

The essential idea is as follows.
A pair of sequences $\alpha = (\alpha_k)$ and $\beta = (\beta_k)$ is given which follows a relationship of the form
\begin{gather}\label{Bailey-pair}
\beta_N = \sum_{j=0}^N B_{Nj} \alpha_j,
\end{gather}
where $\B = (B_{kj})$ is an infinite lower-triangular matrix called the Bressoud matrix (to be defined shortly), with entries indexed by $k$ and $j$. The entries $B_{kj} = B_{kj}(a,b)$ of the matrix depend on two parameters $a$ and $b$ (in addition to the parameter $q$). This relationship is called the {\em Bailey transform}, since it transforms a sequence into another sequence. Given $\B$, the pair of sequences
$(\alpha_k, \beta_k)$ is called a \emph{WP Bailey pair}. The \emph{WP Bailey lemma} is a method to construct sequences $(\alpha^{\prime}_k, \beta^{\prime}_k)$ that also form a WP Bailey pair. The $\alpha_k$ and $\beta_k$ are also dependent on~$a$ and~$b$. $B$ is a~lower triangular matrix, so the equation \eqref{Bailey-pair} corresponds to the matrix equation $\beta = \B\alpha$.

For our purposes, it is useful to note that one can extract the Bressoud matrix and a key WP Bailey pair from Jackson's sum, the $p=0$ case of \eqref{10V9} (given in Gasper and Rahman~\cite[equation~(2.6.2)]{GR90}),
\begin{gather}\label{q-dougall}
\sum_{k=0}^N\!\frac{(1\!-\!aq^{2k}) \qrfac{ a, b, c, d, a^2q^{1+N}\!/bcd, q^{-N}}{k}}
{(1\!-\!a)\qrfac{q, aq/b, aq/c, aq/d, bcdq^{-N}\!/a, aq^{N+1}}{k}}q^k\!
=\frac{\qrfac{aq, aq/bc, aq/bd, aq/cd}{N}} {\qrfac{aq/b, aq/c, aq/d, aq/bcd}{N}}.\!\!\!\!\!\!
\end{gather}

The $b\mapsto qa^2/bcd$ case of \eqref{q-dougall} may be written in the form \eqref{Bailey-pair}, where the Bressoud matrix $\B= (B_{kj}(a,b))$ is defined as
\begin{gather}
B_{k j}(a,b) := \frac{\qrfac{b}{j+k} \qrfac{b/a}{k-j}} { \qrfac{aq}{j+k}\qrfac{q}{k-j}},\label{Bressoud}
\end{gather}
and the sequences $(a_k)$ and $(b_k)$ are defined as
\begin{subequations}\label{alphabeta}
\begin{gather}
\alpha_k(a,b) := \frac{1-aq^{2k}}{1-a}
\frac{\qrfac{a, c, d, a^2q/bcd}{k} }{\qrfac{q, aq/c, aq/d, bcd/a}{k}}
\left( \frac{b}{a}\right)^{k} ,\label{alpha} \\
\beta_k(a,b) := \frac{\qrfac{b, bc/a, bd/a, aq/cd}{k}}{\qrfac{q, aq/c, aq/d, bcd/a}{k}}.\label{beta}
\end{gather}
\end{subequations}
Note that, in view of \eqref{term-from-below}, $B_{kj}(a,b)=0$ unless $k\geq j$, so $\B$ is indeed a lower-triangular matrix.

If $\B$, $(\alpha_k)$ and $(\beta_k)$ satisfy \eqref{Bailey-pair}, we say $(\alpha_k)$ and $(\beta_k)$ form a {\em WP Bailey pair}. Observe that if we set $d=aq/c$, we obtain the {\em unit WP Bailey pair}
\begin{gather*}
\alpha_k(a,b) :=
 \frac{1-aq^{2k}}{1-a}
\frac{\qrfac{a, a/b}{k} }{\qrfac{q, bq}{k}}
\left( \frac{b}{a}\right)^{k} ,\\
\beta_k(a,b) := \delta_{k, 0} =
\begin{cases}
1 & \text{if } k=0 , \\
0 & \text{otherwise.}
\end{cases}
\end{gather*}
The fact that this is a WP Bailey pair translates into an expression of the form
 \begin{gather}\label{Bailey-inversion-1}
\sum_{j=0}^N B_{Nj}(a,b) \alpha_j(a,b) = \delta_{N, 0}.
\end{gather}
One can view this as a matrix inversion and from here obtain an explicit formula for the inverse of~$\B$. To do that, we replace $N$ by $N-K$, shift the index, and after a change of variables (see remarks below), write this sum in the form
\begin{gather*}
\sum\limits_{j=K}^N B_{Nj} B^{-1}_{jK} =\delta_{N, K}.
\end{gather*}
From here, one can read off the formula for the inverse $\B^{-1}$ of the matrix $\B$. The entries of the (uniquely determined) inverse are given by
\begin{gather}
(B(a,b))^{-1}_{k j} =\frac{1-aq^{2k}}{1-a}\frac{1-bq^{2j}}{1-b}\frac{\qrfac{a}{j+k} \qrfac{a/b}{k-j}}
{ \qrfac{bq}{j+k}\qrfac{q}{k-j}}\left( \frac{b}{a}\right)^{k-j}.\label{Bressoud-inverse}
\end{gather}

\begin{Remarks*}\quad
\begin{enumerate}\itemsep=0pt
\item The entries of the inverse of the Bressoud matrix in \eqref{Bressoud-inverse} can be computed as follows:
In~\eqref{Bailey-inversion-1}, replace $N$ by $N-K$, $a$ by $aq^{2K}$, $b$ by $bq^{2K}$, and shift the index $j\mapsto j-K$. Then by
\begin{gather*}
B_{N-K,j-K}\left(aq^{2K},bq^{2K}\right)=\frac{(aq;q)_{2K}}{(b;q)_{2K}}B_{N,j}(a,b)
\end{gather*}
it follows that
\begin{gather*}
\frac{(aq;q)_{2K}}{(b;q)_{2K}}\alpha_{j-K}\left(aq^{2K},bq^{2K}\right)
\end{gather*}
can be identified as the $(j,K)$ entry of the inverse Bressoud matrix $\B^{-1}$. The details of the analogous computation in our work, for example, in the proof of Corollary~\ref{b1-inverse} below, vary slightly from those given here, but the essential idea is the same.
\item By comparing the entries of $\B$ and $\B^{-1}$, one sees that they are almost the same. In fact, Bressoud~\cite{DB1983} expressed them more symmetrically. Let $\M(a,b)=(M_{kj}(a,b))$ with
\begin{gather*}\label{Bressoud-original}
M_{k j}(a,b) := \frac{1-aq^{2j}}{1-a} \frac{\qrfac{b}{j+k} \qrfac{b/a}{k-j}}
{ \qrfac{aq}{j+k}\qrfac{q}{k-j}} a^{k-j}.
\end{gather*}
Then Bressoud showed that $\M(b,c)\M(a,b)=\M(a,c)$. In particular, this implies $\M(b,a)$ is the inverse of $\M(a,b)$. However, in our work, we find it beneficial to follow the exposition of Andrews~\cite{Andrews2001} rather than Bressoud's symmetric formulation. Both approaches are equivalent.
\end{enumerate}
\end{Remarks*}

Given $\B^{-1}$, and the WP Bailey pair $(\alpha_k,\beta_k)$ as in \eqref{alphabeta}, one has the {\it inverse relation}
\begin{gather*}
\alpha_N = \sum_{j=0}^N (B(a,b))^{-1}_{Nj} \beta_j.
\end{gather*}
This is again equivalent to Jackson's sum~\eqref{q-dougall}.

Given a Bailey pair, the WP Bailey lemma (see Andrews~\cite[Theorem~7]{Andrews2001}) gives a method to construct a new WP Bailey pair, with two additional parameters $\rho_1$ and $\rho_2$, given by
\begin{gather*}
{\alpha}^{\prime}_N(a,b) := \frac{\qrfac{\rho_1, \rho_2 }{N}}{\qrfac{aq/\rho_1, aq/\rho_2}{N}}
 \left(\frac{aq}{\rho_1\rho_2}\right)^{N} \alpha_N(a, b\rho_1\rho_2/aq),
\\
{\beta}^{\prime}_N (a,b) :=
\frac{\qrfac{b\rho_1/a , b\rho_2/a }{N}}{\qrfac{aq/\rho_1, aq/\rho_2}{N}}
 \sum_{k=0}^{N} \bigg(
\frac{ \qrfac{\rho_1, \rho_2}{k} }{\qrfac{b\rho_1/a, b\rho_2/a}{k}}
 \frac{\qrfac{b}{k+N}\qrfac{aq/\rho_1\rho_2}{N-k} }{\qrfac{b\rho_1\rho_2/a}{k+N}\qrfac{q}{N-k} }\\
\hphantom{{\beta}^{\prime}_N (a,b) :=}{} \times
 \frac{1- b\rho_1\rho_2q^{2k}/aq} { 1- b\rho_1\rho_2/aq}
 \left(\frac{aq}{\rho_1\rho_2}\right)^{k} \beta_k(a,b\rho_1\rho_2/aq)
 \bigg).
\end{gather*}
The WP Bailey lemma is the assertion that ${\alpha}^{\prime}_k(a,b)$ and ${\beta}^{\prime}_k(a,b)$ also form a WP Bailey pair. This step too depends on Jackson's sum~\eqref{q-dougall}.

If one begins with the Bailey pair $(\alpha_k)$ and $(\beta_k)$ given by \eqref{alphabeta}, then substituting the WP pair ${\alpha}^{\prime}_k(a,b)$ and ${\beta}^{\prime}_k(a,b)$ into the definition of a WP Bailey pair gives a transformation equivalent to Bailey's~$_{10}\phi_9$ transformation~\cite[equation~(2.9.1)]{GR90}:
\begin{gather}
\sum_{k=0}^N \frac {(1-aq^{2k}) \qrfac{ a, b, c, d, e, f, \lambda aq^{1+N}/ef, q^{-N}}{k}}
{(1-a)\qrfac{q, aq/b, aq/c, aq/d, aq/e, aq/f, efq^{-N}/\lambda, aq^{N+1}}{k}}
q^k \notag\\
\qquad{} = \frac{\qrfac{aq, aq/ef, \lambda q/e,\lambda q/f}{N}} {\qrfac{aq/e, aq/f, \lambda q, \lambda q/ef}{N}}\notag\\
\qquad\quad{} \times\sum_{k=0}^N \frac {(1-\lambda q^{2k}) \qrfac{ \lambda, \lambda b/a, \lambda c/a, \lambda d/a, e, f,
\lambda aq^{1+N}/ef , q^{-N}}{k}} {(1-\lambda)\qrfac{q, aq/b, aq/c, aq/d, \lambda q/e , \lambda q/f, efq^{-N}/a, \lambda q^{N+1}}{k}}
q^k,\label{10p9}
\end{gather}
where $\lambda =qa^2/bcd$.

The main summation and transformation formulas of basic hypergeometric series now follow from \eqref{10p9}. For example, Watson's $q$-analog of Whipple's transformation formula follows by taking the limit as $d\to \infty$, and relabeling the parameters. Other key identities such as the terminating, very well-poised $_6\phi_5$ summation and the terminating, balanced $_3\phi_2$ summation are immediate consequences of Watson's transformation formula. (Note that in the case of elliptic hypergeometric series, one cannot let parameters go to $0$ or $\infty$.)

To summarize, a special case of the Jackson summation yields a Bressoud matrix as well as a~WP Bailey pair. A further special case allows us to compute the inverse of the matrix. Another application of Jackson's sum is used to find the WP Bailey lemma. And finally, an application of this yields Bailey's $_{10}\phi_9$ transformation formula.

The ideas outlined above extend immediately to elliptic hypergeometric series. This was shown by Spiridonov~\cite{VPS2002}. Here we begin with Frenkel and Turaev's~\cite{FT1997} $_{10}V_9$ summation~\cite[equation~(11.4.1)]{GR90}, equation \eqref{10V9} above, the elliptic extension of Jackson's sum. Our goal is to extend this analysis to elliptic extensions of multiple basic hypergeometric series associated with root systems. We take a step in this direction in the next section.

\section[Consequences of an $A_n$ elliptic Jackson summation of Rosengren]{Consequences of an $\boldsymbol{A_n}$ elliptic Jackson summation\\ of Rosengren}\label{sec:dougall1}

When the dimension $n=1$, the WP Bailey transform and lemma are consequences of Jackson's sum. In this section, we consider one of Rosengren's~\cite{HR2004} $A_n$ elliptic Jackson sums, and investigate whether the ideas of Section~\ref{sec:WPBailey} can be extended to this setting.

A multiple series extension of \eqref{Bailey-pair} is as follows
\begin{gather}\label{n-Bailey-pair}
\beta_\N=\multsum{j}{N}{r} B_{\N\j} \alpha_\j.
\end{gather}
Here the sequences $\alpha= (\alpha_\k)$ and $\beta= (\beta_\k)$ are indexed by $n$-tuples $\k$ with non-negative integer components. The rows and columns of the matrix $\B=\left(B_{\k\j}\right)$ are indexed by $n$-tuples~$\k$ and~$\j$ of non-negative integers. Following Milne~\cite{Milne1997}, one can consider these $n$-tuples to be ordered lexicographically. With this ordering the matrix operations can be carried out in the usual manner.
Moreover, the matrix is lower-triangular, so $B_{\N\j}=0$ if $\j>\N$. The entries \smash{$B_{\k\j}= B_{\k\j}(a,b)$}, and the sequences $\alpha_\k = \alpha_\k(a,b)$ and $\beta_\k = \beta_\k(a,b)$ depend on two parameters~$a$ and~$b$ (in addition to~$p$ and~$q$ and perhaps other parameters).

The $A_n$ elliptic Jackson summation theorem that we use is due to Rosengren~\cite[Corollary~5.3]{HR2004}, its $p=0$ case being due to Milne~\cite{Milne1988}. Rosengren's result is
\begin{gather}
\multsum{k}{N}{r} \Bigg( \ellipticvandermonde{x}{k}{n}
\sqprod n \frac{\ellipticqrfac{q^{-N_s}\xover{x}}{k_r} }{\ellipticqrfac{q\xover{x}}{k_r} }\nonumber\\
\qquad\quad{}\times
\smallprod n \frac{\elliptictheta{ ax_rq^{k_r+\sumk}}}{\elliptictheta{ax_r}}
\smallprod n \frac{\ellipticqrfac{ax_r}{\sumk}\ellipticqrfac{dx_r, a^2x_rq^{1+\sumN}/bcd}{k_r}}
{\ellipticqrfac{ax_rq^{1+N_r}}{\sumk}\ellipticqrfac{ax_rq/b, ax_rq/c }{k_r}}\nonumber\\
\qquad\quad{} \times
\frac{\ellipticqrfac{b,c}{\sumk}}{\ellipticqrfac{aq/d,bcdq^{-\sumN}/a}{\sumk}}
q^{\sum\limits_{r=1}^n r k_r}\Bigg)\nonumber\\
\qquad{} = \frac{\ellipticqrfac{aq/bd, aq/cd}{\sumN}}{\ellipticqrfac{aq/d, aq/bcd}{\sumN}}
\smallprod n \frac{\ellipticqrfac{a x_rq, a x_r q/bc}{N_r}}{\ellipticqrfac{a x_rq/b, a x_rq/c}{N_r}} .\label{e-8p7-1}
\end{gather}

To extract the $A_n$ extension of Bressoud's matrix and the definition of a WP Bailey pair, we wish to write the case $b\mapsto qa^2/bcd$ of~\eqref{e-8p7-1} in the form of~\eqref{n-Bailey-pair}. After multiplying both sides by
\begin{gather*}
\frac{\smallprod n \ellipticqrfac{bx_r}{\sumN} } {\sqprod n \ellipticqrfac{q\xover{x}}{N_r} }
\end{gather*}
and rearranging factors, we obtain
\begin{gather}
\frac {\ellipticqrfac{b/a}{\sumN }} {\sqprod n \ellipticqrfac{qx_r/x_s}{N_r}}
\smallprod n \frac{ \ellipticqrfac{bx_r}{\sumN} }{\ellipticqrfac{ax_rq}{N_r}} \nonumber\\
\qquad\quad{} \times\multsum{j}{N}{r} \Bigg( \ellipticvandermonde{x}{j}{n}
\sqprod n \frac{\ellipticqrfac{q^{-N_s}\xover{x}}{j_r} }{\ellipticqrfac{q\xover{x}}{j_r} } \nonumber\\
\qquad\quad {}\times \smallprod n \frac{\elliptictheta{ ax_rq^{j_r+\sumj}}}{\elliptictheta{ax_r}}
\frac{\ellipticqrfac{ax_r}{\sumj}\ellipticqrfac{dx_r, bx_rq^{\sumN}}{j_r}}
{\ellipticqrfac{ax_rq^{1+N_r}}{\sumj}\ellipticqrfac{bcdx_r/a, ax_rq/c }{j_r}}\nonumber\\
\qquad\quad{} \times \frac{\ellipticqrfac{qa^2/bcd, c}{\sumj}}{\ellipticqrfac{aq/d,aq^{1-\sumN}/b}{\sumj}}
q^{\sum\limits_{r=1}^n r j_r}\Bigg)\nonumber\\
\qquad {} = \frac{\ellipticqrfac{bc/a, aq/cd}{\sumN}}{\ellipticqrfac{aq/d}{\sumN}
\sqprod n \ellipticqrfac{qx_r/x_s}{N_r}}
\smallprod n \frac{\ellipticqrfac{b x_r}{\sumN} \ellipticqrfac{bd x_r/a}{N_r}}
{\ellipticqrfac{bcd x_r/a, a x_rq/c}{N_r}} .\label{e-8p7-1-A}
\end{gather}
For a moment, ignore the product in front of the sum and compare the rest with the form~\eqref{n-Bailey-pair}. It is easy to separate the terms in the summand that depend on both $\N$ and $\j$ (and so appear as a~part of~$B_{\N\j}$) and the others that depend on~$\j$ (and thus comprise~$\alpha_\j$). The product on the right-hand side depends only on~$\N$, of course.

We use the elementary identity \eqref{GR11.2.47}
to combine terms, for example:
\begin{gather*}\smallprod n \ellipticqrfac{bx_r}{\sumN} \big(bx_rq^{\sumN};q,p\big)_{j_r} = \smallprod n \ellipticqrfac{bx_r}{\sumN+j_r}.
\end{gather*}
Further, we use \eqref{GR11.2.49r} to reverse the products in $\ellipticqrfac{aq^{1-\sumN}/b}{\sumj}$. We also require Lemma~\ref{th:e-magiclemma2}.

In this manner, we obtain an equation of the form
\begin{gather*}\multsum{j}{N}{r} B^{(1)}_{\N\j}(a, b) \alpha_\j(a,b) =\beta_\N(a,b),
\end{gather*}
where the matrix $\B^{(1)}= (B^{(1)}_{\k\j}(a,b))$ is defined in \eqref{e-B1}, and the sequences $\alpha_\k(a,b)$ and $\beta_\k(a,b)$ are as defined by~\eqref{e-alpha-Bailey-PairB1.2}. These considerations motivate the following definition and Theorem~\ref{th:e-Bailey-Pair-B1.2a} below.

\begin{Definition}[an $A_n$ elliptic Bressoud matrix] Let $\B^{(1)}= \big(B^{(1)}_{\k\j}(a,b)\big)$ with entries indexed by $(\k , \j)$ be defined as
\begin{gather}
B^{(1)}_{\k \j}(a,b) := \frac{\ellipticqrfac{b/a}{\sumk-\sumj}}{ \sqprod n \ellipticqrfac{q^{1+j_r-j_s}\xover x }{k_r-j_r}}
 \smallprod n \frac{\ellipticqrfac{bx_r}{j_r+\sumk}}{\ellipticqrfac{ax_rq}{k_r+\sumj}} .\label{e-B1}
\end{gather}
\end{Definition}
We call $\B^{(1)}$ an \emph{elliptic Bressoud matrix}, because it reduces to a form equivalent to \eqref{Bressoud} when $n=1$ and $p=0$.
The label $A_n$ is placed to indicate that it is associated with $A_n$ series. An equivalent form of the $p=0$ case of this matrix appeared in Milne~\cite[Theorem~3.41]{Milne1997}.

For some applications, it helps to use Lemma~\ref{th:e-magiclemma2} to rewrite the terms of the matrix $\B^{(1)}$ as follows
\begin{gather}
B^{(1)}_{\k\j}(a,b)= \frac
{\ellipticqrfac{b/a}{\sumk }}
{\sqprod n \ellipticqrfac{qx_r/x_s}{k_r}}
\smallprod n \frac{ \ellipticqrfac{bx_r}{\sumk} }{\ellipticqrfac{ax_rq}{k_r}}\nonumber\\
\hphantom{B^{(1)}_{\k\j}(a,b)=}{}\times \ellipticvandermonde{x}{j}{n} \smallprod n \frac{ \ellipticqrfac{bx_rq^{\sumk}}{j_r}} { \ellipticqrfac{ax_rq^{1+k_r}}{\sumj}} \nonumber\\
\hphantom{B^{(1)}_{\k\j}(a,b)=}{} \times \frac
{\sqprod n \ellipticqrfac{q^{-k_s}\xover{x}}{j_r}}
{\ellipticqrfac{aq^{1-\sumk}/b}{\sumj}}
 \left(\frac{a}{b}\right)^{\sumj} q^{\sum\limits_{r=1}^n r j_r} .\label{e-B1-form2}
\end{gather}

\begin{Definition}[WP Bailey pair with respect to a Bressoud matrix] Two sequences $\alpha_\N(a,b)$ and $\beta_\N(a,b)$ are said to form a WP Bailey pair with respect to a Bressoud matrix $\B$ if
\begin{gather} \label{BaileyPair-def}
\beta_\N(a,b) = \multsum{j}{N}{r} B_{\N\j}(a,b) \alpha_{\j} (a,b) .
\end{gather}
\end{Definition}
As we shall see, there are many multivariable Bressoud matrices. That is why we find it necessary to mention the matrix $\B$ with respect to which the sequences form a WP Bailey pair.

\begin{Theorem}[an elliptic WP Bailey pair with respect to $\B^{(1)}$]\label{th:e-Bailey-Pair-B1.2a}
The following sequences
\begin{subequations}\label{e-alpha-Bailey-PairB1.2}
\begin{gather}
\alpha_\k(a,b) :=
\smallprod n \frac{\elliptictheta{ax_rq^{k_r+\sumk}}}{\elliptictheta{ax_r}}
\frac{\ellipticqrfac{ax_r}{\sumk} \ellipticqrfac{dx_r}{k_r}}
{\ellipticqrfac{ax_rq/c, bcdx_r/a }{k_r} }\notag\\
\hphantom{\alpha_\k(a,b) :=}{}
\times
\frac{\ellipticqrfac{c, a^2q/bcd}{\sumk} }{\ellipticqrfac{aq/d}{\sumk}
\sqprod n \ellipticqrfac{q\xover x}{k_r}}
 \left( \frac{b}{a}\right)^{\sumk}, \label{e-alpha-Bailey-PairB1.2a} \\ \intertext{and}
\beta_\k(a,b) :=
\frac{\ellipticqrfac{bc/a, aq/cd}{\sumk}}{\ellipticqrfac{aq/d}{\sumk}
\sqprod n\ellipticqrfac{q\xover x}{k_r}}
\smallprod n \frac{\ellipticqrfac{bx_r}{\sumk}\ellipticqrfac{bdx_r/a}{k_r}}
{\ellipticqrfac{ax_rq/c, bcdx_r/a}{k_r}} ,
\label{e-beta-Bailey-PairB1.2a}
\end{gather}
\end{subequations}
form a WP Bailey pair with respect to $\B^{(1)}$.
\end{Theorem}
\begin{proof} We have already indicated how to discover this theorem. Alternatively, we can verify the theorem as follows. With $\alpha_\j(a,b)$ as above, we compute the sum
\begin{gather}\label{BP-general}
\multsum{j}{N}{r} B_{\N\j}(a, b) \alpha_\j(a,b) ,
\end{gather}
with $\B$ replaced by $\B^{(1)}$ to calculate $\beta_\N(a,b)$. (It is helpful to take the form~\eqref{e-B1-form2} for $B^{(1)}_{\N\j}(a,b)$.) The sum can be summed using the $b\mapsto qa^2/bcd$ case of \eqref{e-8p7-1}. After canceling some factors, we immediately obtain the expression in~\eqref{e-beta-Bailey-PairB1.2a} (with~$\k$ replaced by~$\N$).
\end{proof}

As a corollary, we obtain a unit WP Bailey pair.
\begin{Corollary}\label{e-Bailey-Pair-B1.1} The two sequences
\begin{gather*}
\alpha_\k(a,b) :=
\smallprod n \frac{\elliptictheta{ax_rq^{k_r+\sumk}}}{\elliptictheta{ax_r}}
\frac{\ellipticqrfac{ax_r}{\sumk}}
{\ellipticqrfac{ bx_rq }{k_r} } \cdot
\frac{\ellipticqrfac{ a/b}{\sumk} }{ \sqprod n\ellipticqrfac{q\xover x}{k_r}} \left( \frac{b}{a}\right)^{\sumk},\\
\intertext{and}
\beta_\k(a,b) := \smallprod n \delta_{k_r, 0},
\end{gather*}
form a WP Bailey pair with respect to $\B^{(1)}$.
\end{Corollary}
\begin{proof} Take $d=aq/c$ in \eqref{e-alpha-Bailey-PairB1.2} to obtain this simpler WP Bailey pair.
\end{proof}

Take $d=aq/c$ in \eqref{e-8p7-1-A} to obtain an equivalent form of Corollary~\ref{e-Bailey-Pair-B1.1}. This is an expression of the form
 \begin{gather}\label{n-Bailey-inversion}
 \multsum{j}{N}{r} F_{\N\j} \alpha_\j = \smallprod n \delta_{N_r, 0} .
\end{gather}
One can view this as a matrix inversion and from here obtain an explicit formula for the inverse of $\B^{(1)}$.
\begin{Corollary}[inverse of $\B^{(1)}$] \label{b1-inverse}
Let $\B^{(1)}= \big(B^{(1)}_{\k\j}(a,b)\big)$ be defined by \eqref{e-B1}. Then the entries of the inverse are given by
\begin{gather}
\big(B^{(1)}(a,b)\big)^{-1}_{\k \j} = \smallprod n \frac{\elliptictheta{ax_rq^{k_r+\sumk}}}{\elliptictheta{ax_r}}
\frac{\elliptictheta{bx_rq^{j_r+\sumj}}}{\elliptictheta{bx_r}}
\cdot
\left( \frac{b}{a}\right)^{\sumk-\sumj}\nonumber\\
\hphantom{\big(B^{(1)}(a,b)\big)^{-1}_{\k \j} =}{} \times
\frac{\ellipticqrfac{a/b}{\sumk-\sumj}}
{ \sqprod n \ellipticqrfac{q^{1+j_r-j_s}\xover x }{k_r-j_r}}
 \smallprod n \frac{\ellipticqrfac{ax_r}{j_r+\sumk}}{\ellipticqrfac{bx_rq}{k_r+\sumj}} .\label{e-B1-inverse}
\end{gather}
\end{Corollary}
\begin{Remark*}
When $p=0$, this matrix inversion is equivalent to a result of Milne~\cite[Theorem~3.41]{Milne1997}. Rosengren and the second author~\cite{RS2017} have proved more general elliptic matrix inversions, which contain the inverses of most of the matrices in this paper.
 \end{Remark*}
\begin{proof} Take $d=aq/c$ in \eqref{e-8p7-1-A} to obtain an expression of the form~\eqref{n-Bailey-inversion}. We replace $N_r$ by $N_r-K_r$, for $r=1, 2, \dots, n$ and obtain an expression of the form
 \begin{gather*}
 \multsum{j}{N_r-K}{r} F_{(\N-\K),\j} \alpha_\j=\smallprod n \delta_{N_r, K_r} .
\end{gather*}
We shift the indices to write it as
\begin{gather*}
\sum\limits_{\substack{{K_r\le j_r \le N_r} \\
{r =1,2,\dots, n}}} F_{(\N-\K),(\j-\K)} \alpha_{\j-\K} =\smallprod n \delta_{N_r, K_r} .
\end{gather*}
Next, we substitute $x_r\mapsto x_rq^{K_r}$, for $r=1, 2, \dots, n$, $a\mapsto aq^{\sumK}$, $b\mapsto bq^{\sumK}$, and simplify terms using \eqref{GR11.2.47} to obtain
\begin{gather*}
\sum\limits_{\substack{{K_r\le j_r \le N_r} \\
{r =1,2,\dots, n}}} \Bigg( \frac{\ellipticqrfac{b/a}{\sumN-\sumj}}
{ \sqprod n \ellipticqrfac{q^{1+j_r-j_s}\xover x }{N_r-j_r}}
 \smallprod n \frac{\ellipticqrfac{bx_r}{j_r+\sumN}}{\ellipticqrfac{ax_rq}{N_r+\sumj}} \\
\qquad\quad{} \times
\smallprod n \frac{\elliptictheta{ax_rq^{j_r+\sumj}}}{\elliptictheta{ax_r}}
\frac{\elliptictheta{bx_rq^{K_r+\sumK}}}{\elliptictheta{bx_r}}
\cdot \left( \frac{b}{a}\right)^{\sumj-\sumK}\\
\qquad\quad{} \times
\frac{\ellipticqrfac{a/b}{\sumj-\sumK}}
{ \sqprod n \ellipticqrfac{q^{1+K_r-K_s}\xover x }{j_r-K_r}}
 \smallprod n \frac{\ellipticqrfac{ax_r}{K_r+\sumj}}{\ellipticqrfac{bx_rq}{j_r+\sumK}} \Bigg)\\
\qquad{} = \smallprod n \delta_{N_r, K_r}.
\end{gather*}
From here it is easy to read off the entries of the inverse matrix.
\end{proof}

Consider the inverse relation
\begin{gather} \label{multivariable-inverse-relation}
\alpha_\N(a,b) = \multsum{j}{N}{r} (B(a,b))^{-1}_{\N\j} \beta_{\j} (a,b),
\end{gather}
where $\B=\B^{(1)}$, and $\alpha_\k$ and $\beta_\k$ are defined as in Theorem~\ref{th:e-Bailey-Pair-B1.2a} and $(B^{(1)}(a,b))^{-1}_{\k\j} $ is given by~\eqref{e-B1-inverse}. Next, simplify using~\eqref{GR11.2.50}, \eqref{GR11.2.47} and Lemma~\ref{th:e-magiclemma2}. If we now make the substitutions $a\mapsto qa^2/bcd$, $b\mapsto c$, $c\mapsto aq/bd$ and $d\mapsto aq/bc$, we again obtain~\eqref{e-8p7-1}. This is an interesting symmetry of Rosengren's result.

\begin{Theorem}[an elliptic $\big(\B^{(1)} \to\B^{(1)}\big)$ WP Bailey lemma] \label{th:e-WP-BaileyLemma-B1}
Suppose $\alpha_\N(a,b)$ and $\beta_\N(a,b)$ form a WP Bailey pair with respect to the matrix $\B^{(1)}$. Let ${\alpha}^{\prime}_\N(a,b)$ and ${\beta}^{\prime}_\N(a,b)$ be defined as follows
\begin{subequations}
\begin{gather}
{\alpha}^{\prime}_\N(a,b) := \frac{\ellipticqrfac{\rho_1 }{\sumN}}{\ellipticqrfac{aq/\rho_2}{\sumN}}
 \smallprod n \frac{\ellipticqrfac{\rho_2x_r}{N_r}} {\ellipticqrfac{ax_rq/\rho_1}{N_r}}
\cdot \left(\frac{aq}{\rho_1\rho_2}\right)^{\sumN} \alpha_\N(a, b\rho_1\rho_2/aq), \label{e-alphaprime-B1}\\
{\beta}^{\prime}_\N (a,b) :=
\frac{\ellipticqrfac{b\rho_1/a }{\sumN}}{\ellipticqrfac{aq/\rho_2}{\sumN}}
 \smallprod n \frac{\ellipticqrfac{b\rho_2x_r/a}{N_r}}{\ellipticqrfac{ax_rq/\rho_1}{N_r}}\notag \\
\hphantom{{\beta}^{\prime}_\N (a,b) :=}{} \times \multsum{k}{N}{r} \Bigg(
\frac{ \ellipticqrfac{\rho_1}{\sumk} }{\ellipticqrfac{b\rho_1/a}{\sumk}}
\smallprod n \frac{\ellipticqrfac{\rho_2x_r}{k_r}}
{\ellipticqrfac{b\rho_2x_r/a}{k_r}}
 \notag \\
\hphantom{{\beta}^{\prime}_\N (a,b) :=}{}
\times \smallprod n \frac{\elliptictheta{ b\rho_1\rho_2x_rq^{k_r+\sumk}/{aq}}}{\elliptictheta{ b\rho_1\rho_2x_r/{aq}}}
\frac{\ellipticqrfac{bx_r}{k_r+\sumN}}{\ellipticqrfac{b\rho_1\rho_2x_r/a}{\sumk+N_r}} \notag\\
\hphantom{{\beta}^{\prime}_\N (a,b) :=}{}
\times \frac{\ellipticqrfac{aq/\rho_1\rho_2}{\sumN-\sumk} }
{\sqprod n \ellipticqrfac{q^{1+k_r-k_s}\xover x }{N_r-k_r}}
\left(\frac{aq}{\rho_1\rho_2}\right)^{\sumk} \beta_\k(a,b\rho_1\rho_2/aq) \Bigg).\label{e-betaprime-B1}
\end{gather}
\end{subequations}
Then ${\alpha}^{\prime}_\N(a,b)$ and ${\beta}^{\prime}_\N(a,b)$ also form a WP Bailey pair with respect to $\B^{(1)}$.
\end{Theorem}

\begin{Remark*} When $p=0$, $n=1$ and $x_1=1$, then Theorem~\ref{th:e-WP-BaileyLemma-B1} reduces to Theorem~7 of Andrews~\cite{Andrews2001}. When $n=1$ and $x_1=1$, then Theorem~\ref{th:e-WP-BaileyLemma-B1} reduces to the elliptic WP Bailey lemma by Spiridonov~\cite[Theorem~4.3]{VPS2002}. When $p=0$, Theorem~\ref{th:e-WP-BaileyLemma-B1} reduces to an $A_n$ WP Bailey lemma given by Zhang and Liu~\cite{ZL2016}. A~slightly different formulation of Theorem~\ref{th:e-WP-BaileyLemma-B1} appears in unpublished notes of Warnaar~\cite{SOW-notes-2016}.
\end{Remark*}

\begin{proof} Our proof is an extension of Andrews' proof~\cite[Theorem~7]{Andrews2001}. We begin with the expression \eqref{e-betaprime-B1} for $\beta_\N^{\prime}(a,b)$. Substitute for $\beta_\k(a,b\rho_1\rho_2/aq)$ from \eqref{BaileyPair-def} written in the form:
\begin{gather}\label{BP-def2}
\beta_\k(a, b\rho_1\rho_2/aq) = \multsum{j}{k}{r} B_{\k\j}(a, b\rho_1\rho_2/aq) \alpha_\j(a,b\rho_1\rho_2/aq),
\end{gather}
with $\B$ replaced by $\B^{(1)}$, to obtain a double sum. After interchanging the sums, the inner sum is summed using \eqref{e-8p7-1} and the result can be recognized as the defining condition for a WP Bailey pair with respect to the matrix $\B^{(1)}$. The details are as follows.

We interchange the sums and shift the index using the following:
\begin{gather*}\label{interchange-shift-sums}
\multsum{k}{N}{r} \ \multsum{j}{k}{r} A_{\j,\k} =\multsum{j}{N}{r}
\sum\limits_{\substack{{0\le k_r \le N_r-j_r} \\
{r =1,2,\dots, n}}} A_{\j,\k+\j}.
\end{gather*}
We now need the following simplification which follows from Lemma~\ref{th:e-magiclemma2}, by replacing $N_r\mapsto N_r-j_r$ and $x_r\mapsto x_rq^{j_r}$, for $r=1, 2, \dots, n$,
\begin{gather*}
\sqprod n
 \frac{1}{\ellipticqrfac{q^{1+k_r-k_s+j_r-j_s}\xover x }{N_r-k_r-j_r}} =
\triprod{n}
 \frac{\elliptictheta{q^{k_r-k_s+j_r-j_s} \xover {x} }}{\elliptictheta{q^{j_r-j_s} \xover{x}}} \\
\qquad{} \times
 \sqprod n \frac{\ellipticqrfac{q^{-N_s+j_r}\xover x }{k_r}}{\ellipticqrfac{q^{1+j_r-j_s}\xover{x}}{N_r-j_r} } \cdot
(-1)^{\sumk} q^{\sumN\sumk-\sumk\sumj-\binom{\sumk}2 +\sum\limits_{r=1}^n (r-1)k_r}.
\end{gather*}
We also need the elementary identities \eqref{GR11.2.50} and \eqref{GR11.2.47}. In this manner, we obtain the following expression for $\beta_\N^{\prime}(a,b)$
\begin{gather}
\frac{\ellipticqrfac{b\rho_1/a }{\sumN}}{\ellipticqrfac{aq/\rho_2}{\sumN}}
 \smallprod n \frac{\ellipticqrfac{b\rho_2x_r/a}{N_r}}
{\ellipticqrfac{ax_rq/\rho_1}{N_r}}\nonumber\\
\qquad{} \times \multsum{j}{N}{r} \vast(
\frac{\ellipticqrfac{\rho_1}{\sumj}
\ellipticqrfac{aq/\rho_1\rho_2}{\sumN-\sumj} }{\ellipticqrfac{b\rho_1/a}{\sumj}
\sqprod n \ellipticqrfac{q^{1+j_r-j_s}\xover{x}}{N_r-j_r} }\nonumber\\
\qquad\qquad{} \times \smallprod n
\frac{\ellipticqrfac{ \rho_2x_r }{j_r}
\ellipticqrfac{bx_r}{j_r+\sumN}\qrfac{b\rho_1\rho_2x_r/aq}{j_r+\sumj}}
{\ellipticqrfac{ b\rho_2x_r/a }{j_r} \ellipticqrfac{ax_rq}{j_r+\sumj}
 \ellipticqrfac{b\rho_1\rho_2x_r/a}{\sumj+N_r}}\nonumber\\
\qquad\qquad{} \times \smallprod n \frac{\elliptictheta{b\rho_1\rho_2x_rq^{j_r+\sumj}/aq}}{\elliptictheta{ b\rho_1\rho_2x_r/aq}} \cdot
\left(\frac{aq}{\rho_1\rho_2} \right)^{\sumj} \alpha_\j (a, b\rho_1\rho_2/aq)\nonumber\\
\qquad{} \times
\sum\limits_{\substack{{0\le k_r \le N_r-j_r} \\ {r =1,2,\dots, n}}} \Bigg(
\triprod n \frac{\elliptictheta{q^{j_r-j_s+k_r-k_s}x_r/x_s}}
{\elliptictheta{q^{j_r-j_s}x_r/x_s}}
\sqprod n\frac{\ellipticqrfac{q^{-N_s+j_r}\xover x}{k_r}}{\ellipticqrfac{q^{1+j_r-j_s}\xover x}{k_r}}\nonumber\\
\qquad\qquad{} \times
\smallprod n \frac{\elliptictheta {b\rho_1\rho_2 x_rq^{j_r+\sumj+k_r+|\mathbf{k}|}/aq}}
{\elliptictheta{b\rho_1\rho_2 x_rq^{j_r+\sumj}/aq}}
\smallprod n \frac {\ellipticqrfac{b\rho_1\rho_2x_rq^{j_r+\sumj}/aq}{\sumk}}
{\ellipticqrfac{b\rho_1\rho_2x_rq^{N_r+\sumj}/a}{\sumk}}\nonumber\\
\qquad\qquad{} \times \smallprod n \frac{\ellipticqrfac{\rho_2x_rq^{j_r},
bx_rq^{j_r+\sumN} }{k_r}}
{\ellipticqrfac{ b\rho_2 x_rq^{j_r}/a, ax_rq^{1+j_r+\sumj} }{k_r}}\nonumber\\
\qquad\qquad{} \times
\frac{\ellipticqrfac{\rho_1q^{\sumj},b\rho_1\rho_2/a^2q }{\sumk}}
{\ellipticqrfac{b\rho_1q^{\sumj}/a,\rho_1\rho_2q^{\sumj -\sumN}/a }{\sumk}}
q^{\sum\limits_{r=1} ^{n}r\,k_r}\Bigg)\vast).\label{e-bailey-lemma-B1-a}
\end{gather}
The inner sum can be summed using \eqref{e-8p7-1}. Take the equivalent formulation of \eqref{e-8p7-1} obtained by replacing $c$ by $a^2q^{1+\sumN}/bcd$ and use the following substitutions: $x_r\mapsto x_rq^{j_r}$ and $N_r\mapsto N_r-j_r$ for $r=1, 2, \dots, n$,
$a\mapsto b\rho_1\rho_2q^{\sumj}/aq$, $b\mapsto \rho_1q^{\sumj}$, $c\mapsto \rho_2$, $d\mapsto bq^{\sumN}$. In this manner, we find that the inner sum in~\eqref{e-bailey-lemma-B1-a} equals
\begin{gather*}
\frac{\ellipticqrfac{b/a, \rho_2q^{-\sumN}/a}{\sumN-\sumj}}
{\ellipticqrfac{\rho_1\rho_2q^{\sumj-\sumN}/a, b\rho_1 q^{\sumj}/a}{\sumN-\sumj}}
\smallprod n \frac{\ellipticqrfac{b\rho_1\rho_2x_rq^{j_r+\sumj}/a, \rho_1 q^{-N_r}/ax_r}{N_r-j_r}}
{\ellipticqrfac{b\rho_2x_rq^{j_r}/a, q^{-N_r-\sumj}/ax_r}{N_r-j_r}}.
\end{gather*}
Now we use \eqref{GR11.2.50} and \eqref{GR11.2.47} to write the sum (and therefore, $\beta_\N^{\prime}(a,b)$) in the form
\begin{gather*}\label{BP-prime}
\multsum{j}{N}{r} B_{\N\j} \alpha^{\prime}_\j (a,b),
\end{gather*}
where $\B= \B^{(1)}$ and $\alpha^{\prime}_\j (a,b)$ is defined by \eqref{e-alphaprime-B1}. This shows that ${\alpha}^{\prime}_\N(a,b)$ and ${\beta}^{\prime}_\N(a,b)$ form a WP Bailey pair with respect to $\B^{(1)}$.
\end{proof}

\begin{Remark}\label{rem:B1-B1-10p9} An elliptic $A_n$ Bailey $_{10}\phi_9$ transformation formula follows immediately by applying the $\B^{(1)}\to\B^{(1)}$ elliptic Bailey lemma in Theorem~\ref{th:e-WP-BaileyLemma-B1} to the WP Bailey pair in Theorem~\ref{th:e-Bailey-Pair-B1.2a}. This $A_n$ elliptic Bailey transformation is due to Rosengren~\cite[Corollary~8.1]{HR2004}. When $p=0$, this reduces to an $A_n$ Bailey $_{10}\phi_9$ transformation formula found by Denis and Gustafson~\cite{DG1992} and independently, by Milne and Newcomb~\cite[Theorem~3.1]{MN1996}. When $p=0$, this was noted previously by Zhang and Liu~\cite{ZL2016}.
\end{Remark}

We used an equivalent, altered formulation of the elliptic Jackson sum in \eqref{e-8p7-1} in our proof of Theorem~\ref{th:e-WP-BaileyLemma-B1}. More precisely, we altered~\eqref{e-8p7-1} by a specific substitution of variables with the effect that the occurrence of the nonnegative integer sequence $\N$ in the respective factors got changed. By using the altered formulation of~\eqref{e-8p7-1}, we can extract another Bressoud matrix and corresponding WP Bailey pair. The matrix $\B^{(2)}= (B^{(2)}_{\k\j}(a,b))$ is defined as follows.
\begin{Definition}[an $A_n$ elliptic Bressoud matrix] We define the matrix $\B^{(2)}$ with entries indexed by $(\k , \j)$ as
\begin{gather}
B^{(2)}_{\k \j}(a,b) := \frac{\ellipticqrfac{b}{\sumk+\sumj} \smallprod n \ellipticqrfac{bq^{\sumk-k_r}/ax_r}{k_r-j_r}}
{\sqprod n \ellipticqrfac{q^{1+j_r-j_s}\xover x }{k_r-j_r} \smallprod n \ellipticqrfac{ax_rq}{k_r+\sumj}}. \label{e-B2}
\end{gather}
\end{Definition}

\begin{Theorem}[an elliptic WP Bailey pair with respect to $\B^{(2)}$]\label{th:e-Bailey-Pair-B2.2}
The two sequences
\begin{subequations}
\begin{gather}
\alpha_\k(a,b) :=
\smallprod n \frac{\elliptictheta{ax_rq^{k_r+\sumk}}}{\elliptictheta{ax_r}}
\frac{\ellipticqrfac{ax_r}{\sumk} \ellipticqrfac{cx_r, dx_r}{k_r}}
{\ellipticqrfac{bcdx_r/a }{k_r} }\notag\\
\hphantom{\alpha_\k(a,b) :=}{} \times
 \frac{\ellipticqrfac{ a^2q/bcd}{\sumk} }{\ellipticqrfac{aq/c, aq/d}{\sumk}
 \sqprod n \ellipticqrfac{q\xover x}{k_r}}
 \Big(\frac{b}{a}\Big)^{\sumk}
 q^{\sum\limits_{r<s}k_r k_s}
 \smallprod n x_r^{-k_r},
 \label{e-alpha-Bailey-PairB2.2} \\
 \intertext{and}
\beta_\k(a,b) :=
\frac{\ellipticqrfac{b, bc/a, bd/a}{\sumk}}{\ellipticqrfac{aq/c, aq/d}{\sumk}
\sqprod n \ellipticqrfac{q\xover x}{k_r}}
\smallprod n \frac{\ellipticqrfac{aq^{1+\sumk-k_r}/cdx_r}{k_r}}
{\ellipticqrfac{ bcdx_r/a}{k_r}},\label{e-beta-Bailey-PairB2.2}
\end{gather}
\end{subequations}
form a WP Bailey pair with respect to $\B^{(2)}$.
\end{Theorem}
\begin{proof} The proof is analogous to that of Theorem~\ref{th:e-Bailey-Pair-B1.2a}. The proof requires \eqref{e-8p7-1}. With $\alpha_\j(a,b)$ as above, we compute the sum \eqref{BP-general}, with $\B=\B^{(2)}$, to calculate $\beta_\N(a,b)$. We take the altered formulation of \eqref{e-8p7-1} obtained by replacing $c$ by $a^2q^{1+\sumN}/bcd$. The sum can be summed using the $b\mapsto qa^2/bcd$ case of this altered form of \eqref{e-8p7-1}. After canceling some factors, we immediately obtain the expression \eqref{e-beta-Bailey-PairB2.2} (with $\k$ replaced by $\N$).
\end{proof}

\begin{Theorem}[an elliptic ($\B^{(1)} \to\B^{(2)}$) WP Bailey lemma] \label{th:e-WP-BaileyLemma-B1-B2}
Suppose $\alpha_\N(a,b)$ and $\beta_\N(a,b)$ form a WP Bailey pair with respect to~$\B^{(1)}$. Let ${\alpha}^{\prime}_\N(a,b)$ and ${\beta}^{\prime}_\N(a,b)$ be defined as follows
\begin{subequations}
\begin{gather}
{\alpha}^{\prime}_\N(a,b) :=
 \frac{ \smallprod n \ellipticqrfac{\rho_1x_r, \rho_2x_r}{N_r}}
{\ellipticqrfac{aq/\rho_1, aq/\rho_2}{\sumN}} \left(\frac{aq}{\rho_1\rho_2}\right)^{\sumN}
\smallprod n x_r^{-N_r}\cdot q^{\sum\limits_{r<s}N_r N_s} \alpha_\N(a, b\rho_1\rho_2/aq), \!\!\!\! \label{e-alphaprime-B2} \\
{\beta}^{\prime}_\N (a,b) :=
\frac{\ellipticqrfac{b\rho_1/a, b\rho_2/a }{\sumN}}{\ellipticqrfac{aq/\rho_1, aq/\rho_2}{\sumN}} \nonumber\\
\hphantom{{\beta}^{\prime}_\N (a,b) := }{}
 \times \multsum{k}{N}{r} \Bigg(
 \frac{ \smallprod n \ellipticqrfac{\rho_1x_r, \rho_2x_r}{k_r}}
{\ellipticqrfac{b\rho_1/a, b\rho_2/a}{\sumk}}
\frac{\ellipticqrfac{b}{\sumN+\sumk} }
{\sqprod n \ellipticqrfac{q^{1+k_r-k_s}\xover x }{N_r-k_r}}\nonumber\\
\hphantom{{\beta}^{\prime}_\N (a,b) := }{} \times
\smallprod n \frac{\elliptictheta{ b\rho_1\rho_2x_rq^{k_r+\sumk}/{aq}}}{\elliptictheta{ b\rho_1\rho_2x_r/{aq}}}
\smallprod n \frac{\ellipticqrfac{aq^{1+\sumN-N_r}/\rho_1\rho_2 x_r}{N_r-k_r}}{\ellipticqrfac{b\rho_1\rho_2x_r/a}{\sumk+N_r}} \notag\\
\hphantom{{\beta}^{\prime}_\N (a,b) := }{}
 \times \left(\frac{aq}{\rho_1\rho_2}\right)^{\sumk} \smallprod n x_r^{-k_r} \cdot q^{\sum\limits_{r<s} k_r k_s}
\beta_\k(a,b\rho_1\rho_2/aq) \Bigg).\label{e-betaprime-B2}
\end{gather}
\end{subequations}
Then ${\alpha}^{\prime}_\N(a,b)$ and ${\beta}^{\prime}_\N(a,b)$ form a WP Bailey pair with respect to $\B^{(2)}$.
\end{Theorem}

\begin{proof} The proof is analogous to the proof of Theorem~\ref{th:e-WP-BaileyLemma-B1}. We begin with the expres\-sion~\eqref{e-betaprime-B2} for $\beta_\N^{\prime}(a,b)$. Substitute for $\beta_\k(a,b\rho_1\rho_2/aq)$ in \eqref{BP-def2} with $\B=\B^{(1)}$. The inner sum can be summed using \eqref{e-8p7-1}. We take the altered formulation of \eqref{e-8p7-1} obtained by replacing $c$ by $a^2q^{1+\sumN}/bcd$. We use the following substitutions: $x_r\mapsto x_rq^{j_r}$ and $N_r\mapsto N_r-j_r$ for $r=1, 2, \dots, n$, $a\mapsto b\rho_1\rho_2q^{\sumj}/aq$, $b\mapsto bq^{\sumj+\sumN}$, $c\mapsto \rho_1$, $d\mapsto \rho_2$. Now after some simplification, we find that ${\alpha}^{\prime}_\N(a,b)$ and ${\beta}^{\prime}_\N(a,b)$ form a WP Bailey pair with respect to $\B^{(2)}$.
\end{proof}

An elliptic $A_n$ Bailey $_{10}\phi_9$ transformation formula due to Rosengren \cite[Corollary~8.1]{HR2004} follows immediately by applying the $\B^{(1)}\to\B^{(2)}$ elliptic Bailey lemma in Theorem~\ref{th:e-WP-BaileyLemma-B1-B2} to the WP Bailey pair in Theorem~\ref{th:e-Bailey-Pair-B1.2a}. This $A_n$ elliptic Bailey transformation formula is the same as obtained in Remark~\ref{rem:B1-B1-10p9}.

We have seen two $A_n$ Bressoud matrices which follow from the same $A_n$ elliptic Jackson sum. We also obtained a WP Bailey lemma that transforms a WP Bailey pair with respect to a~mat\-rix~$\B^{(1)}$ into a more complicated WP Bailey pair (with 2 additional parameters). However, this time the WP Bailey pair is with respect to the matrix $\B^{(2)}$. In the next section, we explore the consequences of another~$A_n$ elliptic Jackson sum.

\section[An $A_n$ elliptic Bailey transformation]{An $\boldsymbol{A_n}$ elliptic Bailey transformation}\label{sec:dougall5}

In the previous section, we examined the consequences of one of Rosengren's elliptic Jackson summation over $A_n$. In this section, we consider another $A_n$ elliptic Jackson summation. We will find that we can obtain some WP Bailey lemmas relating to the Bressoud matrices obtained in
Section~\ref{sec:dougall1}. In addition, we find another useful Bressoud matrix closely related to~$\B^{(2)}$.

The $A_n$ elliptic Jackson summation we use in this section is due to Rosengren and the second author~\cite{RS2017}. The $p=0$ case is due to the second author~\cite{MS2008}. We have
\begin{gather}
\multsum{k}{N}{r} \Bigg( \ellipticvandermonde{x}{k}{n}
\sqprod n \frac{\ellipticqrfac{q^{-N_s}\xover{x}}{k_r} }{\ellipticqrfac{q\xover{x}}{k_r} } \notag \\
\qquad\qquad{}\times \smallprod n \frac{\ellipticqrfac{d/x_r}{\sumk}
\ellipticqrfac{a^2 x_r q^{\sumN+1}/bcd}{k_r}
 \ellipticqrfac{bcd/a x_r}{\sumk-k_r}}
{\ellipticqrfac{d/x_r}{\sumk-k_r} \ellipticqrfac{ax_rq/d}{k_r}\ellipticqrfac{bcdq^{-N_r}/a x_r}{\sumk}}
\nonumber\\
\qquad\qquad{} \times \frac{\elliptictheta{aq^{2\sumk}}}{\elliptictheta{a}}
\frac{\ellipticqrfac{a, b, c}{\sumk}}{\ellipticqrfac{aq^{\sumN+1}, aq/b, aq/c}{\sumk}}
q^{\sum\limits_{r=1}^n rk_r} \Bigg)\notag \\
\qquad{} = \frac{\ellipticqrfac{aq, aq/bc}{\sumN}}{\ellipticqrfac{aq/b, aq/c}{\sumN}}
\smallprod n \frac{\ellipticqrfac{a x_rq/bd, a x_r q/cd}{N_r}}{\ellipticqrfac{a x_rq/d, a x_rq/bcd}{N_r}}. \label{e-8p7-5}
\end{gather}
Observe that \eqref{e-8p7-5} has a simpler very-well-poised part, namely
\begin{gather*}\frac{\elliptictheta{aq^{2\sumk}}}{\elliptictheta{a}}
\end{gather*}
compared to the one in \eqref{e-8p7-1}, given by
\begin{gather*}
\smallprod n \frac{\elliptictheta{ ax_rq^{k_r+\sumk}}}{\elliptictheta{ax_r}}.
\end{gather*}
We begin with a WP Bailey lemma which follows from \eqref{e-8p7-5}.

\begin{Theorem}[an elliptic $\big(\B^{(2)} \to\B^{(1)}\big)$ WP Bailey lemma] \label{th:e-WP-BaileyLemma-B2-B1}
Suppose $\alpha_\N(a,b)$ and $\beta_\N(a,b)$ form a WP Bailey pair with respect to the matrix~$\B^{(2)}$. Let ${\alpha}^{\prime}_\N(a,b)$ and ${\beta}^{\prime}_\N(a,b)$ be defined as follows
\begin{gather*}
{\alpha}^{\prime}_\N(a,b) := \frac{\ellipticqrfac{\rho_1, \rho_2}{\sumN}}
 {\smallprod n \ellipticqrfac{ax_rq/\rho_1, ax_rq/\rho_2}{N_r}}
 \left(\frac{aq}{\rho_1\rho_2}\right)^{\sumN} 
 \smallprod n x_r^{N_r}\cdot
q^{-\sum\limits_{r<s} N_r N_s} \alpha_\N(a, b\rho_1\rho_2/aq),\\
{\beta}^{\prime}_\N (a,b) := \frac{\ellipticqrfac{b\rho_1/a, b\rho_2/a }{\sumN}}
{\smallprod n \ellipticqrfac{ax_rq/\rho_1, ax_rq/\rho_2}{N_r}}\\
\hphantom{{\beta}^{\prime}_\N (a,b) :=}{}
 \times \multsum{k}{N}{r} \Bigg( \frac{ \ellipticqrfac{\rho_1, \rho_2}{\sumk}}
{\ellipticqrfac{b\rho_1/a, b\rho_2/a}{\sumk}}
 \frac{\elliptictheta{ b\rho_1\rho_2q^{2\sumk}/{aq}}}{\elliptictheta{ b\rho_1\rho_2/{aq}}}\\
\hphantom{{\beta}^{\prime}_\N (a,b) :=}{}\times
\frac{\smallprod n \ellipticqrfac{bx_r}{\sumN+k_r}
 \ellipticqrfac{ax_rq^{1+k_r-\sumk}/\rho_1\rho_2 }{N_r-k_r}}{
\ellipticqrfac{b\rho_1\rho_2/a}{\sumN+\sumk}
\sqprod n \ellipticqrfac{q^{1+k_r-k_s}\xover x }{N_r-k_r}} \notag\\
\hphantom{{\beta}^{\prime}_\N (a,b) :=}{} \times
\left(\frac{aq}{\rho_1\rho_2}\right)^{\sumk} \smallprod n x_r^{k_r}\cdot q^{-\sum\limits_{r<s} k_rk_s} \beta_\k(a,b\rho_1\rho_2/aq) \Bigg).
\end{gather*}
Then ${\alpha}^{\prime}_\N(a,b)$ and ${\beta}^{\prime}_\N(a,b)$ form a WP Bailey pair with respect to $\B^{(1)}$.
\end{Theorem}

\begin{proof} The proof is analogous to the proof of Theorem~\ref{th:e-WP-BaileyLemma-B1}. The only difference is that this time we sum
the inner sum using \eqref{e-8p7-5}, with the substitutions: $x_r\mapsto x_rq^{j_r}$ and $N_r\mapsto N_r-j_r$ for $r=1, 2, \dots, n$,
$a\mapsto b\rho_1\rho_2q^{2\sumj}/aq$, $b\mapsto \rho_1q^{\sumj}$, $c\mapsto \rho_2q^{\sumj}$, $d\mapsto b\rho_1\rho_2q^{\sumj}/a^2q $. The rest of the computations are very similar.
\end{proof}

A new elliptic $A_n$ Bailey $_{10}\phi_9$ transformation formula follows immediately by applying the $\B^{(2)} \to\B^{(1)}$ elliptic Bailey lemma in Theorem~\ref{th:e-WP-BaileyLemma-B2-B1} to the WP Bailey pair in Theorem~\ref{th:e-Bailey-Pair-B2.2}.

\begin{Theorem}[an $A_n$ elliptic Bailey $_{10}\phi_9$ transformation] \label{th:e-10p9-6}
Let $\lambda = qa^2/bcd$. Then
\begin{gather}
\multsum{k}{N}{r} \Bigg( \ellipticvandermonde{x}{k}{n}
\sqprod n \frac{\ellipticqrfac{q^{-N_s}\xover{x}}{k_r} }{\ellipticqrfac{q\xover{x}}{k_r} } \notag \\
\qquad\quad{} \times \smallprod n \frac{\elliptictheta{a x_r q^{k_r+\sumk}}}{\elliptictheta{a x_r}}
 \smallprod n \frac{\ellipticqrfac{a x_r}{\sumk}}
 {\ellipticqrfac{a x_rq^{1+N_r}}{\sumk}} \notag\\
\qquad\quad{} \times \smallprod n \frac{\ellipticqrfac{cx_r, dx_r, \lambda ax_r q^{\sumN+1}/ef}{k_r}}
 {\ellipticqrfac{a x_rq/b, a x_rq/e, ax_rq/f}{k_r}} \notag\\
\qquad\quad{} \times \frac{\ellipticqrfac{b, e, f}{\sumk}}
{\ellipticqrfac{aq/c, aq/d, efq^{-\sumN}/\lambda}{\sumk}}
q^{\sum\limits_{r=1}^n r k_r} \Bigg)\notag \\
\qquad{} = \frac{\ellipticqrfac{\lambda q/e, \lambda q/f}{\sumN}}{\ellipticqrfac{ \lambda q, \lambda q/ef}{\sumN}}
\smallprod n \frac{\ellipticqrfac{a x_rq, a x_r q/ef}{N_r}}
{\ellipticqrfac{a x_rq/e, a x_rq/f}{N_r}}\notag \\
\qquad\quad{} \times \multsum{k}{N}{r} \Bigg( \ellipticvandermonde{x}{k}{n}
\sqprod n \frac{\ellipticqrfac{q^{-N_s}\xover{x}}{k_r} }{\ellipticqrfac{q\xover{x}}{k_r} } \notag \\
\qquad\quad{} \times \frac{\elliptictheta{\lambda q^{2\sumk}}}{\elliptictheta{\lambda }}
\frac{\ellipticqrfac{\lambda ,\lambda c/a, \lambda d/a, e, f}{\sumk}}
{\ellipticqrfac{\lambda q^{\sumN+1}, aq/c, aq/d, \lambda q/e, \lambda q/f}{\sumk}}\nonumber\\
\qquad\quad{} \times \smallprod n \frac{\ellipticqrfac{\lambda b/ax_r }{\sumk}
\ellipticqrfac{\lambda a x_r q^{\sumN+1}/ef}{k_r} }
{\ellipticqrfac{efq^{-N_r}/ax_r}{\sumk} \ellipticqrfac{ax_rq/b}{k_r} } \notag \\
\qquad\quad{} \times \smallprod n \frac{\ellipticqrfac{ef/ax_r}{\sumk-k_r}}
{\ellipticqrfac{\lambda b/ax_r}{\sumk-k_r}
} \cdot q^{\sum\limits_{r=1}^n r k_r} \Bigg). \label{e-10p9-6}
\end{gather}
\end{Theorem}
\begin{Remark*} When $p=0$ and $n=1$, this reduces to an equivalent form of Bailey's $_{10}\phi_9$ transformation formula~\eqref{10p9}, given in \cite[equation~(2.9.1)]{GR90}. After discovering this result, the authors were informed by Rosengren~\cite{HR-PC-2016} that he obtained the same result by following the approach used in \cite{BS1998, HR2004}. Other elliptic Bailey transformations on root systems were given previously by Rosengren~\cite{HR2004, HR2017a}.
\end{Remark*}
\begin{proof} This elliptic $A_n$ Bailey $_{10}\phi_9$ transformation formula is obtained by applying the $\B^{(2)} \to\B^{(1)}$ elliptic Bailey lemma in Theo\-rem~\ref{th:e-WP-BaileyLemma-B2-B1} to the WP Bailey pair with respect to $\B^{(2)}$ in Theo\-rem~\ref{th:e-Bailey-Pair-B2.2}. The result is a WP Bailey pair with respect to the matrix $\B^{(1)}$. Written explicitly, this is an equivalent form of \eqref{e-10p9-6}. The details are as follows.

First write the relation \eqref{BaileyPair-def} explicitly
\begin{gather*}
\beta^{\prime}_\N(a,b) = \multsum{j}{N}{r} B^{(1)}_{\N\j}(a,b) \alpha^{\prime}_{\j} (a,b),
\end{gather*}
with $B^{(1)}_{\N\j}(a,b)$ written in the form \eqref{e-B1-form2}, and $\alpha^{\prime}_k$ and $\beta^{\prime}_k$ given by Theorem~\ref{th:e-WP-BaileyLemma-B2-B1}. Here,
replace $\alpha_k(a, b\rho_1\rho_2/aq) $,
$\beta_k(a, b\rho_1\rho_2/aq)$ by the corresponding expressions from Theorem~\ref{th:e-Bailey-Pair-B2.2} (with $b$ replaced by $b\rho_1\rho_2/aq$). After some algebraic calculations involving the use of \eqref{GR11.2.53}, \eqref{GR11.2.47},
and Lemma~\ref{th:e-magiclemma2},
 we obtain a formula resembling \eqref{e-10p9-6}.
Next, set $\rho_1=e$, $\rho_2=f$ and $b\mapsto bq^{-\sumN}$. Finally, replace $b$ by $a^3q^{2+\sumN}/bcdef = \lambda a q^{1+\sumN}/ef$, where $\lambda = qa^2/bcd$, to obtain \eqref{e-10p9-6}.
\end{proof}

Using an analytic continuation argument, we can write \eqref{e-10p9-6} with sums over an $n$-simplex.

\begin{Theorem}[an $A_n$ elliptic Bailey $_{10}\phi_9$ transformation] \label{th:e-10p9-6-tetra}
Let $\lambda = qa^2/bcd$. Then
\begin{gather}
\sum\limits_{\substack{0\leq \sumk\leq N \\ k_1,k_2, \dots, k_n\geq 0} }
 \Bigg( \ellipticvandermonde{x}{k}{n} \sqprod n \frac{\ellipticqrfac{f_s\xover{x}}{k_r} }{\ellipticqrfac{q\xover{x}}{k_r} } \notag \\
\qquad\quad{} \times \smallprod n \frac{\elliptictheta{a x_r q^{k_r+\sumk}}}{\elliptictheta{a x_r}}
 \smallprod n \frac{\ellipticqrfac{a x_r}{\sumk}} {\ellipticqrfac{a x_rq/f_r}{\sumk}} \notag\\
\qquad\quad{} \times \smallprod n \frac{\ellipticqrfac{cx_r, dx_r, \lambda ax_r q^{N+1}/ef_1\cdots f_n}{k_r}}
 {\ellipticqrfac{a x_rq/b, a x_rq/e, ax_rq^{1+N}}{k_r}} \notag\\
\qquad\quad{} \times \frac{\ellipticqrfac{b, e, q^{-N}}{\sumk}}
{\ellipticqrfac{aq/c, aq/d, ef_1\cdots f_n q^{-N}/\lambda}{\sumk}}
q^{\sum\limits_{r=1}^n r k_r} \Bigg) \notag \\
\qquad{} = \frac{\ellipticqrfac{\lambda q/e, \lambda q/f_1\cdots f_n}{N}}{\ellipticqrfac{ \lambda q, \lambda q/ef_1\cdots f_n}{N}}
\smallprod n \frac{\ellipticqrfac{a x_rq, a x_r q/ef_r}{N}} {\ellipticqrfac{a x_rq/e, a x_rq/f_r}{N}}\notag \\
\qquad\quad{} \times \sum\limits_{\substack{0\leq \sumk\leq N \\ k_1,k_2, \dots, k_n\geq 0} }
\Bigg( \ellipticvandermonde{x}{k}{n} \sqprod n \frac{\ellipticqrfac{f_s\xover{x}}{k_r} }{\ellipticqrfac{q\xover{x}}{k_r} } \notag \\
\qquad\quad{} \times \frac{\elliptictheta{\lambda q^{2\sumk}}}{\elliptictheta{\lambda }}
\frac{\ellipticqrfac{\lambda ,\lambda c/a, \lambda d/a, e, q^{-N}}{\sumk}}
{\ellipticqrfac{\lambda q/f_1\cdots f_n, aq/c, aq/d, \lambda q/e, \lambda q^{1+N}}{\sumk}}\nonumber\\
\qquad\quad{} \times
\smallprod n \frac{\ellipticqrfac{\lambda b/ax_r}{\sumk} \ellipticqrfac{eq^{-N}/ax_r}{\sumk-k_r}}{\ellipticqrfac{\lambda b/ax_r}{\sumk-k_r} \ellipticqrfac{ef_rq^{-N}/ax_r}{\sumk}}\notag \\
\qquad\quad {} \times \smallprod n \frac{\ellipticqrfac{\lambda a x_r q^{N+1}/ef_1\cdots f_n}{k_r}}
{\ellipticqrfac{ax_rq/b}{k_r} }\cdot q^{\sum\limits_{r=1}^n r k_r} \Bigg). \label{e-10p9-6-tetra}
\end{gather}
\end{Theorem}
\begin{proof} Denote the left-hand side minus the right-hand side of the transformation in \eqref{e-10p9-6-tetra} by $F(f_1,\dots,f_n)$. Now $F$ is meromorphic in each of the variables $f_1,\dots,f_n$ in the domain $0<|f_s|<\infty$, for $1\le s\le n$. Using \eqref{GR11.2.55a}, it can be easily checked that~$F$ is periodic in each~$f_s$, i.e.,
\begin{gather*}
F(f_1,\dots,pf_s,\dots,f_n)=F(f_1,\dots,f_n),
\end{gather*}
for $s=1,\dots,n$. For technical reasons we shall first assume that $p$, $q$ are chosen such that $p^mq^n$ are distinct for all integers~$m$ and~$n$.

We will first demonstrate the existence of a convergent sequence of distinct points, such that~$F$ is zero when~$f_1$ is equal to any term of that sequence. This will imply that~$F$ is zero for any $f_1$ where $F$ is defined. It follows from the $f=q^{-N}$ case of Theorem~\ref{th:e-10p9-6}
that~$F$ is zero for $(f_1,\dots,f_n)=\big(q^{-N_1},\dots,q^{-N_n}\big)$, where $N_1,\dots,N_n$ are nonnegative integers. Now for every $r=0, 1, 2, \dots$, there is an integer~$m_r$ such that if $z_r = p^{m_r}q^{-r}$, then $p\le |z_r| \le 1$. By periodicity of $F$, $F$ must be $0$ when $f_1=z_r$, for every~$r$. Since there is an infinite number of $z_r$ in the annulus $p\le |z|\le 1$, the $z_r$ must have a limit point. Thus there is a convergent infinite subsequen\-ce~$(z_{r_k})$ of distinct points, such that~$F$ is~$0$ on this subsequence, and therefore~$F$ must be zero for any~$f_1$, as required.

By iterating this argument for $f_2,\dots,f_n$, we conclude that $F$ is identically zero for $f_1,\dots,f_n$. By analytic continuation this is extended to the degenerate cases where the $p^mq^n$ are not all different (for various~$m$,~$n$), as long as those choices for~$p$,~$q$ leave~$F$ well-defined.
\end{proof}

Some consequences of Theorem~\ref{th:e-10p9-6} are noted in Section~\ref{sec:piszero}. Next we have an $A_n$ elliptic well-poised $\B^{(2)} \to\B^{(2)}$ Bailey lemma.

\begin{Theorem}[an elliptic ($\B^{(2)}\to\B^{(2)}$) WP Bailey lemma] \label{th:e-WP-BaileyLemma-B2-B2}
Suppose $\alpha_\N(a,b)$ and $\beta_\N(a,b)$ form a WP Bailey pair with respect to the matrix~$\B^{(2)}$. Let ${\alpha}^{\prime}_\N(a,b)$ and ${\beta}^{\prime}_\N(a,b)$ be defined as follows
\begin{gather*}
{\alpha}^{\prime}_\N(a,b) :=
\frac{\ellipticqrfac{\rho_1 }{\sumN}}{\ellipticqrfac{aq/\rho_2}{\sumN}}
 \smallprod n \frac{\ellipticqrfac{\rho_2x_r}{N_r}}
{\ellipticqrfac{ax_rq/\rho_1}{N_r}} \cdot
\left(\frac{aq}{\rho_1\rho_2}\right)^{\sumN} \alpha_\N(a, b\rho_1\rho_2/aq),\\
{\beta}^{\prime}_\N (a,b) := \frac{\ellipticqrfac{b\rho_2/a }{\sumN}}{\ellipticqrfac{aq/\rho_2}{\sumN}}
 \smallprod n \frac{\ellipticqrfac{b\rho_1q^{\sumN-N_r}/ax_r}{N_r}}
{\ellipticqrfac{ax_rq/\rho_1}{N_r}}\notag \\
 \hphantom{{\beta}^{\prime}_\N (a,b) :=}{} \times \multsum{k}{N}{r} \Bigg(
\frac{ \ellipticqrfac{\rho_1}{\sumk} }{\ellipticqrfac{b\rho_2/a}{\sumk}}
\smallprod n \frac{\ellipticqrfac{\rho_2x_r}{k_r}
\ellipticqrfac{b\rho_1q^{\sumN}/ax_r}{\sumk-k_r} }
{\ellipticqrfac{b\rho_1q^{\sumN-N_r}/ax_r}{\sumk}} \notag \\
\hphantom{{\beta}^{\prime}_\N (a,b) :=}{}
\times \frac{\elliptictheta{ b\rho_1\rho_2q^{2\sumk}/{aq}}}{\elliptictheta{ b\rho_1\rho_2/{aq}}}
\frac{\ellipticqrfac{b}{\sumN+\sumk}}{\ellipticqrfac{b\rho_1\rho_2/a}{\sumN+\sumk}} \notag\\
\hphantom{{\beta}^{\prime}_\N (a,b) :=}{}
\times \frac{\ellipticqrfac{aq/\rho_1\rho_2}{\sumN-\sumk} }
{\sqprod n \ellipticqrfac{q^{1+k_r-k_s}\xover x }{N_r-k_r}}
\left(\frac{aq}{\rho_1\rho_2}\right)^{\sumk} \beta_\k(a,b\rho_1\rho_2/aq)\Bigg).
\end{gather*}
Then ${\alpha}^{\prime}_\N(a,b)$ and ${\beta}^{\prime}_\N(a,b)$ also form a WP Bailey pair with respect to $\B^{(2)}$.
\end{Theorem}

\begin{proof}
The proof is analogous to the proof of Theorem~\ref{th:e-WP-BaileyLemma-B1}. However, this time we sum the inner sum using \eqref{e-8p7-5}. We use the following substitutions: $x_r\mapsto x_rq^{j_r}$ and $N_r\mapsto N_r-j_r$ for $r=1, 2, \dots, n$, $a\mapsto b\rho_1\rho_2q^{2\sumj}/aq$,
$b\mapsto \rho_1q^{\sumj}$, $c\mapsto bq^{\sumN+\sumj}$, $d\mapsto b\rho_1\rho_2q^{\sumj}/a^2q$. The rest of the proof is similar to the earlier one.
\end{proof}

On applying the $\B^{(2)} \to\B^{(2)}$ elliptic Bailey lemma in Theorem~\ref{th:e-WP-BaileyLemma-B2-B2} to the WP Bailey pair in Theorem~\ref{th:e-Bailey-Pair-B2.2}, we again obtain Theorem~\ref{th:e-10p9-6}.

Next, we present another Bressoud matrix related to \eqref{e-8p7-5}.

\begin{Definition}[an $A_n$ Elliptic Bressoud matrix] We define the matrix $\B^{(3)}$ with entries $B^{(3)}_{\k\j}(a,b)$ as
\begin{gather}
B^{(3)}_{\k\j}(a,b) := \frac{ \smallprod n \ellipticqrfac{bx_r}{j_r+\sumk} \ellipticqrfac{bx_r q^{j_r-\sumj}/a}{k_r-j_r} }
{ \ellipticqrfac{aq}{\sumk+\sumj} \sqprod n \ellipticqrfac{q^{1+j_r-j_s}\xover x }{k_r-j_r} }. \label{e-B3}
 \end{gather}
\end{Definition}
\begin{Remark*} When $p=0$ and further $b=0$, this reduces to a multivariable Bailey transform matrix given by Milne~\cite[Definition~8.24]{Milne1997}. This matrix and its inverse were used in \cite{GB1995} to derive a~result related to~\eqref{6p5-BS1} below.
\end{Remark*}

\begin{Theorem}[an elliptic WP Bailey pair with respect to $\B^{(3)}$]\label{th:e-Bailey-Pair-B3.2} The two sequences
\begin{gather*}
\alpha_\k(a,b) := \frac{\elliptictheta{aq^{2\sumk}}}{\elliptictheta{a}} \frac{ \ellipticqrfac{a,c, d}{\sumk} }{\ellipticqrfac{aq/c, aq/d}{\sumk}
\sqprod n \ellipticqrfac{q\xover x}{k_r}}\notag\\
\hphantom{\alpha_\k(a,b) :=}{}
\times \smallprod n \frac{\ellipticqrfac{a^2q/bcdx_r}{\sumk}}
{\ellipticqrfac{a^2q/bcdx_r}{\sumk-k_r}\ellipticqrfac{bcdx_r/a}{k_r}}\cdot
\left( \frac{b}{a}\right)^{\sumk}
q^{-\sum\limits_{r<s} k_rk_s}
\smallprod n x_r^{k_r}, \\
 \intertext{and}
\beta_\k(a,b) := \frac{ \ellipticqrfac{aq/cd}{\sumk} }{\ellipticqrfac{aq/c, aq/d}{\sumk}}
\sqprod n\frac{1}{\ellipticqrfac{q\xover x}{k_r}} \smallprod n \frac{\ellipticqrfac{bx_r}{\sumk}\ellipticqrfac{bcx_r/a, bdx_r/a}{k_r}}
{\ellipticqrfac{bcdx_r/a}{k_r}},
\end{gather*}
form a WP-Bailey pair with respect to $\B^{(3)}$.
\end{Theorem}
\begin{proof} The proof is analogous to that of Theorem~\ref{th:e-Bailey-Pair-B1.2a}. The proof requires the $d\mapsto qa^2/bcd$ case of~\eqref{e-8p7-5}, after interchanging~$b$ and~$d$. We verify that $\alpha_\j(a,b)$ and $\beta_\N(a,b)$ satisfy the defining condition of a WP Bailey pair with respect to $\B^{(3)}$.
\end{proof}

\begin{Corollary}\label{e-Bailey-Pair-B3-1} The two sequences
\begin{gather*}
\alpha_\k(a,b) := \frac{\elliptictheta{aq^{2\sumk}}}{\elliptictheta{a}}
 \ellipticqrfac{a}{\sumk}
\sqprod n\frac{1}{\ellipticqrfac{q\xover x}{k_r}} \notag\\
\hphantom{\alpha_\k(a,b) :=}{} \times \smallprod n \frac {\ellipticqrfac{a/bx_r}{\sumk}}{\ellipticqrfac{a/bx_r} {\sumk-k_r}
\ellipticqrfac{bx_rq}{k_r}} \cdot \left( \frac{b}{a}\right)^{\sumk} q^{-\sum\limits_{r<s} k_rk_s} \smallprod n x_r^{k_r}, \\
\intertext{and}
\beta_\k(a,b) := \smallprod n \delta_{k_r, 0} ,
\end{gather*}
form a WP-Bailey pair with respect to $\B^{(3)}$.
\end{Corollary}
\begin{proof} Take $d=aq/c$ in Theorem~\ref{th:e-Bailey-Pair-B3.2} to obtain this unit Bailey pair.
\end{proof}
As before, we can derive a formula for the inverse of $\B^{(3)}$ using this unit Bailey pair.
\begin{Corollary}[inverse of $\B^{(3)}$]\label{cor:e-B3-inverse}
Let $\B^{(3)}= \big(B^{(3)}_{\k\j}(a,b)\big)$ be defined by~\eqref{e-B3}. Then the entries of its inverse are given by
\begin{gather}
\big(B^{(3)}(a,b)\big)_{\k \j}^{-1} =
\frac{\elliptictheta{aq^{2\sumk}}}{\elliptictheta{a}}
\smallprod n \frac{\elliptictheta{bx_rq^{j_r+\sumj}}}{\elliptictheta{bx_r}}
\cdot \left( \frac{b}{a}\right)^{\sumk-\sumj}
q^{\sum\limits_{r<s}(j_rj_s- k_rk_s)}
\smallprod n x_r^{k_r-j_r}\nonumber\\
\hphantom{\big(B^{(3)}(a,b)\big)_{\k \j}^{-1} =}{}
\times \frac{\ellipticqrfac{a}{\sumk+\sumj} \smallprod n \ellipticqrfac{aq^{\sumk-k_r}/bx_r}{k_r-j_r}}
{\sqprod n \ellipticqrfac{q^{1+j_r-j_s}\xover x }{k_r-j_r} \smallprod n \ellipticqrfac{bx_rq}{k_r+\sumj}} .\label{e-B3-inverse}
\end{gather}
\end{Corollary}
\begin{proof}
We first write Corollary~\ref{e-Bailey-Pair-B3-1} in the form
 \begin{gather*}\label{n-B3-inverse-1}
 \multsum{j}{N}{r} B^{(3)}_{\N\j} (a,b) \alpha_\j = \smallprod n \delta_{N_r, 0} ,
\end{gather*}
where $\alpha_\j$ is as in Corollary~\ref{e-Bailey-Pair-B3-1}, and the entries of $\B^{(3)}$ is given by \eqref{e-B3}. We now replace $\N$ by $\N-\K$, and shift the index by replacing $\j$ by $\j-\K$. The index of the sum now runs from $K_r\leq j_r \leq N_r$, for $r=1, 2, \dots, n$. Next we take $x_r\mapsto x_rq^{K_r}$, $a\mapsto aq^{2\sumK}$ and $b\mapsto bq^{\sumK}$ to obtain, after some simplification, the sum
 \begin{gather}\label{n-B3-inverse-2}
 \sum\limits_{\substack{{K_r\le j_r \le N_r} \\
{r =1,2,\dots, n}}} B^{(3)}_{\N\j} (a,b) \big(B^{(3)}(a,b)\big)^{-1}_{\j\K}(a,b) = \smallprod n \delta_{N_r, K_r},
\end{gather}
where $\big(B^{(3)}(a,b)\big)^{-1}_{\j\K}$ is given by \eqref{e-B3-inverse} with indices relabeled as $(\k, \j)\mapsto (\j, \K)$.
\end{proof}

Observe that the entries of $(\B^{(3)}(a,b))^{-1}$ consist of the entries of $\B^{(2)}(b,a)$ multiplied by some additional factors, which can be separated into factors containing either terms with index~$\j$ or with index~$\k$. This can help us find the inverse of~$\B^{(2)}$.
\begin{Corollary}[inverse of $\B^{(2)}$]\label{cor:e-B2-inverse} Let $\B^{(2)}= (B^{(2)}_{\k\j}(a,b))$ be defined by~\eqref{e-B2}. Then the entries of its inverse are given by
\begin{gather*}
\big(B^{(2)}(a,b)\big)_{\k \j}^{-1} =
\smallprod n \frac{\elliptictheta{ax_rq^{k_r+\sumk}}}{\elliptictheta{ax_r}}
\cdot \frac{\elliptictheta{bq^{2\sumj}}}{\elliptictheta{b}} \left( \frac{b}{a}\right)^{\sumk-\sumj}
q^{\sum\limits_{r<s}(k_rk_s- j_rj_s)} \smallprod n x_r^{j_r-k_r}\\
 \hphantom{\big(B^{(2)}(a,b)\big)_{\k \j}^{-1} =}{} \times
\frac{ \smallprod n
 \ellipticqrfac{ax_r}{j_r+\sumk} \ellipticqrfac{ax_r q^{j_r-\sumj}/b}{k_r-j_r} }
{ \ellipticqrfac{bq}{\sumk+\sumj} \sqprod n \ellipticqrfac{q^{1+j_r-j_s}\xover x }{k_r-j_r} }.
\end{gather*}
\end{Corollary}
\begin{proof} We can write the sum \eqref{n-B3-inverse-2} in the form
 \begin{gather*}
 \sum\limits_{\substack{{K_r\le j_r \le N_r} \\
{r =1,2,\dots, n}}} \Bigg( B^{(3)}_{\N\j} (a,b) B^{(2)}_{\j\K}(b,a)
 \frac{\elliptictheta{aq^{2\sumj}}}{\elliptictheta{a}}
\smallprod n \frac{\elliptictheta{bx_rq^{K_r+\sumK}}}{\elliptictheta{bx_r}} \\
\qquad{} \times \left( \frac{b}{a}\right)^{\sumj-\sumK} q^{\sum\limits_{r<s}(K_r K_s- j_r j_s)}
\smallprod n x_r^{j_r-K_r} \Bigg)= \smallprod n \delta_{N_r, K_r}.
\end{gather*}
Interchanging $a$ and $b$, this expression can be written as
\begin{gather*}
\smallprod n \frac{\elliptictheta{ax_rq^{K_r+\sumK}}}{\elliptictheta{ax_rq^{N_r+\sumN}}}
\cdot \left( \frac{b}{a}\right)^{\sumK-\sumN} q^{\sum\limits_{r<s}(K_r K_s- N_r N_s)}
\smallprod n x_r^{N_r-K_r} \\
\qquad{} \times \sum\limits_{\substack{{K_r\le j_r \le N_r} \cr
{r =1,2,\dots, n}}} \Bigg( B^{(3)}_{\N\j} (b,a) B^{(2)}_{\j\K}(a,b)
 \frac{\elliptictheta{bq^{2\sumj}}}{\elliptictheta{b}}
\smallprod n \frac{\elliptictheta{ax_rq^{N_r+\sumN}}}{\elliptictheta{ax_r}} \\
\qquad{} \times \left( \frac{b}{a}\right)^{\sumN-\sumj} q^{\sum\limits_{r<s}(N_r N_s- j_r j_s)}
\smallprod n x_r^{j_r-N_r} \Bigg) = \smallprod n \delta_{N_r, K_r}.
\end{gather*}
Note that when $\N=\K,$ the (non-zero) factors outside the sum on the left-hand side reduce to~$1$, and thus cancel. We can now read off the entries of the inverse of $\B^{(2)}(a,b)$ from this expression.
\end{proof}

Consider the inverse relation \eqref{multivariable-inverse-relation} where $\B=\B^{(3)}$, and $\alpha_k$ and $\beta_k$ are defined as in Theorem~\ref{th:e-Bailey-Pair-B3.2} and $(B^{(3)}(a,b))^{-1}_{\k\j} $ is given by~\eqref{e-B3-inverse}. After using Lemma~\ref{th:e-magiclemma2} and canceling some products, we make the substitutions $a\mapsto qa^2/bcd$, $b\mapsto a$, $c\mapsto aq/bd$ and $d\mapsto aq/bc$. Finally, we replace~$c$ by~$a^2q^{1+\sumN}/bcd$ to again obtain \eqref{e-8p7-1}. Thus \eqref{e-8p7-1} and~\eqref{e-8p7-5} are inverse relations. Indeed, Rosengren and the second author originally obtained~\eqref{e-8p7-5} by taking the inverse relation of \eqref{e-8p7-1} using a matrix inversion equivalent to $\B^{(2)}$.

\begin{Theorem}[an elliptic ($\B^{(3)}\to\B^{(3)}$) WP Bailey lemma] \label{th:e-WP-BaileyLemma-B3-B3}
Suppose $\alpha_\N(a,b)$ and $\beta_\N(a,b)$ form a WP Bailey pair with respect to the matrix $\B^{(3)}$. Let ${\alpha}^{\prime}_\N(a,b)$ and ${\beta}^{\prime}_\N(a,b)$ be defined as follows
\begin{gather*}
{\alpha}^{\prime}_\N(a,b) := \frac{\ellipticqrfac{\rho_1, \rho_2 }{\sumN}}
{\ellipticqrfac{aq/\rho_1, aq/\rho_2}{\sumN}}
\left(\frac{aq}{\rho_1\rho_2}\right)^{\sumN} \alpha_\N(a, b\rho_1\rho_2/aq),\\ 
{\beta}^{\prime}_\N (a,b) := \frac{ \smallprod n \ellipticqrfac{b\rho_1x_r/a, b\rho_2x_r/a}{N_r} }
{ \ellipticqrfac{aq/\rho_1, aq/\rho_2}{\sumN}} \notag \\
\hphantom{{\beta}^{\prime}_\N (a,b) :=}{}
\times \multsum{k}{N}{r} \Bigg(
\frac{ \ellipticqrfac{\rho_1, \rho_2}{\sumk} }
{\smallprod n \ellipticqrfac{b\rho_1x_r/a, b\rho_2x_r/a}{k_r}}
\smallprod n \frac{\elliptictheta{ b\rho_1\rho_2x_rq^{k_r+\sumk}/{aq}}}
{\elliptictheta{ b\rho_1\rho_2x_r/{aq}}}\\
\hphantom{{\beta}^{\prime}_\N (a,b) :=}{}\times \frac{\ellipticqrfac{aq/\rho_1\rho_2}{\sumN-\sumk}}
{\sqprod n \ellipticqrfac{q^{1+k_r-k_s}\xover x }{N_r-k_r}}
 \smallprod n \frac{\ellipticqrfac{bx_r}{k_r+\sumN}}{\ellipticqrfac{b\rho_1\rho_2x_r/a}{\sumk+N_r}} \notag \\
\hphantom{{\beta}^{\prime}_\N (a,b) :=}{}\times
\left(\frac{aq}{\rho_1\rho_2}\right)^{\sumk} \beta_\k(a,b\rho_1\rho_2/aq) \Bigg). 
\end{gather*}
Then ${\alpha}^{\prime}_\N(a,b)$ and ${\beta}^{\prime}_\N(a,b)$ also form a WP Bailey pair with respect to $\B^{(3)}$.
\end{Theorem}

\begin{proof}
The proof is analogous to that of Theorem~\ref{th:e-WP-BaileyLemma-B1}. Again we use \eqref{e-8p7-1}, but this time with the following substitutions: $x_r\mapsto x_rq^{j_r}$ and $N_r\mapsto N_r-j_r$ for $r=1, 2, \dots, n$, $a\mapsto b\rho_1\rho_2q^{\sumj}/aq$, $b\mapsto \rho_1q^{\sumj}$, $c\mapsto \rho_2q^{\sumj}$, $d\mapsto b\rho_1\rho_2q^{-\sumj}/a^2q$. The rest of the proof is similar.
\end{proof}

On applying the $\B^{(3)} \to\B^{(3)}$ elliptic Bailey lemma in Theorem~\ref{th:e-WP-BaileyLemma-B3-B3} to the WP Bailey pair in Theorem~\ref{th:e-Bailey-Pair-B3.2}, we again obtain Theorem~\ref{th:e-10p9-6}.

We have now seen three multivariable Bressoud matrices, and various WP Bailey lemmas connecting WP Bailey pairs with respect to them. As a result we obtained a new $_{10}\phi_9$ elliptic Bailey transformation. In the next section, we suspend our study of WP Bailey pairs and lemmas, to record special cases of this new transformation formula.

\section[Special cases: new $A_n$ Watson transformations and related identities]{Special cases: new $\boldsymbol{A_n}$ Watson transformations\\ and related identities}\label{sec:piszero}

We now consider extensions of Watson's transformation that follow from the $A_n$ elliptic Bailey $_{10}\phi_9$ transformation formula \eqref{e-10p9-6}. Previously, multiple series extensions of Watson's transformations have been obtained by, for example, Milne~\cite{Milne1988, Milne1994, Milne2000}, Milne and Lilly~\cite{ML1995}, Milne and Newcomb~\cite{MN1996}, Coskun~\cite{HC2008} and by the authors, see \cite{GB1999a} and \cite{BS1998}. Some interesting applications of one of these transformations to the theory of affine Lie algebras appear in Bartlett and Warnaar~\cite{BW2015} and Griffin, Ono and Warnaar~\cite{GOW2016}.
Below we present some new $A_n$ Watson transformation formulas and some further special cases.

Multiple series extensions of Watson's transformation formula \cite[equation~(2.5.1)]{GR90} can be obtained from the $p=0$ case of \eqref{e-10p9-6} in multiple ways. We can take the limit as $b$, $c$, or $d$ goes to infinity. Alternatively, we can consider the equivalent formulation obtained by replacing $e$ or $f$ by $\lambda a q^{1+\sumN}/ef$ and then take the limits as one of $b$, $c$, or $d$ go to infinity. We can also interchange the role of $\lambda$ and $a$ and then take limits as above. Many of these limits give rise to the same formula, depending on the symmetry of the various parameters.

\begin{Theorem}[an $A_n$ Watson transformation] \label{th:an-watson-BS-1} We have
\begin{gather*}
\multsum{k}{N}{r} \Bigg( \vandermonde{x}{k}{n}
\sqprod n \frac{\qrfac{q^{-N_s}\xover{x}}{k_r} }{\qrfac{q\xover{x}}{k_r} } \notag \\
\qquad\quad{} \times \smallprod n \frac{1-{a x_r q^{k_r+\sumk}}}{1-{a x_r}}
\frac{\qrfac{a x_r}{\sumk}\qrfac{bx_r, cx_r, ex_r}{k_r}}
 {\qrfac{a x_rq^{1+N_r}}{\sumk}\qrfac{a x_rq/d}{k_r}} \smallprod n x_r^{-k_r}\\
 \qquad\quad {} \times \frac{\qrfac{d}{\sumk}} {\qrfac{aq/b, aq/c, aq/e}{\sumk}}
\left( \frac{a^2q^{2+\sumN}}{bcde} \right)^{\sumk}
q^{\sum\limits_{r<s}k_rk_s+\sum\limits_{r=1}^n (r-1)k_r} \Bigg) \notag \\
\qquad{} = \frac{\qrfac{a q/de}{\sumN}}{\qrfac{ a q/e}{\sumN}} \smallprod n \frac{\qrfac{a x_rq}{N_r}}
{\qrfac{a x_rq/d}{N_r}} \notag \\
\qquad\quad{} \times \multsum{k}{N}{r} \Bigg( \vandermonde{x}{k}{n}
\sqprod n \frac{\qrfac{q^{-N_s}\xover{x}}{k_r} }{\qrfac{q\xover{x}}{k_r} } \notag \\
\qquad\quad{} \times \frac{\qrfac{ d}{\sumk}\smallprod n \qrfac{ex_r, a q^{1+\sumk-k_r}/bcx_r }{k_r} }
{\qrfac{ aq/b, aq/c, deq^{-\sumN}/a}{\sumk}} q^{\sum\limits_{r=1}^n rk_r} \Bigg).
\end{gather*}
\end{Theorem}

\begin{proof} We take $p=0$ in \eqref{e-10p9-6}, and replace $e$ by $\lambda aq^{1+\sumN}/ef$. Now we take the limit $b\to \infty$ and replace $d$ by $b $ and $f$ by $d$.
\end{proof}

\begin{Remark*} By applying a standard polynomial argument to Theorem~\ref{th:an-watson-BS-1}, we can obtain an equivalent transformation formula, where both sums are summed over an $n$-simplex. (This result could alternatively be obtained from Theorem~\ref{th:e-10p9-6-tetra} by applying a similar limit and substitution as that used in the proof of Theorem~\ref{th:an-watson-BS-1}.) That is, the summation indices on both sides range in the region $0\leq\sumk\leq N$, where $N$ is a non-negative integer, and $k_r\geq 0$, for $r=1, 2, \dots, n$. This remark applies to all the results of this section. We do not write down these results explicitly. For an example of such a calculation, see the proof of \cite[Theorem~2.4]{Milne1997} or \cite[Theorem~3.7]{BS1998}.
\end{Remark*}

If we take $p=0$ in \eqref{e-10p9-6}, and take the limit $b\to \infty$, we obtain the $A_n$ Watson transformation formula~\cite[Theorem~4.3]{BS1998}. If we take $p=0$ in \eqref{e-10p9-6}, take the limit $d\to \infty$, we obtain an $A_n$ Watson transformation formula due to Milne, see \cite[Theorem~2.1]{Milne2000}. Finally, we take $p=0$ in~\eqref{e-10p9-6}, and replace $e$ by $\lambda aq^{1+\sumN}/ef$, and take the limit $d\to \infty$ to obtain another $A_n$ Watson transformation due to Milne, see \cite[Theorem~A.3]{MN1996}.

Next, we consider the formula obtained from \eqref{e-10p9-6} by first replacing $\lambda$ by $qa^2/bcd$, and then taking (simultaneously) $a\mapsto qa^2/bcd$, $b\mapsto aq/cd$, $c\mapsto aq/bd$ and $d\mapsto aq/bc$. In the resulting formula, we take $\lambda = qa^2/bcd$ and use the relations $\lambda b/a = aq/cd$, $\lambda c/a = aq/bd$ and $\lambda d/a = aq/bc$. In this manner we can write the right-hand side of the series \eqref{e-10p9-6} with special parameter $a$ and the left-hand side with special parameter $\lambda$. Now we take $p=0$ to obtain the following $A_n$ Bailey $_{10}\phi_9$ transformation formula. This is equivalent to the $p=0$ case of \eqref{e-10p9-6}, but with $\lambda$ and $a$ interchanged.
Let $\lambda = qa^2/bcd$. Then
\begin{gather}
\multsum{k}{N}{r} \Bigg( \vandermonde{x}{k}{n}
\sqprod n \frac{\qrfac{q^{-N_s}\xover{x}}{k_r} }{\qrfac{q\xover{x}}{k_r} } \notag \\
\qquad\quad{} \times \frac{{1-aq^{2\sumk}}}{{1-a}}
\frac{\qrfac{a, c, d, e, f}{\sumk}}{\qrfac{aq^{\sumN+1}, aq/c, aq/d, aq/e, aq/f}{\sumk}} \nonumber\\
\qquad\quad{} \times
\smallprod n \frac{\qrfac{b/x_r}{\sumk}\qrfac{\lambda a x_r q^{\sumN+1}/ef}{k_r}
\qrfac{ef/\lambda x_r}{\sumk-k_r}}{\qrfac{b/x_r}{\sumk-k_r}
\qrfac{ax_rq/b}{k_r}\qrfac{efq^{-N_r}/\lambda x_r}{\sumk}}
\cdot q^{\sum\limits_{r=1}^n rk_r} \Bigg)\notag \\
\qquad{} = \frac{\qrfac{aq, aq/ef}{\sumN}}{\qrfac{aq/e, aq/f}{\sumN}}
\smallprod n \frac{\qrfac{\lambda x_rq/e, \lambda x_r q/f}{N_r}}
{\qrfac{\lambda x_rq, \lambda x_rq/ef}{N_r}}\notag \\
\qquad\quad{} \times \multsum{k}{N}{r} \Bigg(\vandermonde{x}{k}{n}
\sqprod n \frac{\qrfac{q^{-N_s}\xover{x}}{k_r} }{\qrfac{q\xover{x}}{k_r} }
\smallprod n \frac{{1-\lambda x_r q^{k_r+\sumk}}}{{1-\lambda x_r}} \notag \\
\qquad\quad{} \times \smallprod n \frac{\qrfac{\lambda x_r}{\sumk}
\qrfac{\lambda cx_r/a, \lambda dx_r/a, \lambda ax_r q^{\sumN+1}/ef}{k_r}}
 {\qrfac{\lambda x_rq^{1+N_r}}{\sumk}
 \qrfac{\lambda x_rq/e, \lambda x_rq/f, ax_rq/b}{k_r}} \notag\\
\qquad\quad{} \times \frac{\qrfac{ e, f, \lambda b/a}{\sumk}} {\qrfac{aq/c, aq/d, efq^{-\sumN}/a}{\sumk}}
q^{\sum\limits_{r=1}^n rk_r} \Bigg). \label{10p9-7-form2}
\end{gather}
Three new $A_n$ Watson transformations follow from this $A_n$ Bailey transformation formula.

\begin{Theorem}[an $A_n$ Watson transformation] We have
\begin{gather*}
\multsum{k}{N}{r} \Bigg( \vandermonde{x}{k}{n}
\sqprod n \frac{\qrfac{q^{-N_s}\xover{x}}{k_r} }{\qrfac{q\xover{x}}{k_r} } \notag \\
\qquad\quad{} \times \frac{1-aq^{2\sumk}}{1-a}
 \frac{\qrfac{a, c, d, e}{\sumk}}{\qrfac{aq^{\sumN+1}, aq/c, aq/d, aq/e}{\sumk}}
 \smallprod n x_r^{k_r}\notag \\
\qquad\quad{} \times \smallprod n \frac{\qrfac{b/x_r}{\sumk}}{\qrfac{b/x_r}{\sumk-k_r} \qrfac{ax_rq/b}{k_r}}
\cdot \left(\frac{a^2q^{2+\sumN}}{bcde}\right)^{\sumk}
q^{- \sum\limits_{r<s}k_rk_s+\sum\limits_{r=1}^n (r-1)k_r }
 \Bigg)\notag \\
\qquad {}= \frac{\qrfac{aq, aq/de}{\sumN}}{\qrfac{aq/d, aq/e}{\sumN}} \multsum{k}{N}{r} \Bigg(\vandermonde{x}{k}{n}
\sqprod n \frac{\qrfac{q^{-N_s}\xover{x}}{k_r} }{\qrfac{q\xover{x}}{k_r} } \notag \\
\qquad\quad{} \times \smallprod n \frac{\qrfac{ax_rq/bc}{k_r}}
{ \qrfac{ax_rq/b}{k_r}} \cdot \frac{\qrfac{ d, e}{\sumk}}{\qrfac{aq/c, deq^{-\sumN}/a}{\sumk} }
q^{\sum\limits_{r=1}^n rk_r } \Bigg). 
\end{gather*}
\end{Theorem}
\begin{proof}In \eqref{10p9-7-form2}, we take the limit as $d\to \infty$ and replace $f$ by $d$ in the result.
\end{proof}

\begin{Theorem}[an $A_n$ Watson transformation] \label{th:an-watson-BS-3} We have
\begin{gather}
\multsum{k}{N}{r} \Bigg( \vandermonde{x}{k}{n}
\sqprod n \frac{\qrfac{q^{-N_s}\xover{x}}{k_r} }{\qrfac{q\xover{x}}{k_r} } \notag \\
\qquad\quad{} \times\smallprod n \frac{\qrfac{ex_r}{k_r} \qrfac{ax_rq^{1+\sumN}/e}{\sumk-k_r}}
{\qrfac{ax_rq^{1+\sumN-N_r}/e}{\sumk}}\cdot
q^{\sum\limits_{r<s}k_rk_s+\sum\limits_{r=1}^n (r-1)k_r } \smallprod n x_r^{-k_r}\notag \\
\qquad\quad{} \times \frac{1-aq^{2\sumk}}{1-a}
\frac{\qrfac{a, b, c, d}{\sumk}}{\qrfac{aq^{\sumN+1}, aq/b, aq/c, aq/d}{\sumk}}
\left(\frac{a^2q^{2+\sumN}}{bcde}\right)^{\sumk} \Bigg)\notag \\
\qquad {}= \frac{\qrfac{aq}{\sumN}}{\qrfac{aq/d}{\sumN}}
\smallprod n \frac{\qrfac{aq^{1+\sumN-N_r}/dex_r}{N_r}}{\qrfac{aq^{1+\sumN-N_r}/ex_r}{N_r}}\notag \\
\qquad\quad{} \times \multsum{k}{N}{r} \Bigg( \vandermonde{x}{k}{n}
\sqprod n \frac{\qrfac{q^{-N_s}\xover{x}}{k_r} }{\qrfac{q\xover{x}}{k_r} } \notag \\
\qquad\quad{} \times \smallprod n \frac{\qrfac{ex_r}{k_r}}
{ \qrfac{dex_rq^{-\sumN}/a}{k_r}} \cdot \frac{\qrfac{ d, aq/bc}{\sumk}}{\qrfac{aq/b, aq/c}{\sumk} }
q^{\sum\limits_{r=1}^n rk_r } \Bigg). \label{an-watson-BS-3}
\end{gather}
\end{Theorem}
\begin{Remark*} If we take $c=1$ in \eqref{an-watson-BS-3}, the sum on the left-hand side reduces to~$1$, and we obtain an equivalent formulation of Milne's balanced $_3\phi_2$ sum~\cite[Theorem~4.1]{Milne1997}.
\end{Remark*}
\begin{proof}
In \eqref{10p9-7-form2}, we replace $e\mapsto \lambda a q^{1+\sumN}/ef$ and take $b\to \infty$, and take $f\mapsto d$ and $d\mapsto b$ in the resulting identity to obtain \eqref{an-watson-BS-3}.
\end{proof}

A summation theorem follows immediately from Theorem~\ref{th:an-watson-BS-3}.
\begin{Theorem}[an $A_n$ very-well-poised $_6\phi_5$ summation] We have
\begin{gather}
\multsum{k}{N}{r} \Bigg( \vandermonde{x}{k}{n}
\sqprod n \frac{\qrfac{q^{-N_s}\xover{x}}{k_r} }{\qrfac{q\xover{x}}{k_r} } \notag \\
\qquad\quad{} \times\smallprod n \frac{\qrfac{cx_r}{k_r} \qrfac{ax_rq^{1+\sumN}/c}{\sumk-k_r}}
{\qrfac{ax_rq^{1+\sumN-N_r}/c}{\sumk} } \smallprod n x_r^{-k_r}\notag \\
\qquad\quad{} \times \frac{1-aq^{2\sumk}}{1-a} \frac{\qrfac{a, b}{\sumk}}{\qrfac{aq^{\sumN+1}, aq/b}{\sumk}}
\left(\frac{aq^{1+\sumN}}{bc}\right)^{\sumk} q^{\sum\limits_{r<s}k_rk_s+\sum\limits_{r=1}^n (r-1)k_r } \Bigg)\notag \\
\qquad{} = \frac{\qrfac{aq}{\sumN}}{\qrfac{aq/b}{\sumN}}
\smallprod n \frac{\qrfac{aq^{1+\sumN-N_r}/bcx_r}{N_r}}{\qrfac{aq^{1+\sumN-N_r}/cx_r}{N_r}} . \label{6p5-BS1}
\end{gather}
\end{Theorem}
\begin{Remark*} When $n{=}1$, this formula reduces to the very-well-poised $_6\phi_5$ sum~\cite[equa\-tion~(2.4.2)]{GR90}. Several other extensions of this formula on root systems have appeared previously, see, for example, \cite{GB1995, GB1999a, vD1997, RG1989, Milne1997, ML1995, MS2007b}.
\end{Remark*}

\begin{proof} We take $c=aq/b$ in \eqref{an-watson-BS-3}. The sum on the right-hand side becomes $1$. In the resulting identity, we replace $d$ by $b$ and $e$ by $c$ to obtain \eqref{6p5-BS1}.
\end{proof}

The identity \eqref{6p5-BS1} is related to the $A_n$ $_6\phi_5$ summation due to the first author~\cite[Theorem~3.6]{GB1995}. It follows from this result by inverting the base or reversing the sum. It can also be obtained from the $A_n$ Jackson sum~\cite[Theorem~4.1]{MS2008}, given by the $p=0$ case of \eqref{e-8p7-5}, by replacing $c$ by $a^2q^{\sumN+1}/bcd$ and letting $d\to\infty$.

\begin{Theorem}[an $A_n$ Watson transformation] We have
\begin{gather}
\multsum{k}{N}{r} \Bigg( \vandermonde{x}{k}{n}
\sqprod n \frac{\qrfac{q^{-N_s}\xover{x}}{k_r} }{\qrfac{q\xover{x}}{k_r} } \notag \\
\qquad\quad{} \times\smallprod n \frac{\qrfac{ex_r}{k_r} \qrfac{b/x_r}{\sumk} \qrfac{ax_rq^{1+\sumN}/e}{\sumk-k_r}}
{\qrfac{ax_rq/b}{k_r}\qrfac{ax_rq^{1+\sumN-N_r}/e}{\sumk} \qrfac{b/x_r}{\sumk-k_r} }\notag \\
\qquad\quad{} \times \frac{1-aq^{2\sumk}}{1-a} \frac{\qrfac{a, c, d}{\sumk}}{\qrfac{aq^{\sumN+1}, aq/c, aq/d}{\sumk}}
 \left(\frac{a^2q^{2+\sumN}}{bcde}\right)^{\sumk} q^{\sum\limits_{r=1}^n (r-1)k_r } \Bigg)\notag \\
\qquad{} = \frac{\qrfac{aq}{\sumN}}{\qrfac{aq/d}{\sumN}}
\smallprod n \frac{\qrfac{aq^{1+\sumN-N_r}/dex_r}{N_r}}{\qrfac{aq^{1+\sumN-N_r}/ex_r}{N_r}}\notag \\
\qquad\quad{} \times \multsum{k}{N}{r} \Bigg( \vandermonde{x}{k}{n}
\sqprod n \frac{\qrfac{q^{-N_s}\xover{x}}{k_r} }{\qrfac{q\xover{x}}{k_r} } \notag \\
\qquad\quad{} \times \smallprod n \frac{\qrfac{ax_rq/bc}{k_r} \qrfac{ex_r}{k_r}}
{ \qrfac{ax_rq/b}{k_r} \qrfac{dex_rq^{-\sumN}/a}{k_r}} \cdot
\frac{\qrfac{ d}{\sumk}}{\qrfac{aq/c}{\sumk} }
q^{\sum\limits_{r=1}^n rk_r } \Bigg). \label{an-watson-BS-4}
\end{gather}
\end{Theorem}
\begin{proof} In \eqref{10p9-7-form2}, we replace $e\mapsto \lambda a q^{1+\sumN}/ef$ and take $d\to \infty$. In the result, we take $f\mapsto d$ to obtain~\eqref{an-watson-BS-4}.
\end{proof}

When $c=1$ in \eqref{an-watson-BS-4} the sum on the left-hand side reduces to $1$ and we obtain an equivalent formulation of Milne's balanced $_3\phi_2$ sum~\cite[Theorem~4.1]{Milne1997}.

If we set $e=a^2q^{1+\sumN}/bcd$ in \eqref{an-watson-BS-4}, we obtain the second author's $A_n$ Jackson's $_8\phi_7$ sum (the $p=0$ case of \eqref{e-8p7-5}). After replacing $e$ as specified, the sum on the right-hand side can be evaluated by setting $a\mapsto a^2q^{1+\sumN}/bcd$, $b=d$, $c\mapsto aq/b$, $x_r\mapsto x_rx_n$, for $r=1, 2, \dots, n$ in Milne's balanced $_3\phi_2$ sum~\cite[Theorem~4.1]{Milne1997}.

Before proceeding to the next section, we remark on the motivation for our search for a new Bailey $_{10}\phi_9$ transformation given in~\eqref{e-10p9-6}. Most results of this section contain a very-well-poised part
\begin{gather*}\frac{1-aq^{2\sumk}}{1-a}\end{gather*}
instead of the usual
\begin{gather*}\smallprod n\frac{1-ax_rq^{k_r+\sumk}}{1-ax_r},\end{gather*}
where the summation index is~$\k$. The first result of this type was given by the first author~\cite{GB1995}, followed by several related results by the second author~\cite{MS2008}. It was natural to search for an $A_n$ $_{10}\phi_9$ transformation involving a series with this kind of very-well-poised part that would contain all those results as special cases.

\section[Another WP Bailey pair for $\B^{(1)}$]{Another WP Bailey pair for $\boldsymbol{\B^{(1)}}$}\label{sec:dougall4}

The next elliptic Jackson summation we consider does not give rise to another Bressoud matrix. Rather surprisingly, it provides yet another WP Bailey pair with respect to the matrix $\B^{(1)}$.

The $D_n$ elliptic Jackson sum we apply in this section is due to Rosengren~\cite[Corollary~6.4]{HR2004}. Its $p=0$ case is due to the first author~\cite{GB1999a}. Rosengren's result is
\begin{gather}
\multsum{k}{N}{r} \Bigg(\ellipticvandermonde{x}{k}{n}
\sqprod n \frac{\ellipticqrfac{q^{-N_s}\xover{x}}{k_r} }{\ellipticqrfac{q\xover{x}}{k_r} }\nonumber\\
\qquad\quad{}\times
\frac{\triprod n \ellipticqrfac{ax_rx_sq/d}{k_r+k_s}}
{\sqprod n \ellipticqrfac{ax_rx_sq/d}{k_r}}
 \smallprod n
\frac{\elliptictheta{ax_rq^{k_r+\sumk}} \ellipticqrfac{ax_r,d/x_r}{\sumk}}
{\elliptictheta{ax_r}\ellipticqrfac{ax_rq^{1+N_r}}{\sumk}\ellipticqrfac{d/x_r}{\sumk-k_r}}\nonumber\\
\qquad\quad{}\times
\frac{\smallprod n{\ellipticqrfac{bx_r, cx_r, a^2x_rq^{1+\sumN}/bcd}{k_r}}
}{\ellipticqrfac{aq/b,aq/c,bcdq^{-\sumN}/a}{\sumk} }
q^{\sum\limits_{r=1}^n rk_r}\Bigg)\nonumber\\
\qquad{} =
 \frac{\triprod n \ellipticqrfac{ax_rx_sq/d}{N_r+N_s}}{\sqprod n \ellipticqrfac{ax_rx_sq/d}{N_r}} \nonumber\\
\qquad\quad{}\times \frac{\smallprod n {\ellipticqrfac{ax_rq, ax_rq/bd, ax_rq/cd, aq^{1+\sumN-N_r}/bcx_r}{N_r}}}
{\ellipticqrfac{aq/b, aq/c, aq/bcd}{\sumN}}.\label{e-8p7-4}
\end{gather}

\begin{Remark*} The $D_n$ series (with summation index $\k$) typically contain the elliptic Vandermonde product
\begin{gather*}
\ellipticvandermonde{x}{k}{n} \triprod n \frac{\elliptictheta{ ax_rx_sq^{k_r+k_s}} }{\elliptictheta{ax_rx_s}}
\end{gather*}
as a factor. (This is not followed very strictly. Sometimes the series is labelled as a $D_n$ series when this factor appears on reversing the sum or inverting the base, as in~\eqref{e-8p7-3}, below.)
\end{Remark*}

We use \eqref{e-8p7-4} to obtain a WP Bailey pair with respect to the matrix $\B^{(1)}$.
\begin{Theorem}[second elliptic WP Bailey pair with respect to $\B^{(1)}$]\label{th:e-Bailey-Pair-B1.3b} The two sequences
\begin{gather*}
\alpha_\k(a,b) :=\smallprod n \frac{\elliptictheta{ax_rq^{k_r+\sumk}}}{\elliptictheta{a}}
\frac{\triprod n \ellipticqrfac{ax_rx_sq/d}{k_r+k_s}}{\sqprod n
\ellipticqrfac{q\xover x}{k_r}\ellipticqrfac{ax_rx_sq/d}{k_r}}\notag\\
\hphantom{\alpha_\k(a,b) :=}{} \times
 \frac{\smallprod n \ellipticqrfac{cx_r, a^2x_rq/bcd}{k_r}
 \ellipticqrfac{ax_r, d/x_r }{\sumk} }
{\ellipticqrfac{aq/c, bcd/a}{\sumk} \smallprod n \ellipticqrfac{d/x_r}{\sumk-k_r} }
 \left( \frac{b}{a}\right)^{\sumk} ,\\
\intertext{and}
\beta_\k(a,b) :=\frac{\triprod n \ellipticqrfac{ax_rx_sq/d}{k_r+k_s} }
 {\sqprod n \ellipticqrfac{ax_rx_sq/d}{k_r} \ellipticqrfac{q\xover x}{k_r}
 }\\
 \hphantom{\beta_\k(a,b) :=}{} \times
\smallprod n \frac{\ellipticqrfac{bx_r,bd/ax_r}{\sumk}} {\ellipticqrfac{bd/ax_r}{\sumk-k_r}} \cdot
\frac{ \smallprod n {\ellipticqrfac{ax_rq/cd, bcx_r/a}{k_r}} }{\ellipticqrfac{aq/c, bcd/a}{\sumk}},
\end{gather*}
form a WP-Bailey pair with respect to $\B^{(1)}$.
\end{Theorem}
\begin{proof}The proof is analogous to that of Theorem~\ref{th:e-Bailey-Pair-B1.2a}, except that we use the $b\mapsto qa^2/bcd$ case of~\eqref{e-8p7-4}.
\end{proof}

Consider the inverse relation of \eqref{e-8p7-4}, in the form \eqref{multivariable-inverse-relation} where $\B=\B^{(1)}$, and $\alpha_k$ and $\beta_k$ are defined as in Theorem~\ref{th:e-Bailey-Pair-B1.3b} and $(B^{(1)}(a,b))^{-1}_{\k\j} $ is given by \eqref{e-B1-inverse}. After canceling some products, we take $a\mapsto bq^{-\sumN}$, $b\mapsto a$, $c\mapsto aq/cd$ and $d\mapsto bd/aq^{\sumN}$, to again obtain~\eqref{e-8p7-4}. Thus we do not obtain a new result by taking the inverse relation.

Since we have another WP Bailey pair with respect to the matrix $\B^{(1)}$, we can apply the $\B^{(1)} \to\B^{(1)}$ WP Bailey lemma or the $\B^{(1)} \to\B^{(2)}$ WP Bailey lemma, to obtain a WP Bailey pair with respect to the matrices $\B^{(1)}$ or $\B^{(2)}$, respectively.

An elliptic $D_n$ Bailey $_{10}\phi_9$ transformation formula due to Rosengren (which follows by reversing the sum in \cite[Corollary~8.5]{HR2004}) follows immediately by applying the $\B^{(1)}\to\B^{(1)}$ elliptic Bailey lemma in Theorem~\ref{th:e-WP-BaileyLemma-B1} to the WP Bailey pair in Theorem~\ref{th:e-Bailey-Pair-B1.3b}. When $p=0$ this reduces to the authors' formula~\cite[Theorem~3.9]{BS1998}.

The same elliptic $D_n$ Bailey $_{10}\phi_9$ transformation also follows by applying the $\B^{(1)}\to\B^{(2)}$ elliptic Bailey lemma in Theorem~\ref{th:e-WP-BaileyLemma-B1-B2} to the WP Bailey pair in Theorem~\ref{th:e-Bailey-Pair-B1.3b}.

Next, we find another Bressoud matrix and another WP Bailey lemma, which will allow us to use the WP Bailey pair of this section again.

\section[The matrix $\B^{(4)}$]{The matrix $\boldsymbol{\B^{(4)}}$}\label{sec:matrix-B4}
In this section we examine some results which are related to multiple series attached to a mix of root systems, such as $A_n$, $C_n$ and $D_n$. These results are a consequence of a~$D_n$ elliptic Jackson sum due to Rosengren~\cite[Corollary~6.3]{HR2004}. The $p=0$ case is due to the second author~\cite{MS1997}. Rosengren's result is
\begin{gather}
\multsum{k}{N}{r} \Bigg( \ellipticvandermonde{x}{k}{n}
\sqprod n \frac{\ellipticqrfac{q^{-N_s}\xover{x}}{k_r} }{\ellipticqrfac{q\xover{x}}{k_r} }\notag\\
\qquad\quad{} \times \smallprod n \frac{\elliptictheta{ax_rq^{k_r+\sumk}}}{\elliptictheta{ax_r}}
\frac{\ellipticqrfac{ax_r}{\sumk}\ellipticqrfac{bcd/ax_r}{\sumk-k_r}}
{\ellipticqrfac{ax_rq^{1+N_r}}{\sumk}\ellipticqrfac{bcdq^{-N_r}/ax_r}{\sumk}}
\notag\\
\qquad\quad{} \times \frac{\ellipticqrfac{b,c,d}{\sumk}
\sqprod n \ellipticqrfac{a^2x_rx_sq^{1+N_s}/bcd}{k_r} }
{\smallprod n \ellipticqrfac{ax_rq/b, ax_rq/c, ax_rq/d}{k_r}
\triprod n \ellipticqrfac{a^2x_rx_sq/bcd}{k_r+k_s}}
q^{\sum\limits_{r=1}^n rk_r}\Bigg)\notag\\
\qquad{} = \smallprod n \frac{\ellipticqrfac{ax_rq, ax_rq/bc, ax_rq/bd, ax_rq/cd}{N_r}}
{\ellipticqrfac{ax_rq/b, ax_rq/c, ax_rq/d, ax_rq/bcd}{N_r}}.\label{e-8p7-3}
\end{gather}
The above summation implies another Bressoud matrix and a WP Bailey pair. The matrix $\B^{(4)}= (B^{(4)}_{\k\j}(a,b))$ is defined as follows:
\begin{Definition}[a $D_n$ elliptic Bressoud matrix] We define the matrix $\B^{(4)}$ with entries indexed by $(\k , \j)$ as
\begin{gather}
B^{(4)}_{\k \j}(a,b) := \frac{\sqprod n \ellipticqrfac{bx_rx_s}{k_r+j_s} }
 {\triprod n \ellipticqrfac{bx_rx_s}{j_r+j_s} \sqprod n \ellipticqrfac{q^{1+j_r-j_s}\xover x }{k_r-j_r}} \nonumber\\
\hphantom{B^{(4)}_{\k \j}(a,b) :=}{} \times \smallprod n \frac{\ellipticqrfac{bx_r/a}{k_r-\sumj}}
{\ellipticqrfac{bx_r/a}{j_r-\sumj} \ellipticqrfac{ax_rq}{k_r+\sumj}} . \label{e-B4}
\end{gather}
\end{Definition}

\begin{Theorem}[an elliptic WP Bailey pair with respect to $\B^{(4)}$]\label{th:e-Bailey-Pair-B4.2} The two sequences
\begin{subequations}
\begin{gather}
\alpha_\k(a,b) := \smallprod n \frac{\elliptictheta{ax_rq^{k_r+\sumk}}}{\elliptictheta{ax_r}}
\frac{\ellipticqrfac{ax_r}{\sumk}} {\ellipticqrfac{ax_rq/c,ax_rq/d, bcdx_r/a }{k_r} } \notag\\
\hphantom{\alpha_\k(a,b) :=}{} \times
\frac{\ellipticqrfac{c, d, qa^2/bcd}{\sumk} }
{\sqprod n \ellipticqrfac{q\xover x}{k_r}}
 \left( \frac{b}{a}\right)^{\sumk}
q^{-\sum\limits_{r<s} k_rk_s} \smallprod n x_r^{k_r} , \label{e-Bailey-PairB4.2a} \\
 \intertext{and}
\beta_\k(a,b) := \sqprod n\frac{\ellipticqrfac{bx_rx_s}{k_s}}{\ellipticqrfac{q\xover x}{k_r}}
\smallprod n \frac{\ellipticqrfac{ax_rq/cd, bcx_r/a, bdx_r/a}{k_r}}
{\ellipticqrfac{bcdx_r/a, ax_rq/c, ax_rq/d}{k_r}},\label{e-Bailey-PairB4.2}
\end{gather}
\end{subequations}
form a WP-Bailey pair with respect to $\B^{(4)}$.
\end{Theorem}
\begin{proof}The proof is similar to that of Theorem~\ref{th:e-Bailey-Pair-B1.2a}. We use the $b\mapsto qa^2/bcd$ case of \eqref{e-8p7-3} to verify that $\alpha_\k(a,b)$ and $\beta_\k(a,b)$ form a WP-Bailey pair with respect to $\B^{(4)}$.
\end{proof}

\begin{Theorem}[an elliptic $\big(\B^{(1)} \to\B^{(4)}\big)$ WP Bailey lemma] \label{th:e-WP-BaileyLemma-B1-B4}
Suppose $\alpha_\N(a,b)$ and $\beta_\N(a,b)$ form a WP Bailey pair with respect to the matrix $\B^{(1)}$. Let
${\alpha}^{\prime}_\N(a,b)$ and ${\beta}^{\prime}_\N(a,b)$ be defined as follows:
\begin{gather*}
{\alpha}^{\prime}_\N(a,b) := \frac{\ellipticqrfac{\rho_1, \rho_2}{\sumN}}
 {\smallprod n \ellipticqrfac{ax_rq/\rho_1, ax_rq/\rho_2}{N_r}}
 \left(\frac{aq}{\rho_1\rho_2}\right)^{\sumN}
\smallprod n x_r^{N_r}\cdot q^{-\sum\limits_{r<s} N_rN_s} \alpha_\N(a, b\rho_1\rho_2/aq),\\
{\beta}^{\prime}_\N (a,b) :=
\smallprod n \frac{\ellipticqrfac{b\rho_1x_r/a, b\rho_2x_r/a }{N_r}}
{ \ellipticqrfac{ax_rq/\rho_1, ax_rq/\rho_2}{N_r}} \multsum{k}{N}{r} \Bigg(
 \frac{ \ellipticqrfac{\rho_1, \rho_2}{\sumk}}
{\smallprod n \ellipticqrfac{b\rho_1x_r/a, b\rho_2x_r/a}{k_r}}\\
\hphantom{{\beta}^{\prime}_\N (a,b) :=}{} \times
\smallprod n \frac{\elliptictheta{ b\rho_1\rho_2x_rq^{k_r+\sumk}/{aq}}}
{\elliptictheta{ b\rho_1\rho_2x_r/{aq}}}
\frac{ \ellipticqrfac{ax_rq^{1+k_r-\sumk}/\rho_1\rho_2 }{N_r-k_r}}{
\ellipticqrfac{b\rho_1\rho_2 x_r/a}{N_r+\sumk}} \notag\\
\hphantom{{\beta}^{\prime}_\N (a,b) :=}{} \times
\frac{\sqprod n \ellipticqrfac{bx_rx_s}{N_s+k_r}}{\triprod n \ellipticqrfac{bx_rx_s}{k_r+k_s}
\sqprod n \ellipticqrfac{q^{1+k_r-k_s}\xover x }{N_r-k_r}}\\
\hphantom{{\beta}^{\prime}_\N (a,b) :=}{} \times
\left(\frac{aq}{\rho_1\rho_2}\right)^{\sumk} \smallprod n x_r^{k_r}\cdot q^{-\sum\limits_{r<s} k_rk_s} \beta_\k(a,b\rho_1\rho_2/aq) \Bigg).
\end{gather*}
Then ${\alpha}^{\prime}_\N(a,b)$ and ${\beta}^{\prime}_\N(a,b)$ form a WP Bailey pair with respect to $\B^{(4)}$, defined by \eqref{e-B4}.
\end{Theorem}
\begin{Remark*} A matrix reformulation of Theorem~\ref{th:e-WP-BaileyLemma-B1-B4} appears in unpublished notes of Warnaar~\cite{SOW-notes-2016}.
\end{Remark*}

\begin{proof} The proof is similar to that of Theorem~\ref{th:e-WP-BaileyLemma-B1}. We need to use \eqref{e-8p7-3}, with the substitutions: $x_r\mapsto x_rq^{j_r}$ and $N_r\mapsto N_r-j_r$ for $r=1, 2, \dots, n$, $a\mapsto b\rho_1\rho_2q^{2\sumj}/aq$, $b\mapsto \rho_1q^{\sumj}$, $c\mapsto \rho_2q^{\sumj}$, $d\mapsto b\rho_1\rho_2q^{\sumj}/a^2q $.
\end{proof}

An elliptic $D_n$ Bailey $_{10}\phi_9$ transformation formula due to Rosengren \cite[Corollary~8.5]{HR2004} follows immediately by applying the $\B^{(1)} \to\B^{(4)}$ elliptic Bailey lemma in Theorem~\ref{th:e-WP-BaileyLemma-B1-B4} to the WP Bailey pair in Theorem~\ref{th:e-Bailey-Pair-B1.2a}. When $p=0$ this reduces to the authors' formula in \cite[Theorem~3.13]{BS1998}.

If instead we use the second WP Bailey pair in Theorem~\ref{th:e-Bailey-Pair-B1.3b}, we obtain a different elliptic $D_n$ Bailey $_{10}\phi_9$ transformation formula, again due to Rosengren~\cite[Corollary~8.4]{HR2004}. When $p=0$ this reduces to \cite[Theorem~3.1]{BS1998}.

We will compute the inverse of $\B^{(4)}$ in the next section.

\section[Consequences of a $C_n$ elliptic Jackson sum due to Rosengren]{Consequences of a $\boldsymbol{C_n}$ elliptic Jackson sum due to Rosengren}\label{sec:cn-case}

In this section we consider a $C_n$ elliptic Jackson summation theorem due to Rosengren~\cite[Theorem~7.1]{HR2004}. The $p=0$ case
was found independently by Denis and Gustafson~\cite[Theorem~4.1]{DG1992} and Milne and Lilly~\cite[Theorem~6.13]{ML1995}. We will find that the Bressoud matrix following from this result is closely related to the one in Section~\ref{sec:matrix-B4}. Again, the results are closely related to both~$D_n$ and~$A_n$ series. The summation theorem we consider is
\begin{gather}
\multsum{k}{N}{r} \Bigg(\ellipticvandermonde{x}{k}{n}
\triprod n \frac{\elliptictheta{ ax_rx_sq^{k_r+k_s}} }{\elliptictheta{ax_rx_s}}\nonumber\\
\qquad\quad{} \times \smallprod n \frac{\elliptictheta{ax_r^2q^{2k_r}}}{\elliptictheta{ax_r^2}}
 \sqprod n \frac{\ellipticqrfac{q^{-N_s}\xover{x}, ax_rx_s}{k_r} }
 {\ellipticqrfac{q\xover{x}, ax_rx_sq^{1+N_s}}{k_r} } \nonumber\\
\qquad\quad{} \times
\smallprod n \frac{\ellipticqrfac{bx_r, cx_r, dx_r, a^2x_rq^{1+\sumN}/bcd}{k_r}}
{\ellipticqrfac{ax_rq/b, ax_rq/c, ax_rq/d, bcdx_rq^{-\sumN}/a}{k_r}}
\cdot q^{\sum\limits_{r=1}^n r k_r}\Bigg)\nonumber\\
\qquad{} =
\frac{\sqprod n \ellipticqrfac{ax_rx_s}{N_r} }
{\triprod n \ellipticqrfac{ax_rx_s}{N_r+N_s} } \nonumber\\
\qquad\quad\times
\frac{\ellipticqrfac{aq/bc, aq/bd, aq/cd}{\sumN}}
{ \smallprod n \ellipticqrfac{ax_rq/b, ax_rq/c, ax_rq/d, aq^{1+\sumN-N_r}/bcdx_r}{N_r}}.\label{e-8p7-2}
\end{gather}
\begin{Remark*} The
$C_n$ series (with summation index $\k$) contain the elliptic Vandermonde product
\begin{gather*}
\ellipticvandermonde{x}{k}{n}
\triprod n \frac{\elliptictheta{ ax_rx_sq^{k_r+k_s}} }{\elliptictheta{ax_rx_s}}
 \smallprod n \frac{\elliptictheta{ax_r^2q^{2k_r}}}{\elliptictheta{ax_r^2}}
\end{gather*}
as a factor.
\end{Remark*}

We use this result to obtain an elliptic WP Bailey lemma.

\begin{Theorem}[an elliptic $\big(\B^{(4)} \to\B^{(1)}\big)$ WP Bailey lemma] \label{th:e-WP-BaileyLemma-B4-B1}
Suppose $\alpha_\N(a,b)$ and $\beta_\N(a,b)$ form a WP Bailey pair with respect to the matrix $\B^{(4)}$ defined in \eqref{e-B4}. Let
${\alpha}^{\prime}_\N(a,b)$ and ${\beta}^{\prime}_\N(a,b)$ be defined as follows
\begin{subequations}
\begin{gather}
{\alpha}^{\prime}_\N(a,b) :=
 \frac{\smallprod n \ellipticqrfac{\rho_1x_r, \rho_2x_r}{N_r}}
 {\ellipticqrfac{aq/\rho_1, aq/\rho_2}{\sumN}}
 \left(\frac{aq}{\rho_1\rho_2}\right)^{\sumN}
\smallprod n x_r^{-N_r}\cdot
 q^{\sum\limits_{r<s} N_rN_s}
 \alpha_\N(a, b\rho_1\rho_2/aq), \label{e-alphaprime-B4-B1} \\
{\beta}^{\prime}_\N (a,b) :=
\frac{\smallprod n \ellipticqrfac{b\rho_1x_r/a, b\rho_2x_r/a }{N_r}}
{ \ellipticqrfac{aq/\rho_1, aq/\rho_2}{\sumN}}\nonumber\\
\hphantom{{\beta}^{\prime}_\N (a,b) := }{}
 \times \multsum{k}{N}{r} \Bigg(
\smallprod n \frac{ \ellipticqrfac{\rho_1x_r, \rho_2x_r}{k_r}
\ellipticqrfac{aq^{1+\sumN-N_r}/\rho_1\rho_2 x_r}{N_r-k_r}}
{ \ellipticqrfac{b\rho_1x_r/a, b\rho_2x_r/a}{k_r}}\nonumber\\
\hphantom{{\beta}^{\prime}_\N (a,b) := }{} \times
\triprod n \frac{\elliptictheta{ b\rho_1\rho_2x_rx_sq^{k_r+k_s}/{aq}}}
{\elliptictheta{ b\rho_1\rho_2x_rx_s/{aq}}}
\smallprod n \frac{\elliptictheta{ b\rho_1\rho_2x_r^2q^{2k_r}/{aq}}}
{\elliptictheta{ b\rho_1\rho_2x_r^2/{aq}}} \nonumber\\
\hphantom{{\beta}^{\prime}_\N (a,b) := }{} \times
\frac{\triprod n \ellipticqrfac{b\rho_1\rho_2x_rx_s/a}{N_r+N_s}
\smallprod n \ellipticqrfac{bx_r}{\sumN+k_r}}{\sqprod n \ellipticqrfac{b\rho_1\rho_2x_rx_s/a}{N_s+k_r}
 \ellipticqrfac{q^{1+k_r-k_s}\xover x }{N_r-k_r}}\nonumber\\
\hphantom{{\beta}^{\prime}_\N (a,b) := }{} \times \left(\frac{aq}{\rho_1\rho_2}\right)^{\sumk}
\smallprod n x_r^{-k_r}\cdot q^{\sum_{r<s}k_rk_s}
\beta_\k(a,b\rho_1\rho_2/aq) \Bigg).\label{e-betaprime-B4-B1}
\end{gather}
\end{subequations}
Then ${\alpha}^{\prime}_\N(a,b)$ and ${\beta}^{\prime}_\N(a,b)$ form a WP Bailey pair with respect to $\B^{(1)}$, defined by~\eqref{e-B1}.
\end{Theorem}
 \begin{Remark*}
 A matrix reformulation of Theorem~\ref{th:e-WP-BaileyLemma-B4-B1} appears in unpublished notes of Warnaar~\cite{SOW-notes-2016}.
 \end{Remark*}

\begin{proof} The proof is analogous to that of Theorem~\ref{th:e-WP-BaileyLemma-B1}. The only difference is that we use \eqref{e-8p7-2}, with the substitutions: $x_r\mapsto x_rq^{j_r}$ and $N_r\mapsto N_r-j_r$ for $r=1, 2, \dots, n$,
$a\mapsto b\rho_1\rho_2/aq$, $b\mapsto \rho_1$,
$c\mapsto \rho_2$,
$d\mapsto b\rho_1\rho_2q^{-\sumj}/a^2q $. The remaining calculations are very similar.
\end{proof}

\begin{Remark} \label{rem:B4-B1-10p9}
An elliptic $C_n\to A_n$ Bailey $_{10}\phi_9$ transformation formula due to Rosengren~\cite[Corollary~8.3]{HR2004} follows immediately by applying the $\B^{(4)}\to\B^{(1)}$ elliptic Bailey lemma in Theorem~\ref{th:e-WP-BaileyLemma-B4-B1} to the WP Bailey pair in
Theorem~\ref{th:e-Bailey-Pair-B4.2}. When $p=0$ this reduces to \cite[Theorem~2.1]{BS1998}. Note that the $\beta^{\prime}_\N$ defined in \eqref{e-betaprime-B4-B1} has the very-well-poised part usually present in $C_n$ series. However the $\alpha_\N$ in \eqref{e-alphaprime-B4-B1} (which comes from the definition \eqref{e-Bailey-PairB4.2a} of the WP Bailey pair with respect to $\B^{(4)}$) contains the usual $A_n$ very-well-poised part.
\end{Remark}

Next we have another elliptic Bressoud matrix and a WP Bailey pair from \eqref{e-8p7-2}. The matrix $\B^{(5)}= \big(B^{(5)}_{\k\j}(a,b)\big)$ is defined as follows.
\begin{Definition}[a $C_n$ elliptic Bressoud matrix] We define the matrix $\B^{(5)}$ with entries indexed by $(\k , \j)$ as
\begin{gather}
B^{(5)}_{\k \j}(a,b) := \frac{\triprod n \ellipticqrfac{ax_rx_sq}{k_r+k_s}
 \smallprod n \ellipticqrfac{bq^{\sumk-k_r}/ax_r}{k_r-j_r} \ellipticqrfac{bx_r}{\sumk+j_r} } {\sqprod n \ellipticqrfac{q^{1+j_r-j_s}\xover x }{k_r-j_r} \ellipticqrfac{ax_rx_sq}{k_s+j_r} }.\label{e-B5}
\end{gather}
\end{Definition}

\begin{Theorem}[an elliptic WP Bailey pair with respect to $\B^{(5)}$]\label{th:e-Bailey-Pair-B5.2} The two sequences
\begin{gather*}
\alpha_\k(a,b) := \triprod n \frac{\elliptictheta{ax_rx_sq^{k_r+k_s}}}{\elliptictheta{ax_rx_s}} \smallprod n \frac{\elliptictheta{ax_r^2q^{2k_r}}}{\elliptictheta{ax_r^2}} \cdot \left( \frac{b}{a}\right)^{\sumk} q^{\sum\limits_{r<s} k_r k_s} \notag\\
\hphantom{\alpha_\k(a,b) :=}{} \times
\sqprod n\frac{\ellipticqrfac{ax_rx_s}{k_r}}{\ellipticqrfac{q\xover x}{k_r} } \smallprod n \frac{\ellipticqrfac{a^2x_rq/bcd, cx_r, dx_r}{k_r}}
{\ellipticqrfac{ax_rq/c,ax_rq/d, bcdx_r/a }{k_r} }\smallprod n x_r^{-k_r}, \\
\intertext{and}
\beta_\k(a,b) := \frac{\ellipticqrfac{bc/a, bd/a, aq/cd}{\sumk}
 \smallprod n \ellipticqrfac{bx_r}{\sumk} }
 { \sqprod n\ellipticqrfac{q\xover x}{k_r}
 \smallprod n \ellipticqrfac{bcdx_r/a, ax_rq/c, ax_rq/d}{k_r} },
\end{gather*}
form a WP-Bailey pair with respect to $\B^{(5)}$.
\end{Theorem}
\begin{proof}The proof requires the $C_n$ elliptic Jackson sum given in~\eqref{e-8p7-2}. We verify that~$\alpha_\k(a,b)$ and $\beta_\k(a,b)$ form a WP-Bailey pair with respect to~$\B^{(5)}$ using the $b\mapsto qa^2/bcd$ case of \eqref{e-8p7-2}.
\end{proof}

As a corollary, we obtain a unit WP Bailey pair.
\begin{Corollary}\label{th:e-Bailey-Pair-B5.1}
The two sequences
\begin{gather*}
\alpha_\k(a,b) := \triprod n \frac{\elliptictheta{ax_rx_sq^{k_r+k_s}}}{\elliptictheta{ax_rx_s}}
\smallprod n \frac{\elliptictheta{ax_r^2q^{2k_r}}}{\elliptictheta{ax_r^2}}
\cdot \left( \frac{b}{a}\right)^{\sumk}
 q^{\sum\limits_{r<s} k_r k_s} \\
\hphantom{\alpha_\k(a,b) :=}{} \times
\sqprod n\frac{\ellipticqrfac{ax_rx_s}{k_r}}{\ellipticqrfac{q\xover x}{k_r} } 
 \smallprod n \frac{\ellipticqrfac{ax_r/b}{k_r}}
{\ellipticqrfac{bx_rq }{k_r} } \smallprod n x_r^{-k_r},\\
 \intertext{and}
\beta_\k(a,b) := \smallprod n \delta_{k_r, 0},
\end{gather*}
form a WP-Bailey pair with respect to $\B^{(5)}$.
\end{Corollary}
\begin{proof} Take $d=aq/c$ in Theorem~\ref{th:e-Bailey-Pair-B5.2}.
\end{proof}
We can find a formula for the inverse of $\B^{(5)}$ using the unit Bailey pair.

\begin{Corollary}[inverse of $\B^{(5)}$]Let $\B^{(5)}= \big(B^{(5)}_{\k\j}(a,b)\big)$ be defined by~\eqref{e-B5}. Then the entries of its inverse are given by
\begin{gather}
\big(B^{(5)}(a,b)\big)_{\k \j}^{-1} =
\triprod n
\frac{\elliptictheta{ax_rx_sq^{k_r+k_s}}}{\elliptictheta{ax_rx_s}}
\smallprod n
\frac{\elliptictheta{ax_r^2q^{2k_r}}}{\elliptictheta{ax_r^2}}
\frac{\elliptictheta{bx_rq^{j_r+\sumj}}}{\elliptictheta{bx_r}}\nonumber\\
\hphantom{\big(B^{(5)}(a,b)\big)_{\k \j}^{-1} =}{} \times
\smallprod n \frac{\ellipticqrfac{ax_rq^{j_r-\sumj}/b}{k_r-j_r}}
{\ellipticqrfac{bx_rq}{k_r+\sumj}} \cdot \left( \frac{b}{a}\right)^{\sumk-\sumj} q^{\sum\limits_{r<s} (k_r k_s-j_rj_s)}
\smallprod n x_r^{j_r-k_r}\nonumber\\
\hphantom{\big(B^{(5)}(a,b)\big)_{\k \j}^{-1} =}{} \times
 \frac{\sqprod n \ellipticqrfac{ax_rx_s}{k_s+j_r} }
 {\triprod n\ellipticqrfac{ax_rx_s}{j_r+j_s} \sqprod n \ellipticqrfac{q^{1+j_r-j_s}\xover x }{k_r-j_r}}.\label{e-B5-inverse}
\end{gather}
\end{Corollary}
\begin{proof}
The derivation is analogous to that of Corollary~\ref{cor:e-B3-inverse} and left to the reader.
\end{proof}

\begin{Corollary}[inverse of $\B^{(4)}$] Let $\B^{(4)}= \big(B^{(4)}_{\k\j}(a,b)\big)$ be defined by~\eqref{e-B4}. Then the entries of its inverse is given by
\begin{gather*}
\big(B^{(4)}(a,b)\big)_{\k \j}^{-1} =
\smallprod n \frac{\elliptictheta{ax_rq^{k_r+\sumk}}}{\elliptictheta{ax_r}} \cdot
\left( \frac{b}{a}\right)^{\sumk-\sumj}
 q^{\sum\limits_{r<s}(j_rj_s- k_r k_s)}
\smallprod n x_r^{k_r-j_r}\\
\hphantom{\big(B^{(4)}(a,b)\big)_{\k \j}^{-1}=}{} \times
\triprod n
\frac{\elliptictheta{bx_rx_sq^{j_r+j_s}}}{\elliptictheta{bx_rx_s}}
\smallprod n
\frac{\elliptictheta{bx_r^2q^{2j_r}}}{\elliptictheta{bx_r^2}} \cr
\hphantom{\big(B^{(4)}(a,b)\big)_{\k \j}^{-1}=}{} \times
\frac{\triprod n \ellipticqrfac{bx_rx_sq}{k_r+k_s} }
 {\sqprod n \ellipticqrfac{q^{1+j_r-j_s}\xover x }{k_r-j_r} \ellipticqrfac{bx_rx_sq}{k_s+j_r} } \\
\hphantom{\big(B^{(4)}(a,b)\big)_{\k \j}^{-1}=}{}
\times \smallprod n \ellipticqrfac{aq^{\sumk-k_r}/bx_r}{k_r-j_r} \ellipticqrfac{ax_r}{\sumk+j_r}.
\end{gather*}
\end{Corollary}
\begin{proof} Observe that the entries of $\B^{(5)}(a,b)^{-1}$ consist of the entries of $\B^{(4)}(b,a)$ multiplied by some additional factors, which can be separated into factors containing either terms with index~$\j$ or with index~$\k$. This can help us find the inverse of $\B^{(4)}$ as in the proof of Corollary~\ref{cor:e-B2-inverse}. We leave the details to the reader.
\end{proof}

Consider the inverse relation \eqref{multivariable-inverse-relation} where $\B=\B^{(5)}$, and $\alpha_k$ and $\beta_k$ are defined as in Theorem~\ref{th:e-Bailey-Pair-B5.2} and $\big(B^{(5)}(a,b)\big)^{-1}_{\k\j} $ is given by~\eqref{e-B5-inverse}. After canceling some products, make the substitutions $a\mapsto qa^2/bcd$, $b\mapsto a$, $c\mapsto aq/bd$ and $d\mapsto aq/bc$ to obtain~\eqref{e-8p7-3}. Thus~\eqref{e-8p7-2} and \eqref{e-8p7-3} are inverse relations. This approach provides an alternate derivation of \eqref{e-8p7-3} beginning with \eqref{e-8p7-2}.

\begin{Theorem}[Zhang and Huang \cite{ZH-preprint}; an elliptic $\big(\B^{(5)} \to\B^{(5)}\big)$ WP Bailey lemma] \label{th:e-WP-BaileyLemma-B5-B5}
Suppose $\alpha_\N(a,b)$ and $\beta_\N(a,b)$ form a WP Bailey pair with respect to the matrix $\B^{(5)}$. Let ${\alpha}^{\prime}_\N(a,b)$ and ${\beta}^{\prime}_\N(a,b)$ be defined as follows
\begin{gather*}
{\alpha}^{\prime}_\N(a,b) :=
 \smallprod n \frac{ \ellipticqrfac{\rho_1x_r, \rho_2x_r}{N_r}}
{\ellipticqrfac{ax_rq/\rho_1, ax_rq/\rho_2}{N_r}}
\cdot \left(\frac{aq}{\rho_1\rho_2}\right)^{\sumN}
 \alpha_\N(a, b\rho_1\rho_2/aq),\\
{\beta}^{\prime}_\N (a,b) :=
\frac{\ellipticqrfac{b\rho_1/a, b\rho_2/a }{\sumN}}
{\smallprod n \ellipticqrfac{ax_rq/\rho_1, ax_rq/\rho_2}{N_r}}\\
\hphantom{{\beta}^{\prime}_\N (a,b) :=}{} \times \multsum{k}{N}{r} \Bigg( \frac{ \smallprod n \ellipticqrfac{\rho_1x_r, \rho_2x_r}{k_r}}
{\ellipticqrfac{b\rho_1/a, b\rho_2/a}{\sumk}}
\frac{\ellipticqrfac{aq/\rho_1\rho_2}{\sumN-\sumk} }
{\sqprod n \ellipticqrfac{q^{1+k_r-k_s}\xover x }{N_r-k_r}}\\
\hphantom{{\beta}^{\prime}_\N (a,b) :=}{} \times
\smallprod n \frac{\elliptictheta{ b\rho_1\rho_2x_rq^{k_r+\sumk}/{aq}}}{\elliptictheta{ b\rho_1\rho_2x_r/{aq}}}
\frac{\ellipticqrfac{bx_r}{\sumN+k_r}}{\ellipticqrfac{b\rho_1\rho_2x_r/a}{N_r+\sumk}} \notag\\
\hphantom{{\beta}^{\prime}_\N (a,b) :=}{} \times
\left(\frac{aq}{\rho_1\rho_2}\right)^{\sumk}
\beta_\k(a,b\rho_1\rho_2/aq) \Bigg).
\end{gather*}
Then ${\alpha}^{\prime}_\N(a,b)$ and ${\beta}^{\prime}_\N(a,b)$ form a WP Bailey pair with respect to $\B^{(5)}$.
\end{Theorem}
\begin{Remark*} Theorem~\ref{th:e-Bailey-Pair-B5.2} is equivalent to a theorem of Zhang and Huang~\cite[Theorem~5.3]{ZH-preprint}. The matrix they consider is equivalent to \eqref{e-B5}, with slightly different notation. This result appears in unpublished notes of Warnaar~\cite{SOW-notes-2016} too.
\end{Remark*}
\begin{proof} The proof follows the model of Theorem~\ref{th:e-WP-BaileyLemma-B1}, except that we use \eqref{e-8p7-4}, with the following substitutions: $x_r\mapsto x_rq^{j_r}$ and $N_r\mapsto N_r-j_r$ for $r=1, 2, \dots, n$, $a\mapsto b\rho_1\rho_2q^{\sumj}/aq$, $b\mapsto \rho_1$, $c\mapsto \rho_2$, $d\mapsto b\rho_1\rho_2q^{\sumj}/a^2q$.
\end{proof}

An elliptic $C_n\to A_n$ Bailey $_{10}\phi_9$ transformation formula due to Rosengren~\cite[Corollary~8.3]{HR2004} follows immediately by applying the $\B^{(5)} \to\B^{(5)}$ elliptic Bailey lemma in Theorem~\ref{th:e-WP-BaileyLemma-B5-B5} to the WP Bailey pair in Theorem~\ref{th:e-Bailey-Pair-B5.2}. This $C_n\to A_n$ elliptic transformation formula is the same as obtained in Remark~\ref{rem:B4-B1-10p9}.

We have completed our study of the existing elliptic Jackson theorems on root systems that are relevant in this theory. In the next section, we summarize our results, and examine our results from another perspective to see whether we have missed anything that fits our approach.

\section{Summary of results}\label{sec:summary}
In Sections~\ref{sec:dougall1}--\ref{sec:cn-case}, we have systematically considered the consequences of five elliptic Jackson summation theorems. In this section, we provide a summary of our findings so far and examine our results. Our examination suggests one more idea to follow up before closing this study.

Here is a list of our findings.
\begin{enumerate}\itemsep=0pt
\item We have considered five elliptic Jackson summations on root systems.
\item We have defined five elliptic Bressoud matrices, denoted $\B^{(1)}$--$\B^{(5)}$, so far.
\item We found six WP Bailey pairs. There were two WP Bailey pairs with respect to $\B^{(1)}$. (We are not counting the unit WP Bailey pairs.)
\item In case there is a unit WP Bailey pair, we are able to find a formula for the inverse of the matrix.
\item Up to normalization, the inverse of $\B^{(1)}(a,b)$ is given by $\B^{(1)}(b,a)$. Similarly, the inverse of $\B^{(2)}(a,b)$ is (up to normalization) $\B^{(3)}(b,a)$. The matrices $\B^{(4)}$ and $\B^{(5)}$ are similarly related.
\item We computed the inverse relations arising out of the elliptic Jackson summations and the related Bressoud matrix inverses. We found the following relationships between the elliptic Jackson summations.
\begin{itemize}\itemsep=0pt
\item If we write the inverse relation of \eqref{e-8p7-1} with respect to $\B^{(1)}$, we obtain an equivalent form of \eqref{e-8p7-1}.
\item If we write the inverse relation of \eqref{e-8p7-5} with respect to $\B^{(3)}$, we obtain an equivalent form of \eqref{e-8p7-1}. To go in the other direction, we use $\B^{(2)}$.
\item If we write the inverse relation of \eqref{e-8p7-4} with respect to $\B^{(1)}$, we obtain an equivalent form of \eqref{e-8p7-4}.
\item If we write the inverse relation of \eqref{e-8p7-2} with respect to $\B^{(5)}$, we obtain an equivalent form of \eqref{e-8p7-3}. In the other direction, we use $\B^{(4)}$.
\end{itemize}
\item We have found eight WP Bailey lemmas (see Table~\ref{table:2} for the list).
\item We recovered the elliptic Bailey transformations in Rosengren~\cite{HR2004} by the WP Bailey lemma approach. In addition, we found one new $A_n$ elliptic Bailey transformation, so far.
\item As basic hypergeometric special cases of the new $A_n$ Bailey $_{10}\phi_9$ transformation, we found four new $A_n$ Watson transformations and one $A_n$ terminating, very-well-poised, $_6\phi_5$ summation.
\end{enumerate}

\begin{table}
\centering
\begin{tabular}{| r | c | c | c | c | c | }
\hline
Matrix: & $\B^{(1)}$ & $\B^{(2)}$ & $\B^{(3)}$ & $\B^{(4)}$ & $\B^{(5)}$ \\
\hline
$\B^{(1)}$ & Theorem~\ref{th:e-WP-BaileyLemma-B1} & Theorem~\ref{th:e-WP-BaileyLemma-B1-B2} & &
Theorem~\ref{th:e-WP-BaileyLemma-B1-B4} &
 \\
 \hline
$\B^{(2)}$ & Theorem~\ref{th:e-WP-BaileyLemma-B2-B1} & Theorem~\ref{th:e-WP-BaileyLemma-B2-B2} & & &
 \\
 \hline
$\B^{(3)}$ && & Theorem~\ref{th:e-WP-BaileyLemma-B3-B3} & &
 \\
 \hline
$\B^{(4)}$ & Theorem~\ref{th:e-WP-BaileyLemma-B4-B1} & & &
 &
 \\
\hline
$\B^{(5)}$ & & & & & Theorem~\ref{th:e-WP-BaileyLemma-B5-B5}
 \\
 \hline
\end{tabular}
\caption{The WP Bailey lemmas.}\label{table:2}
\end{table}

As we have seen, we can apply the WP Bailey lemmas to obtain elliptic Bailey transformation formulas on root systems. As in the dimension 1 case, we can iterate the WP Bailey lemmas and obtain results with more parameters. For example, one can begin with a WP Bailey pair with respect to the matrix $\B^{(1)}$ and apply a $\B^{(1)} \to\B^{(4)}$ Bailey lemma to obtain a WP Bailey pair with respect to $\B^{(4)}$. One can now use the $\B^{(4)} \to\B^{(1)}$ WP Bailey lemma to obtain a WP Bailey pair with respect to $\B^{(1)}$. In the next step, one can apply the $\B^{(1)} \to\B^{(2)}$ WP Bailey lemma.

If we look closely at the definitions of $\beta^{\prime}_\N$ in any of the WP Bailey lemmas, we can observe an interesting pattern. For example, consider the definition of $\beta^{\prime}_\N$ given in the $\B^{(4)}\to\B^{(1)}$ WP Bailey lemma in \eqref{e-betaprime-B4-B1}. Observe that this expression contains the expression $B^{(5)}_{\N\k}(a, b)$, where $a$ is replaced by $b\rho_1\rho_2/aq$. One can say something similar for all the WP Bailey lemmas presented in this paper, except for one.

The one exception is the $\B^{(2)} \to\B^{(2)}$ WP Bailey lemma given by Theorem~\ref{th:e-WP-BaileyLemma-B2-B2}. This suggests that we may have missed the Bressoud matrix (stated here for $p=0$):
\begin{gather*}
 B_{\k \j}(a,b) := \frac{\qrfac{b}{\sumk+\sumj}
 \qrfac{b/a}{\sumk-\sumj}}{\qrfac{aq}{\sumk+\sumj} \sqprod n \qrfac{q^{1+j_r-j_s}\xover x }{k_r-j_r} } .
 \end{gather*}
Indeed, we have the $A_n$ Jackson sum
\begin{gather}
\multsum{k}{N}{r} \Bigg( \vandermonde{x}{k}{n}
\sqprod n \frac{\qrfac{q^{-N_s}\xover{x}}{k_r} }{\qrfac{q\xover{x}}{k_r} } \nonumber\\
\qquad\quad{} \times \frac {(1-aq^{2\sumk}) \qrfac{a, b, c, d, a^2q^{1+\sumN}/b c d}{\sumk}}
{(1-a)\qrfac{aq^{\sumN+1}, aq/b , aq/c, aq/d, b c dq^{-\sumN}/a}{\sumk}}
 q^{\sum\limits_{r=1}^n rk_r}
 \Bigg)\nonumber\\
\qquad{} =\frac{\qrfac{aq, aq/b c, aq/b d, aq/cd}{\sumN}} {\qrfac{aq/b , aq/c, aq/d, aq/b c d}{\sumN}},\label{e-8p7-7b-conjecture}
\end{gather}
which follows from the following result due to Milne~\cite{Milne1997}, which the first author~\cite{GB1995} dubbed the ``Fundamental Theorem of ${\mathrm U}(n)$ series'', namely
\begin{gather}\label{fundamental-theorem}
\sum\limits_{\substack{\sumk = K \\
k_1,k_2, \dots, k_n\geq 0} }\!
 \vandermonde{x}{k}{n}\!
\sqprod n \frac{\qrfac{a_s\xover{x}}{k_r} }{\qrfac{q\xover{x}}{k_r} }
\cdot q^{\sum\limits_{r=1}^n (r-1)k_r}
= \frac{\qrfac{a_1\cdots a_n}{K}}{\qrfac{q}{K}} .\!\!\!
\end{gather}
Unfortunately, if we formally replace each term of~\eqref{e-8p7-7b-conjecture} by its elliptic analogue, the resulting summation is false. However, in the next section we find an elliptic Jackson summation which contains~\eqref{e-8p7-7b-conjecture} as a special case.

Before heading to the next section, we note that the observation above can be explained by the matrix approach to the WP Bailey lemma, given by Agarwal, Andrews and Bressoud~\cite{AAB1988} and Warnaar~\cite{SOW2003}. Indeed, Warnaar~\cite{SOW-notes-2016} extended this matrix formulation for his (unpublished) work on multivariable WP Bailey lemmas.

\section{Other elliptic Jackson summations, with an extra parameter}\label{sec:new-dougall}

The objective of this section is to give an elliptic extension of \eqref{e-8p7-7b-conjecture} by adding another parameter. To do that, we present a nice trick that is useful in many contexts. Essentially, this trick is an elliptic extension of one of Milne's lemmas~\cite[Lemma~7.3]{Milne1997} that he used~\cite{Milne1985} to prove one of the Macdonald identities.

We will use the following theorem of Rosengren, which can be shown to be equivalent to \eqref{e-8p7-1}, by replacing $n$ by $n+1$ relabeling parameters, and using an analytic continuation argument
\begin{gather}
\sum\limits_{\substack{\sumk = K \\
k_1,k_2, \dots, k_n\geq 0} } \Bigg( \ellipticvandermonde{x}{k}{n}
\sqprod n \frac{\ellipticqrfac{a_s\xover{x}}{k_r} }{\ellipticqrfac{q\xover{x}}{k_r} } \nonumber\\
\qquad\quad{} \times
\smallprod n \frac{\ellipticqrfac{bx_r/a_1\cdots a_n}{k_r}}{\ellipticqrfac{bx_r}{k_r} }\cdot q^{\sum\limits_{r=1}^n (r-1)k_r}\Bigg)\nonumber\\
\qquad{} = \frac{\ellipticqrfac{a_1\cdots a_n}{K}}{\ellipticqrfac{q}{K}}
\smallprod n \frac{\ellipticqrfac{bx_r/a_r}{K}}{\ellipticqrfac{bx_r}{K} } .\label{e-fundamental-theorem}
\end{gather}
This is equivalent to Rosengren's result~\cite[Theorem~5.1]{HR2004}, where we take $N=K$, $z_k\mapsto x_k$, $a_k\mapsto a_k/x_k$ (for $k=1, 2, \dots, n$) and replace $a_{n+1}$ by $b/a_1a_2\cdots a_n$. Note that when $p=0$ and $b=0$, \eqref{e-fundamental-theorem} reduces to \eqref{fundamental-theorem}.

\begin{Theorem}\label{th:e-trivial-lemma1} Given the sequence $f_k$, $k=0,1,2, \dots$, and $N\geq 0$, we have
\begin{gather}
\sum\limits_{\substack{0\leq \sumk \leq N \\
k_1,k_2, \dots, k_n\geq 0} }
 \Bigg( \ellipticvandermonde{x}{k}{n}
\sqprod n \frac{\ellipticqrfac{a_s\xover{x}}{k_r} }{\ellipticqrfac{q\xover{x}}{k_r} } \nonumber\\
\qquad\quad{} \times \smallprod n \frac{\ellipticqrfac{bx_r/a_1\cdots a_n}{k_r} \ellipticqrfac{bx_r}{\sumk}}
{\ellipticqrfac{bx_r}{k_r} \ellipticqrfac{bx_r/a_r}{\sumk}}
\cdot q^{\sum\limits_{r=1}^n (r-1)k_r} f_{\sumk} \Bigg)\nonumber\\
\qquad{} = \sum_{K=0}^N\frac{\ellipticqrfac{a_1\cdots a_n}{K}}{\ellipticqrfac{q}{K}} f_K .\label{e-trivial-lemma1}
\end{gather}
\end{Theorem}
\begin{Remark*} When $p=0$ and $b=0$ in \eqref{e-trivial-lemma1}, we obtain Milne's lemma~\cite[Lemma~7.3]{Milne1997}. Observe that the right-hand side of \eqref{e-trivial-lemma1} is not dependent on $x_1, x_2, \dots, x_n$ and $b$. Further, note that the sums are indefinite.
\end{Remark*}
\begin{proof} The theorem follows by taking the products with parameter $b$ from the right-hand side of \eqref{e-fundamental-theorem} to the left, multiplying both sides by $f_K$ and then summing over $K$ from $0$ to $N$.
\end{proof}

\begin{Remark*} We can obtain a similar result from Rosengren~\cite[Theorem~3.1]{HR2017a}.
\end{Remark*}

Theorem~\ref{th:e-trivial-lemma1} allows us to choose the $f_k$ appropriately and use a dimension 1 identity to obtain its multiple series extension. In particular, we now obtain an $A_n$ elliptic Jackson sum in this manner.
\begin{Theorem}[an $A_n$ elliptic Jackson summation]\label{th:e-8p7-7a} We have
\begin{gather}
\sum\limits_{\substack{0\leq \sumk \leq N \\ k_1,k_2, \dots, k_n\geq 0} }
 \Bigg( \ellipticvandermonde{x}{k}{n}
\sqprod n \frac{\ellipticqrfac{b_s\xover{x}}{k_r} }{\ellipticqrfac{q\xover{x}}{k_r} }
\cdot \frac{\elliptictheta{aq^{2\sumk}}}
{\elliptictheta{a}}\nonumber\\
\qquad\quad {} \times \frac {\ellipticqrfac{ a, c, d, a^2q^{1+N}/b_1\cdots b_n c d, q^{-N}}{\sumk}}
{\ellipticqrfac{ aq/b_1\cdots b_n , aq/c, aq/d, b_1\cdots b_n c dq^{-N}/a, aq^{N+1}}{\sumk}}
\nonumber\\
\qquad\quad{} \times \smallprod n \frac{\ellipticqrfac{ex_r/b_1\cdots b_n}{k_r} \ellipticqrfac{ex_r}{\sumk}}
{\ellipticqrfac{ex_r}{k_r} \ellipticqrfac{ex_r/b_r}{\sumk}} \cdot q^{\sum\limits_{r=1}^n rk_r}\Bigg)\nonumber\\
 \qquad{} = \frac
{\ellipticqrfac{aq, aq/b_1\cdots b_n c, aq/b_1\cdots b_n d, aq/cd}{N}}
{\ellipticqrfac{aq/b_1\cdots b_n , aq/c, aq/d, aq/b_1\cdots b_n c d}{N}} .\label{e-8p7-7a}
\end{gather}
\end{Theorem}
\begin{proof}
We take $a_k=b_k$ for $k=1, 2, \dots, n$, $b=e$ in \eqref{e-trivial-lemma1}, and take
\begin{gather*}f_K= \frac {\elliptictheta{aq^{2K}}}{\elliptictheta{a}}\frac{\ellipticqrfac{ a, c, d, a^2q^{1+N}/b_1\cdots b_n c d, q^{-N} }{K}}
{\ellipticqrfac{ aq/b_1\cdots b_n , aq/c, aq/d, b_1\cdots b_n c dq^{-N}/a, aq^{N+1}}{K}} q^{K}.
\end{gather*}
The left-hand side of \eqref{e-trivial-lemma1} then becomes the left-hand side of \eqref{e-8p7-7a}. Now in the right-hand side of~\eqref{e-trivial-lemma1}, we use the Frenkel--Turaev summation~\eqref{10V9}, with $b\mapsto b_1\cdots b_n$. In this manner, we obtain the right-hand side of~\eqref{e-8p7-7a}.
\end{proof}

We can rewrite this identity so that the sum is over an $n$-dimensional rectangle.

\begin{Theorem}[an $A_n$ elliptic Jackson summation]\label{th:e-8p7-7b} We have
\begin{gather}
\multsum{k}{N}{r}
 \Bigg(\ellipticvandermonde{x}{k}{n}
\sqprod n \frac{\ellipticqrfac{q^{-N_s}\xover{x}}{k_r} }{\ellipticqrfac{q\xover{x}}{k_r} } \nonumber\\
\qquad\quad{}\times \frac {\elliptictheta{aq^{2\sumk}}} {\elliptictheta{a}}
\frac{\ellipticqrfac{a, b, c, d, a^2q^{1+\sumN}/b c d}{\sumk}}{\ellipticqrfac{aq^{\sumN+1}, aq/b , aq/c, aq/d, b c dq^{-\sumN}/a}{\sumk}}
\nonumber\\
\qquad\quad{} \times \smallprod n \frac{\ellipticqrfac{ex_rq^{\sumN}}{k_r}
\ellipticqrfac{ex_r}{\sumk}}{\ellipticqrfac{ex_r}{k_r} \ellipticqrfac{ex_rq^{N_r}}{\sumk}}q^{\sum\limits_{r=1}^n rk_r} \Bigg)\nonumber\\
\qquad{} = \frac{\ellipticqrfac{aq, aq/b c, aq/b d, aq/cd}{\sumN}}{\ellipticqrfac{aq/b , aq/c, aq/d, aq/b c d}{\sumN}} .\label{e-8p7-7b}
\end{gather}
\end{Theorem}
\begin{proof} This can be obtained by an analytic continuation argument, or directly from Theorem~\ref{e-fundamental-theorem}, by observing that when $a_r=q^{-N_r}$, then the indices $\k$ satisfy the additional condition $k_r\leq N_r$, for $r=1, \dots, n$. Then taking $N=\sumN$, and choosing $f_K$ appropriately, we obtain~\eqref{e-8p7-7b}.
\end{proof}

Next, we obtain another $A_n$ Bressoud matrix, from the $b\mapsto qa^2/bcd$ case of~\eqref{e-8p7-7b}.
\begin{Definition}[an $A_n$ elliptic Bressoud matrix] We define the matrix $\B^{(6)}$ with entries indexed by $(\k , \j)$ as
\begin{gather}
B^{(6)}_{\k \j}(a,b) := \frac{\ellipticqrfac{b}{\sumk+\sumj} \ellipticqrfac{b/a}{\sumk-\sumj}}
{ \ellipticqrfac{aq}{\sumk+\sumj} \sqprod n \ellipticqrfac{q^{1+j_r-j_s}\xover x }{k_r-j_r} }
 \smallprod n \frac{\ellipticqrfac{ex_r}{\sumk+j_r}} {\ellipticqrfac{ex_r}{k_r+\sumj}}. \label{e-B6}
\end{gather}
\end{Definition}

The inverse of this matrix is as follows.

\begin{Corollary}[inverse of $\B^{(6)}$] Let $\B^{(6)}= \big(B^{(6)}_{\k\j}(a,b)\big)$ be defined by \eqref{e-B6}. Then its inverse is given by
\begin{gather*}
\big(B^{(6)}(a,b)\big)_{\k \j}^{-1} = \frac{\elliptictheta{aq^{2\sumk}}}{\elliptictheta{a}}
\frac{\elliptictheta{bq^{2\sumj}}}{\elliptictheta{b}}
 \smallprod n \frac{\ellipticqrfac{ex_r}{\sumk+j_r}}
 {\ellipticqrfac{ex_r}{k_r+\sumj}}
\cdot \left( \frac{b}{a}\right)^{\sumk-\sumj}\notag\\
\hphantom{\big(B^{(6)}(a,b)\big)_{\k \j}^{-1} =}{} \times \frac{\ellipticqrfac{a}{\sumk+\sumj}
\ellipticqrfac{a/b}{\sumk-\sumj}} {\ellipticqrfac{bq}{\sumk+\sumj} \sqprod n \ellipticqrfac{q^{1+j_r-j_s}\xover x }{k_r-j_r} }.
\end{gather*}
\end{Corollary}
\begin{proof} We begin with the equivalent formulation of Theorem~\ref{th:e-8p7-7b} which is obtained by taking $b\mapsto qa^2/bcd$. Now take $d=aq/c$ to obtain a~Kronecker delta function on the product side. The rest of the calculation is analogous to the proof of Corollary~\ref{cor:e-B3-inverse}.
\end{proof}

We have not stated the WP Bailey pair with respect to $\B^{(6)}$ and the corresponding $\B^{(6)} \to\B^{(6)}$ WP Bailey lemma explicitly. While it is not difficult to find a corresponding WP Bailey lemma, it appears that in this case the WP Bailey lemma is not so useful to derive further identities. (We therefore have decided to omit it.) For this purpose it is actually better to apply Theorem~\ref{th:e-trivial-lemma1} instead. We illustrate this by writing down an elliptic Bailey $_{10}\phi_9$ transformation formula, which transforms an $n$ dimensional sum to a multiple of an $m$-dimensional sum. (This result cannot be obtained from the $\B^{(6)} \to\B^{(6)}$ WP Bailey lemma.)

\begin{Theorem}[an $A_n\to A_m$ elliptic Bailey $_{10}\phi_9$ transformation theorem] \label{th:e-10p9-7-n-m-tetra} Let $\lambda = qa^2/bcd$. Then
\begin{gather}
\sum\limits_{\substack{0\leq \sumk\leq N \\
k_1,k_2, \dots, k_n\geq 0} } \Bigg( \ellipticvandermonde{x}{k}{n}
\sqprod n \frac{\ellipticqrfac{e_s\xover{x}}{k_r} }{\ellipticqrfac{q\xover{x}}{k_r} } \notag \\
\qquad\quad{} \times
\frac{\elliptictheta{a q^{2\sumk}}}{\elliptictheta{a }}
\frac{\ellipticqrfac{a, b, c, d}{\sumk}}
{\ellipticqrfac{ a q^{N+1}, a q/b, aq/c, aq/d}{\sumk}}\nonumber\\
\qquad\quad{} \times
\frac{\ellipticqrfac{f_1\cdots f_m, \lambda aq^{1+N}/e_1\cdots e_n f_1\cdots f_m, q^{-N}}{\sumk}}
{\ellipticqrfac{a q/e_1\cdots e_n, aq/f_1\cdots f_m,
e_1\cdots e_n f_1\cdots f_m q^{-N}/\lambda}{\sumk}}\nonumber\\
\qquad\quad{} \times
\smallprod n \frac{\ellipticqrfac{g x_r/e_1\cdots e_n}{k_r}
\ellipticqrfac{g x_r}{\sumk}}{\ellipticqrfac{g x_r}{k_r} \ellipticqrfac{g x_r/e_r}{\sumk} } \cdot q^{\sum\limits_{r=1}^n rk_r} \Bigg)
\notag \\
\qquad{}= \frac{\ellipticqrfac{aq, aq/e_1\cdots e_n f_1\cdots f_m, \lambda q/e_1\cdots e_n, \lambda q/f_1\cdots f_m}{N}}{\ellipticqrfac{ a q/e_1\cdots e_n , aq/f_1\cdots f_m, \lambda q,
\lambda q/e_1\cdots e_n f_1\cdots f_m}{N}}\notag \\
\qquad\quad{} \times \sum\limits_{\substack{0\leq \sumj\leq N \\ j_1,j_2, \dots, j_m\geq 0} } \Bigg( \ellipticvandermonde{y}{j}{m}
\sqprod m \frac{\ellipticqrfac{f_s\xover{y}}{j_r} }{\ellipticqrfac{q\xover{y}}{j_r} } \notag \\
\qquad\quad{} \times \frac{\elliptictheta{\lambda q^{2\sumj}}}{\elliptictheta{\lambda }}
\frac{\ellipticqrfac{\lambda , \lambda b/a, \lambda c/a , \lambda d/a}{\sumj}}
{\ellipticqrfac{ \lambda q^{N+1}, a q/b, aq/c, aq/d}{\sumj}}\nonumber\\
\qquad\quad{} \times \frac{\ellipticqrfac{e_1\cdots e_n,
\lambda aq^{1+N}/e_1\cdots e_n f_1\cdots f_m, q^{-N}}{\sumj}}
{\ellipticqrfac{ \lambda q/e_1\cdots e_n, \lambda q/f_1\cdots f_m,
e_1\cdots e_n f_1\cdots f_m q^{-N}/a}{\sumj}}\nonumber\\
\qquad\quad{} \times \smallprod m \frac{\ellipticqrfac{h y_r/f_1\cdots f_m}{j_r}
\ellipticqrfac{h y_r}{\sumj}} {\ellipticqrfac{h y_r}{j_r} \ellipticqrfac{h y_r/f_r}{\sumj} }\cdot
 q^{\sum\limits_{r=1}^m rj_r} \Bigg). \label{e-10p9-7-n-m-tetra}
\end{gather}
\end{Theorem}
Note that the series on the right-hand side is of the same type as that on the
left-hand-side.
\begin{proof}
We begin with the left-hand side, and write it in the form
\begin{gather*}\sum_{K=0}^N f_K \sum_{\sumk = K} A_\k.\end{gather*}
We now use Theorem~\ref{th:e-trivial-lemma1} (with $a_k\mapsto e_k$ and $b\mapsto g$) to obtain a single sum. We transform this sum using the $n=1$ case of the elliptic Bailey transformation formula given in \eqref{e-10p9-6}. Once again, we use Theorem~\ref{th:e-trivial-lemma1}, this time with $n=m$, $x_k\mapsto y_k$, $b\mapsto h$, and $a_j\mapsto f_j$. In this manner, we obtain the right-hand side of \eqref{e-10p9-7-n-m-tetra}.
\end{proof}

An analytic continuation argument similar to the one used in the proof of Theorem~\ref{th:e-10p9-6-tetra} can be applied to Theorem~\ref{th:e-10p9-7-n-m-tetra} to obtain a transformation for a sum over an $n$-dimensional rectangle into a multiple of a sum over an $n$-simplex.
\begin{Theorem}[an $A_n$ elliptic Bailey $_{10}\phi_9$ transformation theorem] \label{th:e-10p9-7-n-m-rect-tetra}
Let $\lambda = qa^2/bcd$. Then
\begin{gather*}
\sum\limits_{\substack{0\leq k_r\leq N_r \\ r=1,2,\dots,n} }
\Bigg( \ellipticvandermonde{x}{k}{n} \sqprod n \frac{\ellipticqrfac{q^{-N_s}\xover{x}}{k_r} }{\ellipticqrfac{q\xover{x}}{k_r} } \notag \\
 \qquad\quad{} \times \frac{\elliptictheta{a q^{2\sumk}}}{\elliptictheta{a }}
\frac{\ellipticqrfac{a, b, c, d}{\sumk}} {\ellipticqrfac{aq^{\sumN+1}, aq/b, aq/c, aq/d}{\sumk}} \\
 \qquad\quad{} \times
\frac{\ellipticqrfac{e,f_1\cdots f_m, \lambda aq^{1+\sumN}/e f_1\cdots f_m}{\sumk}}
{\ellipticqrfac{a q/e, aq/f_1\cdots f_m,
e f_1\dots f_m q^{-\sumN}/\lambda}{\sumk}} \\
\qquad\quad{}\times \smallprod n \frac{\ellipticqrfac{g x_rq^\sumN}{k_r}
\ellipticqrfac{g x_r}{\sumk}} {\ellipticqrfac{g x_r}{k_r} \ellipticqrfac{g x_rq^{N_r}}{\sumk} }
\cdot q^{\sum\limits_{r=1}^n rk_r} \Bigg)\notag \\
\qquad{} =
\frac{\ellipticqrfac{aq, aq/e f_1\cdots f_m, \lambda q/e, \lambda q/f_1\cdots f_m}{\sumN}}{\ellipticqrfac{ a q/e, aq/f_1\cdots f_m, \lambda q,
\lambda q/e f_1\cdots f_m}{\sumN}}\notag \\
\qquad\quad{}\times
\sum\limits_{\substack{0\leq \sumj\leq \sumN \\
j_1,j_2, \dots, j_m\geq 0} }
\Bigg( \ellipticvandermonde{y}{j}{m}
\sqprod m \frac{\ellipticqrfac{f_s\xover{y}}{j_r} }{\ellipticqrfac{q\xover{y}}{j_r} } \notag \\
\qquad\quad{} \times
\frac{\elliptictheta{\lambda q^{2\sumj}}}{\elliptictheta{\lambda }}
\frac{\ellipticqrfac{\lambda , \lambda b/a, \lambda c/a , \lambda d/a}{\sumj}}
{\ellipticqrfac{ \lambda q^{\sumN+1}, a q/b, aq/c, aq/d}{\sumj}}\\
 \qquad\quad{} \times
\frac{\ellipticqrfac{e,
\lambda aq^{1+\sumN}/e f_1\cdots f_m, q^{-\sumN}}{\sumj}}
{\ellipticqrfac{ \lambda q/e, \lambda q/f_1\cdots f_m,
e f_1\cdots f_m q^{-\sumN}/a}{\sumj}}
\\
\qquad\quad{} \times
\smallprod m \frac{\ellipticqrfac{h y_r/f_1\cdots f_m}{j_r}
\ellipticqrfac{h y_r}{\sumj}}{\ellipticqrfac{h y_r}{j_r}\ellipticqrfac{h y_r/f_r}{\sumj} } \cdot q^{\sum\limits_{r=1}^m rj_r} \Bigg). 
\end{gather*}
\end{Theorem}

As another example, we obtain a result similar to Theorem~\ref{th:e-trivial-lemma1}. For this we start with the following $D_n$ elliptic Jackson sum
from Rosengren~\cite[Theorem~6.1]{HR2004} (rewritten, using elementary manipulations of $q,p$-shifted factorials) which is equivalent to~\eqref{e-8p7-4} upon replacing $n$ by $n+1$ and using an analytic continuation argument. In the $p=0$ case the corresponding $D_n$ Jackson sum was first given in \cite[Theorem~5.17]{MS1997}
\begin{gather*}
\sum\limits_{\substack{\sumk = K \\
k_1,k_2, \dots, k_n\geq 0} } \Bigg(\prod_{1\le r<s\le n}\frac{\theta\big(q^{k_r-k_s}x_r/x_s;p\big)}
{\theta(x_r/x_s;p)(x_rx_s;q,p)_{k_r+k_s}}\\
 \qquad\quad{} \times \smallprod n \frac{(b/x_r;q,p)_{\sumk-k_r} \prod\limits_{s=1}^{n-1}(x_ra_s,x_r/a_s;q,p)_{k_r}}
{(bx_r;q,p)_{k_r} \prod\limits_{s=1}^n(qx_r/x_s;q,p)_{k_r}}\cdot q^{\sum\limits_{r=1}^n(r-1)k_r}\Bigg)\\
 \qquad{} = \frac{\prod\limits_{s=1}^{n-1}(ba_s,b/a_s;q,p)_K} {(q;q,p)_K\prod\limits_{r=1}^n(bx_r;q,p)_K}.
\end{gather*}

We use this identity to obtain a $D_n$ version of Theorem~\ref{th:e-trivial-lemma1}.

\begin{Theorem}\label{th:e-trivial-lemma1-dn} Given the sequence $f_k$, $k=0,1,2, \dots$, and $N\geq 0$, we have
\begin{gather}
\sum\limits_{\substack{0\leq \sumk \leq N \\
k_1,k_2, \dots, k_n\geq 0} }
\Bigg(
\prod_{1\le r<s\le n}\frac{\theta\big(q^{k_r-k_s}x_r/x_s;p\big)}
{\theta(x_r/x_s;p)(x_rx_s;q,p)_{k_r+k_s}}\prod_{s=1}^{n-1}\frac{
\prod\limits_{r=1}^n(x_ra_s,x_r/a_s;q,p)_{k_r}}
{(ba_s,b/a_s;q,p)_{\sumk}}\nonumber\\
\qquad{} \times
\smallprod n\frac{(bx_r;q,p)_{\sumk}(b/x_r;q,p)_{\sumk-k_r}}
{(bx_r;q,p)_{k_r}\prod\limits_{s=1}^n(qx_r/x_s;q,p)_{k_r}}
\cdot q^{\sum\limits_{r=1}^n(r-1)k_r}
f_{\sumk}\Bigg) =
\sum_{K=0}^N
\frac{f_K}{\ellipticqrfac{q}{K}}.\label{e-trivial-lemma1-dn}
\end{gather}
\end{Theorem}

As in Theorem~\ref{th:e-trivial-lemma1} we can choose $f_K$ in Theorem~\ref{th:e-trivial-lemma1-dn} appropriately and use a single series identity to obtain a multiple series extension. In particular, we now obtain a new $D_n$ elliptic Jackson sum in this manner which involves two independent bases~$q$ and $\tilde{q}$, as well as two independent nomes~$p$ and~$\tilde{p}$.
\begin{Theorem}[a $D_n$ Jackson sum]\label{th:e-8p7-7a-dn} We have
\begin{gather}
\sum\limits_{\substack{0\leq \sumk \leq N \\ k_1,k_2, \dots, k_n\geq 0} }
\Bigg(
\prod_{1\le r<s\le n}\frac{\theta\big(q^{k_r-k_s}x_r/x_s;p\big)}
{\theta(x_r/x_s;p)(x_rx_s;q,p)_{k_r+k_s}}\prod_{s=1}^{n-1}\frac{
\prod\limits_{r=1}^n(x_ry_s,x_r/y_s;q,p)_{k_r}}
{(gy_s,g/y_s;q,p)_{\sumk}}\nonumber\\
\qquad\quad{} \times \smallprod n\frac{(gx_r;q,p)_{\sumk}(g/x_r;q,p)_{\sumk-k_r}} {(gx_r;q,p)_{k_r}\prod\limits_{s=1}^n(qx_r/x_s;q,p)_{k_r}}\cdot
q^{\sum\limits_{r=1}^n(r-1)k_r} (q;q,p)_\sumk\nonumber\\
\qquad\quad{} \times \frac{\theta\big(a\tilde{q}^{2\sumk};\tilde{p}\big)}{\theta\big(a;\tilde{p}\big)}
\frac{(a, b, c, d, a^2\tilde{q}^{1+N}/b c d,\tilde{q}^{-N};\tilde{q},\tilde{p})_\sumk}{(\tilde{q},a\tilde{q}/b, a\tilde{q}/c,
a\tilde{q}/d, b c d\tilde{q}^{-N}/a,a\tilde{q}^{N+1};\tilde{q},\tilde{p})_\sumk} \tilde{q}^{\sumk}\Bigg)\nonumber\\
 \qquad{} = \frac{(a\tilde{q}, a\tilde{q}/bc, a\tilde{q}/bd, a\tilde{q}/cd;\tilde{q},\tilde{p})_N}
{(a\tilde{q}/b, a\tilde{q}/c, a\tilde{q}/d, a\tilde{q}/b c d;\tilde{q},\tilde{p})_N}.\label{e-8p7-7a-dn}
\end{gather}
\end{Theorem}
\begin{proof}We replace $b$ by $g$, and $a_s$ by $y_s$, for $s=1,\dots,n-1$ in~\eqref{e-trivial-lemma1-dn}, and take
\begin{gather*}f_K=(q;q,p)_K
\frac{\theta\big(a\tilde{q}^{2K};\tilde{p}\big)}{\theta\big(a;\tilde{p}\big)}
\frac{( a, b, c, d, a^2\tilde{q}^{1+N}/b c d,\tilde{q}^{-N};
\tilde{q},\tilde{p})_K}{(\tilde{q},a\tilde{q}/b, a\tilde{q}/c,
a\tilde{q}/d, b c d\tilde{q}^{-N}/a,a\tilde{q}^{N+1};\tilde{q},\tilde{p})_K} \tilde{q}^{K}.
\end{gather*}
The rest of the proof is similar to the proof of Theorem~\ref{th:e-8p7-7a}.
\end{proof}

The last example shows that to obtain an elliptic extension of a terminating basic hypergeometric series identity is not just a matter of replacing $q$-shifted factorials by $q,p$-shifted factorials. While in the elliptic case, the factors depending on $g$ in~\eqref{e-8p7-7a-dn} are essential, they are not required in the basic case (where one could let $g\to0$).

This brings us to the end of our study.

\subsection*{Acknowledgements} The first author thanks Hjalmar Rosengren and the organizers of OPSF-S6 for the series of lectures~\cite{HR2016-lectures} on this subject. We thank Ole Warnaar for showing his notes~\cite{SOW-notes-2016} and much encouragement, Slava Spiridonov for some comments, and Zhizheng Zhang and Junli Huang for sending us their preprint~\cite{ZH-preprint}. We thank the anonymous referees for many insightful suggestions. We thank the Erwin Schr\"{o}dinger Institute for its workshop on {\em Elliptic hypergeometric functions in combinatorics, integrable systems and physics} held in Vienna in March 2017, where we benefited from discussions with other participants. Finally, research of both authors was supported by a grant of the Austrian Science Fund (FWF): F 50-N15.

\pdfbookmark[1]{References}{ref}
\LastPageEnding

\end{document}